\documentclass[11pt]{amsart}

\usepackage{a4wide, amsmath, amsfonts, amssymb, mathrsfs, amsthm, bbm}

\usepackage[shortlabels]{enumitem}
\usepackage[hyperfootnotes=false]{hyperref}
\numberwithin{equation}{section}

\newtheorem{theorem}{Theorem}[section]
\newtheorem{lemma}[theorem]{Lemma}
\newtheorem{corollary}[theorem]{Corollary}
\newtheorem{remark}[theorem]{Remark}
\newtheorem{proposition}[theorem]{Proposition}
\newtheorem{definition}[theorem]{Definition}

\newtheorem{assumption}[theorem]{Assumption}

\allowdisplaybreaks[2]

\renewcommand{\d}{\mathrm{d}}
\newcommand{\Id}{\,\mathrm{d}}
\renewcommand{\epsilon}{\varepsilon}
\newcommand{\R}{\mathbb{R}}
\newcommand{\N}{\mathbb{N}}
\renewcommand{\P}{\mathbb{P}}
\newcommand{\E}{\mathbb{E}}
\newcommand{\indicator}[1]{\mathbbm{1}_{#1}}
\newcommand{\norm}[1]{\left\lVert#1\right\rVert}
\newcommand{\qv}[1]{\langle #1 \rangle}
\newcommand{\testfunctions}[1]{C_c^{\infty}(#1)}

\newcommand{\cF}{\mathcal{F}}

\title[Mean-field limit and quantitative relative entropy estimates]{Quantitative relative entropy estimates on the whole space for convolution interaction forces}

\author[Nikolaev]{Paul Nikolaev}
\address{Paul Nikolaev, University of Mannheim, Germany}
\email{pnikolae@mail.uni-mannheim.de}

\author[Pr{\"o}mel]{David J. Pr{\"o}mel}
\address{David J. Pr{\"o}mel, University of Mannheim, Germany}
\email{proemel@uni-mannheim.de}

\date{\today}

\begin{document}

\begin{abstract}
  Quantitative estimates are derived, on the whole space, for the relative entropy between the joint law of random interacting particles and the tensorized law at the limiting systeme. The developed method combines the relative entropy method under the moderated interaction scaling introduced by Oeschl{\"a}ger, and the propagation of chaos in probability. The result includes the case that the interaction force does not need to be a potential field. Furthermore, if the interaction force is a potential field, with a convolutional structure, then the developed estimate also provides the modulated energy estimates. Moreover, we demonstrate propagation of chaos for large stochastic systems of interacting particles and discuss possible applications to various interacting particle systems, including the Coulomb interaction case.
\end{abstract}

\maketitle



\section{Introduction}

In this article we study \(N\)-particle systems \(\mathbf{X}^N= (X^1,\ldots, X^N)\) given by stochastic differential equations (SDEs) of the form
\begin{equation*}
  \Id X_t^{i} = -\frac{1}{N} \sum\limits_{j=1}^N k(X_t^{i} -X_t^{j}) \Id t + \sigma \Id B_t^{i}, \quad i=1,\ldots,N , \; \; \mathbf{X}_0^N \sim \overset{N}{\underset{i=1}{\otimes}} \rho_0,
\end{equation*}
starting from i.i.d. initial data \(\rho_0\). Such interacting systems arise naturally in various areas of science and engineering, including physics, chemistry, biology, ecology, and social sciences. For instance, they represent the behavior of ion channels, chemotaxis~\cite{KellerSegel1970,HillenPainter2009,Horstmann2004}, angiogenesis on the microscopic level and swarm movement~\cite{Topaz2006}, flocking~\cite{HaLiu2009}, opinion dynamics~\cite{LORENZ_2007_bounded_confidence_survey, noorazar2020recent}, cancer invasion~\cite{Domschke2014} on the macroscopic level. The macroscopic level is often described through the evolution of the density of particles/individuals~\(\rho\) known to satisfy an aggregation-diffusion equation, which in general is a non-local, non-linear partial differential equation (PDE). Transitioning from microscopic models to continuum descriptions, i.e. \(N\) approaches infinity, entails to explore the mean-field limit, see e.g.~\cite{snitzman_propagation_of_chaos,CarilloChoiPil2014,JabinWangZhenfu2016,JabinEmanuel2014}. It consists of demonstrating the convergence of the empirical measure \(\mu_t^N \) for all \(t \ge 0\), where \(\mu_t^N\) is defined as
\begin{equation*}
  \omega \mapsto \mu^N_t(\omega,A) := \frac{1}{N} \sum\limits_{i=1}^N \delta_{X^{i}_{t}(\omega)}  (A), \quad \quad A \in \mathcal{B}(\R). 
\end{equation*}
Various topologies are considered for the convergence, such as weak convergence, convergence in Wasserstein distance, convergence in terms of the Boltzmann entropy, and convergence in terms of the Fisher information. A comprehensive analysis can be found in~\cite{HaurayMischler2014}. In the present article we will focus on the convergence in entropy.

\smallskip
\noindent\textbf{Main contribution:} We present a novel method to derive propagation of chaos in entropy on the whole space for both non-conservative field and potential field  possessing a convolution structure. Inspired by Oelschl{\"a}ger~\cite{Oeschlager1987}, the presented method is based on the crucial observation that, under the convolution structure, the expectation of mollified $L^2$ norm and the modulated energy (also as a weighted \(L^2\)-norm) can be estimated using the dynamics of the underlying systems in conjunction with the propagation of chaos in probability, as demonstrated in~\cite{lazarovici2017mean,HuangLiuPickl2020,Fetecau_2019,Nikolaev2023}. The key contribution of the present work lies in the technique of combining propagation of chaos in probability~\cite{lazarovici2017mean,HuangHiuLius2019,LiuYang2019,Fetecau_2019,HaKimPickl2019,CarrilloChoiSalem2019,HuangLiuPickl2020,ChenLiPickl2020,Nikolaev2023} with the underlying entropy structure from~\cite{JabinWangZhenfu2016,Serfaty_2020,ChenHolzingerHuo2023} and the fluctuation estimates in~\cite{Oeschlager1987}. Consequently, we prove that convergence in probability for an interaction kernel, which is obtained by some type of mollification technique, implies convergence in relative entropy for an algebraic cut-off \(N^{-\beta}\). This demonstrates that convergence in probability is actually a quite strong convergence result.

We emphasize that the main quantitative estimate, Theorem~\ref{theorem: emp_measure_l2_estimate}, is presented in a general manner and can easily be extended to a multi-dimensional setting, allowing its application to a wide range of kernels. We refer to Remark~\ref{remark: general_kernels_l2main_thm} for more details and to Section~\ref{sec: application} for some interesting examples from the fields of chemotaxis and opinion dynamics. In particular, the method can be further applied in handling the attractive and repulsive Coulumb interaction potential in dimension \(d \ge 2\), which includes the Keller--Segel model. Finally, we derive an estimate on the supremum norm in time of the relative entropy between the law of the approximated particle system and the chaotic law of the approximated mean-field SDE system of rate greater than \(1/2\). Moreover, the approximation is of algebraic order, which is sharper than the logarithmic cut-off derived from the standard coupling method~\cite{snitzman_propagation_of_chaos,LiuYang2019}.

Theorem~\ref{theorem: emp_measure_l2_estimate} can be considered as an intermediate result on the approximated level. On the one hand, the remaining limit of the regularized mean-field equation to the mean-field equation reduces to a question regarding the convergence on the PDE level. On the other hand, the convergence of the regularized particle system to the particles system is a question about the stability of solutions to the stochastic differential equations. For bounded interaction kernels, we also provide both convergences in the \(L^1\)-norm. Consequently, we prove the \(L^1\)-convergence of the \(m\)-th marginal of the Liouville equation to the \(m\)-th chaotic law of the non-linear diffusion-aggregation equation. This final convergence result is only presented for bounded kernels since, in general, the existence for the linear Liouville equation~\eqref{eq: Liouville_equation} on \(\R^N\) is not given, see~\cite[Proposition~4.2]{BreschDidierJabinWangZhenfu2019} for the torus setting.

\smallskip
\noindent\textbf{Related literature:} The study of propagation of chaos for a globally Lipschitz continuous interaction force \(k\) has already a fairly long history, see e.g.~\cite{McKean1967,snitzman_propagation_of_chaos,HaurayMischler2014}. One of the first idea was to utilize the coupling method, i.e. comparing \((X_t^{i},t \ge 0)\) with their associated McKean--Vlasov SDEs.

Motivated by models, particularly from physics, with bounded measurable or even singular interaction force kernels, extensive efforts have been devoted to investigated propagation of chaos for particle systems with such kernels. Initially, approaches to treat such irregular kernels were often based on compactness methods in combination with the martingale problems associated to the McKean--Vlasov SDEs, see e.g. \cite{Oeschlager1984,Hirofumi1987,Gartner1988,FournierJourdain2017,godinho_Keller_Segel2015,LiLeiLiu2019,LiLeiLiu2019,olivera2020quantitative}. For general \(L^p\)-interaction force kernels~\(k\), the propagation of chaos was demonstrated for first and second order systems on the torus~\cite{bresch2023} and on the whole space~\(\R^d\) \cite{hao2022,han2023,Lacker2023}. Another approach, initiated by Lazarovici and Pickl for the {V}lasov--{P}oisson system~\cite{lazarovici2017mean}, allows to deduce propagation of chaos in probability. This method is well-suited for singular interaction kernels, even when the underlying systems may not be well-defined. 

For the moderate interacting system in deriving porous medium equation, Oeschl\"ager has actually a series of contributions many decades ago, for example in \cite{Oeschlager1984,Oeschlager1987}. Especially, for the fluctuation analysis, a smoothed $L^2$ estimate with  convergence rate $o(N^{-1/2})$ has been obtained. The convolution structure of the moderate interaction played important roles. In the estimates proposed in \cite{Oeschlager1987}, the repulsive moderate interaction  provides an essential quantity to absorb the rests from interacting effect. Recently, \cite{holzingerDr23} obtained Oeschl\"ager's $L^2$ estimate for the moderate interacting system with attractive potential, under the assumption that the convergence of probability for the moderating interacting particle system holds, which is still an open problem. Recently \cite{ChenHolzingerHuo2023}, derived also a connection between the relative entropy and the regularized $L^2$-norm in the moderate interaction framework by directly citing the estimate from \cite{Oeschlager1987}. The novelty of our work is that we do not follow the framework provided by \cite{Oeschlager1987}, but generate a direct estimation method in a general framework.

Another way to treat singular kernels such as the Coulomb potential \(x/|x|^s\) for \( s \ge 0 \) was investigated in the deterministic setting~\cite{Serfaty_2020,Nguyen2022} (\(\sigma = 0)\) as well as in the random setting~\cite{JabinWang2018,BreschDidierJabinWangZhenfu2019,bresch2020,RosenzweigSerfaty2023} (\(\sigma > 0 \)). The aforementioned references introduced the modulated free energy, which is a practical quantity suited for the Coulomb case. In particular, it metrize the weak convergence of the empirical measures~\cite{RosenzweigSerfaty2023}. A drawback of the modulated free energy approach in combination with the relative entropy is the torus domain as well as the requirement of entropy solutions on the particle level (microscopic level), see~\cite[Proposition~4.2]{BreschDidierJabinWangZhenfu2019}, which is non-trivial outside a setting on the torus. Furthermore, in order to apply the large deviation result in~\cite{JabinWang2018}, strict conditions are required on the initial data and the solution of the Fokker--Planck equation. Recently, Wang and Feng extended these results to the \(2D\)-viscous point vortex model on the whole space \(\R^2\). The idea is to show exponential decay of the solution~\cite[Theorem~4.4]{Feng2023} to be able to apply the large deviation result in~\cite{JabinWang2018}. Again strict restrictions on the initial conditions such as exponential decay of the initial data are necessary.

In the present article we manage to avoid the large deviation principle~\cite{JabinWang2018} and the strict conditions on the initial data by utilizing the convergence in probability, see~\eqref{eq: convergence_in_probability} below. We also can treat general forces such as rotational fields or magnetic fields in physics. We also manage to derive quantitative bounds on singular forces such as attractive Coulomb interaction kernels on the whole space, which to our knowledge require approximation techniques by the nature of their singularities on the level of the Lioville equation. The price we pay lies in the obtained convergence rate. While~\cite{JabinWang2018} can establish the convergence in the sense of the Boltzmann entropy on the level of the associated Fokker--Planck equations with an order of \(N^{-1}\), we achieve a rate of \(N^{-1/2-\vartheta}\) for some \(\vartheta >0\). Nevertheless, the convergence is faster than \(1/2\) and, therefore, we are optimistic that this result can be used as a stepping stone for Gaussian fluctuation. 

\smallskip
\noindent\textbf{Organization of the paper:} In Section~\ref{sec: preliminaries} we introduce the notation, the interacting particle systems and their associated diffusion-aggregation equations, give the necessary assumptions, and list the main results of this paper. We present the main ideas and the main estimate (Theorem~\ref{theorem: emp_measure_l2_estimate}) in Section~\ref{sec: relative_entropy_method}. In Section~\ref{sec: PDE convergence}, we demonstrate the propagation of chaos in the case of bounded interaction forces for the non-regularized systems by establishing the convergence of the approximated PDEs to the non-approximated counterparts. In Section~\ref{sec: application}, we showcase the applicability of the developed method by discussing, e.g., the regularized, singular Keller--Segel models and bounded confidence models.

\smallskip
\noindent\textbf{Acknowledgments:} P. Nikolaev and D. J. Pr\"omel would like to deeply thank L. Chen for fruitful discussions and suggestions leading to a significant improvement of the present work.

\section{Problem setting, preliminaries and main results}\label{sec: preliminaries}

In this section we introduce the basic setting, the interacting particle systems, their associated partial differential equations, some preliminary results, as well we the main results of this article.

\subsection{Particle systems}\label{subsec: particle system}

In this subsection we introduce the probabilistic setting, in particular, for the \(N\)-particle system and the associated McKean--Vlasov equation. To that end, let \((\Omega, \mathcal{F}, ( \mathcal{F}_t)_{t \ge 0 } , \P)\) be a complete probability space with right-continuous filtration \((\mathcal{F}_t)_{t \ge 0 } \) and \((B_t^{i}, t\ge 0)\), \( i=1, \ldots, N\), be independent one-dimensional Brownian motions with respect to \((\cF_t, t \ge 0)\). In the following, we use the notation $X\sim \rho$ to represent that $\rho$ is the law of random variable $X$.

The \(N\)-particle system \(\mathbf{X}_t^N:= (X_t^{1}, \ldots,X_t^{N})\) is given by
\begin{equation}\label{eq: particle_system}
  \Id X_t^{i} = -\frac{1}{N} \sum\limits_{j=1}^N k(X_t^{i} -X_t^{j}) \Id t + \sigma \Id B_t^{i}, \quad i=1,\ldots,N , \; \; \mathbf{X}_0^N \sim \overset{N}{\underset{i=1}{\otimes}} \rho_0,
\end{equation}
where \(\sigma >0\) is the diffusion parameter and \( \mathbf{X}_0^N\) is independent of the Brownian motions \((B_t^{i}, t\ge 0)\), \( i=1, \ldots, N\). The particle system~\eqref{eq: particle_system} induces in the limiting case \(N\to \infty\) the following i.i.d. sequence \(\mathbf{Y}_t^N := (Y_t^{1}, \ldots,Y_t^{N})\) of mean-field particles
\begin{equation}\label{eq: mean_field_trajectories}
  \Id Y_t^{i} = -  (k* \rho_t)(Y_t^{i}) \Id t + \sigma \Id B_t^{i}, \quad i=1,\ldots,N , \; \; \mathbf{Y}_0^N=\mathbf{X}_0^N,
\end{equation}
where \(\rho_t:= \rho(t,\cdot)\) denotes the probability density of the i.i.d. random variable \(Y_t^{i}\).

To introduce the regularized versions of \eqref{eq: particle_system} and \eqref{eq: mean_field_trajectories}, we take the smooth approximation \((k^\epsilon, \epsilon > 0)\) of \(k\) and replace the drift term with its approximation. Hence, the regularized microscopic \(N\)-particle system \(\mathbf{X}_t^{N,\epsilon } := (X_t^{1,\epsilon}, \ldots,X_t^{N,\epsilon})\) is given by
\begin{equation}\label{eq: regularized_particle_system}
  \d X_t^{i,\epsilon} = -\frac{1}{N} \sum\limits_{j=1}^N k^\epsilon(X_t^{i,\epsilon} -X_t^{j,\epsilon}) \Id t + \sigma \Id B_t^{i}, \quad i=1,\ldots,N , \; \; \mathbf{X}_0^{N,\epsilon} \sim \overset{N}{\underset{i=1}{\otimes}} \rho_0,
\end{equation}
and the regularized mean-field trajectories \(\mathbf{Y}_t^{N,\epsilon} := (Y_t^{1,\epsilon}, \ldots,Y_t^{N,\epsilon})\) by
\begin{equation}\label{eq: regularized_mean_field_trajectories}
  \d Y_t^{i,\epsilon} = -  (k^\epsilon* \rho_t^\epsilon)(Y_t^{i,\epsilon}) \Id t + \sigma \Id B_t^{i}, \quad i=1,\ldots,N , \; \; \mathbf{Y}_0^{N,\epsilon}=\mathbf{X}_0^{N,\epsilon},
\end{equation}
where \(\rho_t^\epsilon:=\rho^\epsilon(t,\cdot)\) denotes the probability density of the i.i.d. random variable \(Y_t^{i,\epsilon}\).

Finally, let the empirical measure of the regularized interaction system be given by 
\begin{equation}\label{eq: def_emprirical_measure_reg}
  \mu_t^{N,\epsilon}(\omega) := \frac{1}{N} \sum\limits_{i=1}^N \delta_{X_t^{i,\epsilon}(\omega) } \in \mathcal{P}(\R),
\end{equation}
where \(\delta\) is the Dirac measure. 

\subsection{Associated PDEs}

It{\^o}'s formula implies that the associated probability densities of the particle systems, introduced in Subsection~\ref{subsec: particle system}, satisfy partial differential equations (PDEs). Indeed, the interacting particle system~\eqref{eq: particle_system} induces the following Liouville equation on \(\R^N\),
\begin{align}\label{eq: Liouville_equation}
  \begin{cases}
  \partial_t \rho_t^{N}(\mathrm{X}^N)  &=  \frac{\sigma^2}{2}\sum\limits_{i=1}^N  \partial_{x_ix_i} \rho_t^{N} (\mathrm{X}^N )+ \sum\limits_{i=1}^N
  \partial_{x_i} \Bigg ( \rho^{N}_t(\mathrm{X}^N) \frac{1}{N} \sum\limits_{j=1}^N  k (x_{i} -x_{j} ) \Bigg)   \\
  \rho_0^{N}(\mathrm{X}^N) &= \prod\limits_{i=1}^N \rho_0(x_i)
\end{cases}
\end{align}
for \(\mathrm{X}^N = (x_1, \ldots, x_N) \in \R^N \), the system~\eqref{eq: mean_field_trajectories} induces the non-linear aggregation-diffusion equation
\begin{align}\label{eq: aggregation_diffusion_pde}
  \begin{cases}
  \partial_t \rho_t = \frac{\sigma^2}{2} \partial_{xx}\rho + \partial_x (\rho_t k*\rho_t ) \quad & \forall (t,x) \in  [0,T)  \times \R \\
  \; \; \, \rho(x,0) = \rho_0 &\forall x \in \R
  \end{cases},
\end{align}
the regularized particle system~\eqref{eq: regularized_particle_system} the Liouville equation
\begin{align}\label{eq: regularized_Liouville_equation}
  \begin{cases}
  \partial_t \rho_t^{N,\epsilon} (\mathrm{X}^N)  =  \frac{\sigma^2}{2}\sum\limits_{i=1}^N  \partial_{x_ix_i} \rho_t^{N,\epsilon} (\mathrm{X}^N) + \sum\limits_{i=1}^N
  \partial_{x_i} \Bigg ( \rho^{N,\epsilon}_t (\mathrm{X}^N) \frac{1}{N} \sum\limits_{j=1}^N  k^{\epsilon} (x_{i} -x_{j} ) \Bigg)   \\
  \; \; \,\,\rho_0^{N,\epsilon}(\mathrm{X}^N) = \prod\limits_{i=1}^N \rho_0(x_i)
  \end{cases}
\end{align}
and the regularized system~~\eqref{eq: regularized_mean_field_trajectories} the aggregation-diffusion equation
\begin{align}\label{eq: regularized_aggregation_diffusion_pde}
\begin{cases}
\partial_t \rho_t^\epsilon = \frac{\sigma^2}{2} \partial_{xx}\rho_t^\epsilon + \partial_x (\rho_t^\epsilon k^\epsilon*\rho_t^\epsilon ) \quad & \forall (t,x) \in   [0,T) \times \R  \\
\; \; \, \rho^\epsilon(x,0) = \rho_0 &\forall x \in \R
\end{cases} . 
\end{align}

Note that we use \(\rho_t\) and \(\rho_t^\epsilon \) for the solutions of the PDEs~\eqref{eq: aggregation_diffusion_pde} and \eqref{eq: regularized_aggregation_diffusion_pde} as well as for the probability densities of the particle systems \eqref{eq: mean_field_trajectories} and \eqref{eq: regularized_mean_field_trajectories}, respectively, since these objects coincide by the superposition principle, see \cite{barburoeckner2020}, in combination with existence results of densities for considered SDEs, see~\cite{romito2018}.

Furthermore, we need to define the marginal of the system of rank \(1\leq m\leq N\), 
\begin{equation}\label{eq: def_m_martingale}
  \rho_t^{N,m} = \int_{\R^{N-m}} \rho_t^N(x_1, \ldots, x_N) \Id x_{m+1} \ldots \Id x_N.
\end{equation}
We remark that the \(m\)-th martingale solves the following Liouville equation 
\begin{align}\label{eq: first_martingale_liuville_equation}
  \begin{split}
  \partial_t \rho_t^{N,m}
  =&\;  \frac{\sigma^2}{2} \sum\limits_{i=1}^m  \int_{\R^{N-m}}  \partial_{x_ix_i} \rho_t^N(x_1,\ldots,x_N) \\
  &+ \partial_{x_i} \Bigg ( \rho^N_t(x_1,\ldots,x_N) \frac{1}{N} \sum\limits_{j =1 }^N  k (x_{i} -x_{j} ) \Bigg) \Id x_{m+1} \ldots \Id x_N  .
  \end{split}
\end{align}
Similar to~\eqref{eq: def_m_martingale} we denote by \(\rho_t^{N,m,\epsilon}\) the m-th marginal of the approximated Lioville equation, i.e. 
\begin{equation*}
  \rho_t^{N,m,\epsilon} := \int_{\R^{N-m}} \rho_t^{N,\epsilon}(x_1, \ldots, x_N) \Id x_{m+1} \ldots \Id x_N,
\end{equation*}
which solves~\eqref{eq: first_martingale_liuville_equation} with \(k^\epsilon\) instead of \(k\). 
Additionally, we define the chaotic law
\begin{equation*}
  \rho_t^{\otimes m, \epsilon}(x_1,\ldots,x_m):= \prod\limits_{i=1}^m \rho_t^\epsilon(x_i),
\end{equation*}
which solves the following equation
\begin{align*}
  \begin{split}
  &\;  \partial_t \rho_t^{\otimes m, \epsilon}(x_1,\ldots,x_m) \\
  &\quad = \;
  \frac{\sigma^2}{2} \sum\limits_{i=1}^m \partial_{x_ix_i} \rho_t^{\otimes m, \epsilon}(x_1,\ldots,x_m)
  + \sum\limits_{i=1}^m \partial_{x_i} \Big( (k*\rho_t^\epsilon)(x_i) \rho_t^{\otimes m, \epsilon}(x_1,\ldots,x_m) \Big)
  \end{split}
\end{align*}
with initial data \(\rho_0^{\otimes m, \epsilon} = \rho_0^{\otimes m}\). 

\subsection{Preliminary results}

In this subsection, we gather essential definitions, the requisite function spaces and preliminary results for the well-posedness of the above mentioned SDEs and PDEs.

Throughout the entire paper, we use the generic constant \(C\) for inequalities, which may change from line to line. The constants \(\alpha,\beta_\alpha,\beta\) are always fix and will be given by our Assumptions~\ref{ass: convergence_in_probability},~\ref{ass: law_of_large_numbers}.

For \( 1 \le p \le \infty\) we denote by \(L^p(\R)\) the space of measurable functions whose \(p\)-th power is Lebesgue integrable (with the usual modification for \(p = \infty\)) equipped with the norm \(\norm{\cdot}_{L^p(\R)}\), by \(L^1(\R, |x|^2\Id x)\) the space of all measurable functions \(f\) such that \(\int_{\R}|f(x)||x|^2\Id x <\infty\), by \(\testfunctions{\R}\) the space of all infinitely differentiable functions with compact support on \(\R\), and by \(\mathcal{S}(\R)\) the space of all Schwartz functions, see \cite[Chapter~6]{YoshidaKosaku1995FA} for more details.

Let \((Z,\norm{\cdot}_Z)\) be a Banach space. We denote by \(L^p([0,T];Z)\) the space of all strongly measurable functions \(u\colon [0,T] \to Z\) such that
\begin{equation*}
  \norm{u}_{L^p([0,T];Z)}:= \left\{\begin{array}{ll}
  \Big( \displaystyle\int_0^T \norm{u(t)}_Z^p \Id t \Big)^{\frac{1}{p}} < \infty
  , & \text{for } 1 \le p < \infty,\\[5mm]
  \displaystyle \operatorname*{ess\,sup}_{ t\in [0,T]} \norm{u(t)}_Z < \infty, & \text{for }p=\infty.
  \end{array}\right.
\end{equation*}
The Banach space \(C([0,T];Z)\) consists of all continuous functions \(u \colon [0,T] \to Z\), equipped with the norm
\begin{equation*}
  \max\limits_{t \in [0,T]} \norm{u(t)}_Z < \infty.
\end{equation*}

For a smooth function \(u \colon [0,T] \times \R^d \mapsto \R \) and a multiindex \(\kappa\) with length \(|\kappa| := \sum_i \kappa_i\), we denote the derivative with respect to  \(x^\kappa=x_1^{\kappa_1} \cdots x_d^{\kappa_d}\) by \(\partial^{\kappa} u(t,x) := \prod_i \big(\frac{\partial}{\partial_{x_i}} \big)^{\kappa_i} u(t,x) \), where we write \(\partial_{x_i} u\) or \(u_{x_i}(t,x)\) for \(\frac{\partial}{\partial x_{i}} u(t,x)\). The derivative with respect to time we denote by \(\partial_t u(t,x)\). For \(u \in \mathcal{S}'(\R^d)\) we define the Fourier transform \(\mathcal{F}[u]\) and inverse Fourier transform \( \mathcal{F}^{-1}[u]\) by
\begin{equation*}
  \mathcal{F}[u](\xi) :=  \int_{\R^d} e^{- 2\pi  i  \eta \cdot x } u(x) \Id x
  \quad \text{and}\quad
  \mathcal{F}^{-1}[u](\xi):=\int_{\R^d} e^{ 2\pi  i \eta \cdot x } u(x) \Id x .
\end{equation*}
We denote the Bessel potential for each \(s \in \R\) and \(u \in \mathcal{S}'(\R^d)\) by
\begin{equation*}
  (1-\Delta)^{s/2}u := \mathcal{F}^{-1}[(1+4\pi^2|\xi|^2)^{s/2} \mathcal{F}[u]]
\end{equation*}
and define the Bessel potential space \(\mathnormal{H}_p^s \) for \(p \in (1,\infty)\) and \(s \in \R\) by
\begin{equation*}
  \mathnormal{H}_p^s:= \{ u \in \mathcal{S}'(\R^d) \; : \;  (1-\Delta)^{s/2} u \in L^p(\R^d) \}, \mbox{ with norm } \norm{u}_{\mathnormal{H}^{s}_p(\R^d)}:= \norm{(1-\Delta)^{s/2}u}_{L^p(\R^d)}
\end{equation*}

Applying \cite[Theorem~2.5.6]{TriebelHans1983Tofs} we can characterize the above Bessel potential spaces \(\mathnormal{H}_p^m\) for \(1 <p < \infty \) and \( m \in \N\) as Sobolev spaces
\begin{equation*}
  W^{m,p} (\R^d) : = \bigg \{ u \in L^p(\R^d) \; : \; \norm{u}_{W^{m,p}(\R^d)}:= \sum\limits_{ \kappa \in \mathcal{A} ,\  |\alpha|\le m } \norm{\partial^\kappa f }_{L^p(\R^d)}  < \infty \bigg \},
\end{equation*}
where \(\partial^\kappa u \) is to be understood as weak derivatives \cite{AdamsRobertA2003Ss} and \(\mathcal{A} \) is the set of all multi-indices. Moreover, we will use the following abbreviation \(H^s(\R^d):= H^s_2(\R^d)\).

\vskip3mm
For the partial differential equations~\eqref{eq: Liouville_equation},~\eqref{eq: aggregation_diffusion_pde},~\eqref{eq: regularized_Liouville_equation} and \eqref{eq: regularized_aggregation_diffusion_pde} we rely on the concept of weak solutions, which we recall in the next definition.

\begin{definition}[Weak solutions]\label{def: weak_solution_Lioville_equation}
  Fix a time \(T>0\). A function \(\rho^{N,\epsilon} \in L^2([0,T];H^1(\R^N)) \cap L^\infty([0,T];L^2(\R^N))\) with \(\partial_t \rho^{N,\epsilon} \in L^2([0,T];H^{-1}(\R^N))\) is a weak solution of \eqref{eq: regularized_Liouville_equation} if for every \(\eta \in L^2([0,T];H^1(\R^N))\),
  \begin{align}\label{eq: regularized_Lioville_weak_solution}
    & \; \int\limits_0^T \qv{\partial_t \rho^{N,\epsilon}_t, \eta_t}_{H^{-1}(\R^N), H^1(\R^N)} \Id t    \\
    = & \;  - \sum\limits_{i=1}^N  \int\limits_0^T \int_{\R^N} \left ( \frac{\sigma^2}{2} \partial_{x_i} \rho_t^{N,\epsilon} (\mathrm{X}^N ) +  \rho^{N,\epsilon}_t (\mathrm{X}^N) \frac{1}{N} \sum\limits_{j=1}^N k^\epsilon(x_i-x_j) \rho^{N,\epsilon}_t(\mathrm{X}^N) \right) \partial_{x_i} \eta_t(\mathrm{X}^N) \Id \mathrm{X}^N  \Id t\nonumber
  \end{align}
  and \(\rho^{N,\epsilon}(0,\mathrm{X}^N ) = \prod\limits_{i=1}^N \rho_0(x_i)\). We note that the regularity \(\rho^{N,\epsilon} \in L^2([0,T];H^1(\R^N))\) and \(\partial_t \rho^{N,\epsilon} \in L^2([0,T];H^{-1}(\R^N))\) imply \(\rho^{N,\epsilon} \in C([0,T];L^2(\R^N))\) (see for example \cite[Chapter~5.9]{EvansLawrenceC2015Pde}). Similarly, we say that \(\rho^N \in L^2([0,T];H^1(\R^N)) \cap L^\infty([0,T];L^2(\R^N))\) with \(\partial_t \rho^N \in L^2([0,T];H^{-1}(\R^N))\) is a weak solution of \eqref{eq: Liouville_equation} if \eqref{eq: regularized_Lioville_weak_solution} hold with the interaction force kernel~\(k\) instead of its approximation~\(k^\epsilon\).
\end{definition}

\begin{definition}[Weak solutions]\label{def: weak_solution}
  Fix \(\epsilon > 0\) and \(T>0\). We say \(\rho^\epsilon \in L^2([0,T];H^1(\R)) \cap L^\infty([0,T];L^2(\R))\) with \(\partial_t \rho^\epsilon \in L^2([0,T];H^{-1}(\R))\) is a weak solution of \eqref{eq: regularized_aggregation_diffusion_pde} if, for every \(\eta \in L^2([0,T];H^1(\R))\),
  \begin{equation}\label{eq: regularized_weak_solution}
    \int\limits_0^T \qv{\partial_t \rho^\epsilon_t, \eta_t}_{H^{-1}(\R), H^1(\R)} \Id t = - \int\limits_0^T \int_\R \left ( \frac{\sigma^2}{2} \partial_x\rho_t^\epsilon + (k^\epsilon*\rho_t^\epsilon) \rho_t^\epsilon \right) \partial_x \eta_t \Id x \Id t
  \end{equation}
  and \(\rho^\epsilon(0,\cdot) = \rho_0\). Note that \(\rho^\epsilon \in L^2([0,T];H^1(\R))\) with \(\partial_t \rho^\epsilon \in L^2([0,T];H^{-1}(\R))\) implies \(\rho^\epsilon \in C([0,T];L^2(\R))\), see \cite[Chapter~5.9]{EvansLawrenceC2015Pde}. Similarly, we say that \(\rho \in L^2([0,T];H^1(\R)) \cap L^\infty([0,T];L^2(\R))\) with \(\partial_t \rho \in L^2([0,T];H^{-1}(\R))\) is a weak solution of \eqref{eq: aggregation_diffusion_pde} if \eqref{eq: regularized_weak_solution} holds with the interaction force kernel \(k\) instead of its approximation \(k^\epsilon\).
\end{definition}

By the regularity of the solution in Definition~\ref{def: weak_solution} we can actually weaken the assumption on \(\eta\) in equations~\eqref{eq: regularized_Lioville_weak_solution} and~\eqref{eq: regularized_weak_solution} to \(\eta \in C([0,T];C_c^\infty(\R))\).

\begin{remark}
  The divergence structure of the PDEs \eqref{eq: aggregation_diffusion_pde} and \eqref{eq: regularized_aggregation_diffusion_pde}, respectively, implies mass conservation/the normalisation condition
  \begin{equation*}
    1 = \int_\R \rho_t(x) \Id x = \int_\R \rho^\epsilon_t(x) \Id x
  \end{equation*}
  for all \(0 \le t \le T\) under Assumption~\ref{ass: initial condition}. This is an immediate consequence by plugging in a cut-off sequence, see~\cite[Lemma~8.4]{BrezisHaim2011FaSs}, which converges to the constant function \(1\) as a test function in~\eqref{eq: regularized_weak_solution}.
\end{remark}

Throughout the entire paper we make the following assumptions on the initial condition~\(\rho_0\) of the interacting particle system and the interaction force kernel~$k$.

\begin{assumption} \label{ass: initial condition}
  The initial condition \(\rho_0 \colon \R \to \R \) fulfills
  \begin{equation}\label{eq: initial_condition_integrability}
	\rho_0 \in L^1(\R) \cap L^\infty(\R) \cap L^1(\R, |x|^2 \Id x ) ,\quad
	\rho_0 \ge 0 ,\quad\text{and}\quad
	\int_{\R} \rho_0(x) \Id x = 1 .
  \end{equation}
\end{assumption}

We recall some general facts, which will be used throughout the article. First, we notice that we have a solution \((\rho_t^{N,\epsilon}, t \ge 0)\) of the regularized PDE~\eqref{eq: regularized_Liouville_equation} in the sense of Definition~\ref{def: weak_solution_Lioville_equation}, which follows from the regularity of \(k^\epsilon\). We also have a solution \((\rho_t^{N}, t \ge 0)\) in the sense of Definition~\ref{def: weak_solution_Lioville_equation} in the case \(k\in L^\infty(\R)\) and the equation is linear. By standard SDE theory we also obtain strong solutions \((\mathbf{X}_t^{N,\epsilon},t \ge 0)\), \((\mathbf{Y}_t^{N,\epsilon},t \ge 0)\) to the regularized SDEs~\eqref{eq: regularized_particle_system},~\eqref{eq: regularized_mean_field_trajectories}. For the well-posedness of the particle system~\eqref{eq: particle_system} and McKean--Vlasov SDE~\eqref{eq: mean_field_trajectories} we refer to~\cite[Theorem~3.7]{hao2022} and~\cite[Theorem~4.10]{hao2022}, respectively. Additionally, \cite[Section~3]{Nikolaev2023} guarantees the well-posedness of PDEs~\eqref{eq: regularized_aggregation_diffusion_pde},~\eqref{eq: aggregation_diffusion_pde}, which are bounded in time and space uniformly in \(\epsilon\). Consequently, our framework is well-defined and, in particular, the empirical measure \(\mu_t^{N, \epsilon}\) given by~\eqref{eq: def_emprirical_measure_reg} is well-defined. 

\vskip3mm

The analysis of the entropy relies on the convergence of the particle system~\eqref{eq: regularized_particle_system} to the particle system~\eqref{eq: regularized_mean_field_trajectories} in probability.  Hence, we introduce the following convergence in probability assumption.

\begin{assumption}\label{ass: convergence_in_probability}
  Let \((\mathbf{X}_t^{N,\epsilon}, t \ge 0)\), \((\mathbf{Y}_t^{N,\epsilon}, t \ge 0)\) be given by~\eqref{eq: regularized_particle_system},~\eqref{eq: regularized_mean_field_trajectories}. Then for \(\alpha \in (0,1/2)\), \(\beta_\alpha \in (0, \alpha)\), \(\beta \le \beta_\alpha\), \(\epsilon \sim N^{-\beta}\) there exists an \(N_0 \in \N\) such that for all \(N \ge N_0\), \(\gamma > 0\) we have
  \begin{equation}\label{eq: convergence_in_probability}
    \P\left(\sup\limits_{0 \le t \le T} \sup\limits_{1\le i \le N} |X_t^{i,\epsilon}-Y_t^{i,\epsilon}| \ge N^{-\alpha}\right) \le C(\gamma) N^{-\gamma},
  \end{equation}
  where \(C(\gamma)\) depends on the initial density \(\rho_0\), the final time \(T>0\), \(\alpha\) and \(\gamma\).
\end{assumption}

\vskip3mm
This assumptions is satisfied by a variety of models~\cite{lazarovici2017mean,HuangHiuLius2019,LiuYang2019,Fetecau_2019,HaKimPickl2019,CarrilloChoiSalem2019,HuangLiuPickl2020,ChenLiPickl2020}. In particular for bounded \(k\) or even singular kernels this assumption is fulfilled, see~\cite{Nikolaev2023}.

\vskip3mm
Furthermore, we need the following law of large numbers result.  

\begin{assumption}\label{ass: law_of_large_numbers}
  Let \((\mathbf{Y}_t^{N,\epsilon}, t \ge 0)\) and \(\rho^\epsilon_t\) be given by ~\eqref{eq: regularized_mean_field_trajectories}. Assume further that \(0< \alpha, \delta\), \(0<\alpha+\delta<1/2\), \(\epsilon \sim N^{-\beta}\) with \(\beta_\alpha \in (0, \alpha)\), \(\beta \le \beta_\alpha\) and define for \(0 \le t \le T\) the following sets
  \begin{equation*}
    B_t^\alpha := \bigg\{ \max\limits_{1 \le i \le N} \bigg| \sum\limits_{j=1}^N  k^\epsilon(Y_t^{i,\epsilon}-Y_t^{j,\epsilon})-(k*\rho_t^\epsilon)(Y_t^{i,\epsilon}) \bigg|\le N^{-(\delta+\alpha)}  \bigg\}.
  \end{equation*}
  Then, for each \(\gamma > 0\) there exists a \(C(\gamma) > 0\) such that
  \begin{equation}\label{eq: lln_small_b_set}
    \P( (B_t^\alpha)^{\mathrm{c}}) 
    \le C(\gamma) N^{-\gamma }
  \end{equation}
  for every \(0 \le t \le T\), where the constant \(C(\gamma) \) is independent of \(t \in [0,T]\).
\end{assumption}

We refer again to~\cite{lazarovici2017mean,HuangHiuLius2019,LiuYang2019,Fetecau_2019,HaKimPickl2019,CarrilloChoiSalem2019,HuangLiuPickl2020,ChenLiPickl2020}. In particular, the assumption is satisfied for bounded forces \(k\), which satisfy a local Lipschitz bound~\cite{Nikolaev2023}.

\subsection{Main results:}

Let \(J^\epsilon(x) : = \frac{1}{\epsilon} J\big( \frac{x}{\epsilon} \big )\) with \(J \colon \R \to \R\) a given mollification kernel and let \(\zeta\) be a cut-off function, which satisfies \(|\zeta| \le 1 \), \(\zeta = 1 \) on \(B(0,1)\) and \(\zeta = 0\) on \(B(0,2)^{\mathrm{c}}\), \(\zeta^\epsilon(x) = \zeta(\epsilon x)\). 
We need the following assumptions on the mollified version of interaction force separately to state the main result of this paper.
 
\begin{definition}\label{def: admissible}
  We say \(W^\epsilon, V^\epsilon\) are admissible approximations, if \(W^\epsilon \in L^2(\R)\) and \(V^\epsilon \in H^2(\R)\) with
  \begin{equation}\label{eq: main_thm_v_w_norm}
    \norm{W^\epsilon}_{L^2(\R)} \le  C\epsilon^{-a_W} , \quad \norm{V^\epsilon}_{H^2(\R)} \le C\epsilon^{-a_V}
  \end{equation}
  for some \(C>0\) and \(a_W,a_V >0\). We say admissible approximations \(W^\epsilon, V^\epsilon\) are strongly admissible approximations, if the above inequality holds for \(\norm{W^\epsilon}_{H^2(\R)}\) instead of \(\norm{W^\epsilon}_{L^2(\R)} \).
\end{definition} 

In general we will consider two type of forces. First, \(k^\epsilon= W^\epsilon*V^\epsilon\) and second \(k^\epsilon= (W^\epsilon*V^\epsilon)_x\). The potential field structure of the latter one will be required for the definition of the modulated energy (see Section~\ref{sec: relative_entropy_method}).
This assumption on \(k\) include many different forces, where no potential field is needed.

\begin{remark}
  Some typical examples for the above structure are as follows:
  \begin{enumerate}
    \item The interaction force kernel \(k \in L^2(\R)\). Then \(W^\epsilon = k\) and \(V^\epsilon = J^\epsilon\) is just the standard mollified version of \(k\).
    \item If \(k \in L^p(\R) \) for \(p< \infty\), we can choose \(W^\epsilon= k*J^\epsilon\) and \(V^\epsilon = J^\epsilon\), which is also just a mollification of \(k\).
    \item If \(k\in L^\infty(\R)\) we may choose \(W^\epsilon= \zeta^\epsilon(k*J^\epsilon)\) and \(V^\epsilon = J^\epsilon\), where \(\zeta^\epsilon\) is defined as a cut-off function to guarantee integrability of the mollification \(k*J^\epsilon\).
  \end{enumerate}
\end{remark}

The first main result of this paper is the propagation of chaos on the mollified level with $\epsilon=N^{-\beta}$:

\begin{theorem}\label{maintheorem}
  Let \(\rho^{N,\epsilon} \) and \(\rho^{\epsilon}\) be the non-negative solutions of~\eqref{eq: regularized_Liouville_equation} and  of~\eqref{eq: regularized_aggregation_diffusion_pde} respectively. Assume that the convergence in probability, Assumption~\ref{ass: convergence_in_probability}, and the law of large numbers, Assumption~\ref{ass: law_of_large_numbers} hold for \(\alpha \in (\frac{1}{4},\frac{1}{2})\). Let \(k^\epsilon=W^\epsilon*V^\epsilon\) and \(W^\epsilon \in L^2(\R), V^\epsilon \in H^2(\R)\) be admissible in the sense of Definition~\ref{def: admissible} with rate \(a_W,a_V\). Then there exists a $\beta_1\in (0,\beta_\alpha)$ such that $\forall \beta\in(0,\beta_1)$, the following propagation of chaos result holds for $\epsilon=N^{-\beta}$ between \eqref{eq: regularized_Liouville_equation} and  of~\eqref{eq: regularized_aggregation_diffusion_pde}.
  \begin{equation} \label{eq: main_thm_entropy_estimate}
	\norm{\rho_t^{N,2,\epsilon} - \rho_t^{ \epsilon} \otimes \rho_t^{ \epsilon}}_{L^1(\R^2)}^2
	\le 2{\mathcal{H}}_2(\rho_t^{N,1,\epsilon} \vert \rho_t^{ \epsilon})
	\le 4{\mathcal{H}}_N(\rho_t^{N,\epsilon} \vert \rho_t^{\otimes N , \epsilon})
	=o\bigg(\frac{1}{\sqrt{N}}\bigg).
  \end{equation}
  where $\rho^{N,2,\epsilon}$ is the $2$-marginal density of $\rho^{N,\epsilon} $.

  Furthermore, if \(k^\epsilon = (W^\epsilon*V^\epsilon)_x\) with \(W^\epsilon,V^\epsilon\) being admissible approximations with the same rate \(a_W,a_V\) as before, then the estimate~\eqref{eq: main_thm_entropy_estimate} still holds with \(\beta \in (0,\beta_1)\). Moreover, if \(W^\epsilon,V^\epsilon\) are strongly admissible, then there exists $\beta_2\in (0,\beta_\alpha)$ such that $\forall \beta\in(0,\beta_2)$, the following estimate for regularized modulated energy holds with $\epsilon=N^{-\beta}$ between \eqref{eq: regularized_Liouville_equation} and  of~\eqref{eq: regularized_aggregation_diffusion_pde}.
  \begin{align*}
	 {\mathcal K}_{N}  (\rho_t^{N,\epsilon}\vert \rho_t^{\otimes N, \epsilon})
	 = \frac{1}{\sigma^2}\E \bigg( \int_{\R^2}( W^\epsilon*V^\epsilon)(x-y)\ \Id  (\mu_t^{N,\epsilon} -\rho^{\epsilon}_t)(x)\,\Id  (\mu_t^{N,\epsilon}-\rho^\epsilon_t)(y) \bigg)=o\bigg(\frac{1}{\sqrt{N}}\bigg). 
  \end{align*}
\end{theorem}

\begin{remark}\label{remark: alpha_restriction}
  In obtaining the estimate for smoothed modulated energy, the proof has been done with the identity
  \begin{align*}
    &{\mathcal K}_{N}  (\rho_t^{N,\epsilon}\vert \rho_t^{\otimes N, \epsilon}) =\; \frac{1}{\sigma^2} \E \bigg( \Big\langle \hat{W}^\epsilon*(\mu_t^{N,\epsilon}-\rho_t^\epsilon), V^\epsilon*(\mu_t^{N,\epsilon}-\rho_t^\epsilon)\Big\rangle \bigg) ,
  \end{align*}
  where \(\hat{W}(x) = W(-x)\) is the reflection. Again choosing for instance \(W^\epsilon=J^\epsilon\) we may borrow an additional factor from the mollification kernel $J^\epsilon$, which will weaken the convergence rate estimate, or in other words, one has to choose even smaller $\beta$ to achieve the order $o(\frac{1}{\sqrt{N}})$. The restriction \(\alpha \in (\frac{1}{4},\frac{1}{2})\) is in place to guarantee the order $o(\frac{1}{\sqrt{N}})$. The convergence of the relative entropy holds also without this restriction.
\end{remark}

Additionally, for bounded force, we know from~\cite{Nikolaev2023} that convergence in probability holds for approximations \((k^\epsilon,\epsilon >0\), which satisfy a local Lipschitz bound.  Therefore, we can obtain the propagation of chaos result without mollification. 

\begin{theorem}\label{maintheorem2}
  Assume that $k\in L^\infty(\R)$, the condition for initial data \ref{ass: initial condition} holds. Suppose the Assumptions~\ref{ass: convergence_in_probability},~\ref{ass: law_of_large_numbers} hold for the approximation \(k^\epsilon=(\zeta^\epsilon(k*J^\epsilon))*J^\epsilon\). Then, for any fix \(m\in \N\), we have the convergence of the \(m\)-th marginal of the Lioville equation~\eqref{eq: Liouville_equation} to the aggregation-diffusion equation~\eqref{eq: aggregation_diffusion_pde} in the \(L^1(\R^m)\)-norm, i.e.
  \begin{equation*}
    \lim\limits_{N \to \infty} \norm{\rho^{N,m}-\rho^{\otimes m}}_{L^1([0,T];L^1(\R^m))} = 0.
  \end{equation*}
\end{theorem}

\begin{remark}
  The Theorem holds for more general approximation \(k^\epsilon\) as long as the approximation \(k^\epsilon \in L^2(\R)\) and the convergence in probability holds. We refer to~\cite{Nikolaev2023} for an overview of the topic of convergence in probability in the bounded case \(k\in L^\infty(\R)\).
\end{remark}

Let us finish the section with an overview over the constants: 
\begin{itemize}
  \item \(\alpha \in (0,1/2)\) provides the rate on the distance of the particles in the convergence in probability
  \begin{equation*}
    \sup\limits_{1 \le i \le N} |X_t^{i,\epsilon}-Y_t^{i,\epsilon}| \ge N^{-\alpha}
  \end{equation*}
  and in the law of large numbers
  \begin{equation*}
    \bigg\{ \max\limits_{1 \le i \le N} \bigg| \sum\limits_{j=1}^N  k^\epsilon(Y_t^{i,\epsilon}-Y_t^{j,\epsilon})-(k*\rho_t^\epsilon)(Y_t^{i,\epsilon}) \bigg|\ge N^{-(\delta+\alpha)}  \bigg\}.
  \end{equation*}
  \item \(\beta_\alpha \in (0,\alpha)\) provides the maximum interval \((0,\beta_\alpha)\) for the cut-off parameter \(\beta\), for which the convergence in probability and law of large numbers hold.
  \item \(\beta\) is the convergence rate of the approximated particles \(X_t^{i,\epsilon},Y_t^{i,\epsilon}\) such that \(\epsilon= N^{-\beta}\).
  \item \(\beta_1,\beta_2\) provide the maximum intervals \((0,\beta_1),(0,\beta_2)\) such that the relative entropy and modulated energy converges with rate greater than \(1/2\), (see~\eqref{eq: main_thm_entropy_estimate}).
\end{itemize}

\section{Relative entropy method}\label{sec: relative_entropy_method}

This section is devoted to present the relative entropy method for the moderate interacting problem and its connection to the $L^2$ estimate proposed by Oelschl\"ager~\cite{Oeschlager1987}. We derive the smoothed $L^2$ estimate for given force $k$ (no requirement as a potential field), and the smoothed modulated energy for potential field with convolution structure. Both lead to the estimate of the relative entropy between $\rho^{N,\epsilon}$ and $\rho^{\otimes N,\epsilon}$.

The main idea is to use the assumption of convergence in probability (Assumption~\ref{ass: convergence_in_probability}), the structure of the PDEs~\eqref{eq: regularized_Liouville_equation}, \eqref{eq: regularized_aggregation_diffusion_pde} and the law of large number (Assumption~\ref{ass: law_of_large_numbers}). Applying the Csisz{\'a}r--Kullback--Pinsker inequality~\cite[Chapter~22]{Villani2009} we provide an estimate on the \( L^1(\R)\)-norm of the marginals \(\rho^{N,m,\epsilon}\) and \(\rho^{\otimes m ,\epsilon}\) for fix \(m \in \N\).

We emphasize that the method developed in Theorem~\ref{theorem: emp_measure_l2_estimate} can be applied in different settings. Indeed, since we are working on the approximation level, our assumptions are only needed in the regularized setting. Hence, in general the assumptions on \(k\), \(V\) and \(W\) itself can be chosen more irregular, extending even to singular models. We refer to Remark~\ref{remark: general_kernels_l2main_thm} and the applications Section~\ref{sec: application} for more details.
    
\subsection{Relative entropy and modulated energy}

In this section we introduce our main quantities the relative entropy and the modulated free energy. We then show the connection between the \(L^2\)-norm
\begin{equation}\label{eq: l2_norm_intro_sec}
  \E \bigg(   \norm{ V^\epsilon*(\mu_t^{N,\epsilon}-\rho_t^\epsilon)}_{L^2(\R)}^2 \bigg),
\end{equation}
the relative entropy \( {\mathcal H}_N(\rho_t^{N,\epsilon}\vert \rho_t^{\otimes N, \epsilon})\) as well as the expectation of the modulated free energy \({\mathcal K}_{N}(\rho_t^{N,\epsilon}\vert \rho_t^{\otimes N, \epsilon})\).
This can be viewed as a combination of Oelschl\"{a}ger's results on moderated interaction and fluctuations~\cite{Oeschlager1987} and the relative entropy method developed among others by Serfaty, Jabin, Wang, Bresch and Lacker~\cite{JabinWangZhenfu2016,JabinWang2018,BreschDidierJabinWangZhenfu2019,bresch2020,Serfaty_2020,Nguyen2022,RosenzweigSerfaty2023,bresch2023,Lacker2023} for the mean-field setting. The aim is to demonstrate how both concepts connect under the convolution assumption. Finally, we derive an estimate on the relative entropy in terms of the above \(L^2\)-norm.

Following~\cite{BreschDidierJabinWangZhenfu2019}, we introduce the modulated free energy 
\begin{equation*}
  E_{N} \bigg(\rho^{N,\epsilon} \,|\;\rho^{\otimes N, \epsilon}  \bigg)
  := {\mathcal H}_N(\rho^{N,\epsilon} \vert  \rho^{\otimes N, \epsilon} )
  + {\mathcal K}_{N} ( \rho^{N,\epsilon} \vert \rho^{\otimes N, \epsilon} ) ,
\end{equation*}
where
\begin{equation*}
  {\mathcal H}_N(\rho_t^{N,\epsilon}\vert \rho_t^{\otimes N, \epsilon})
  := \frac{1}{N} \int_{\R^{N}} \rho_t^{N,\epsilon}(x_1,\ldots,x_N) \log\Bigl(\frac{\rho_t^{N,\epsilon}(x_1,\ldots,x_N)}{ \rho_t^{\otimes N,\epsilon}(x_1,\ldots,x_N)}\Bigr)
  \, \Id x_1,\ldots,x_N
\end{equation*}
is the relative entropy introduced in~\cite{JabinWangZhenfu2016} and if \(k^\epsilon=(W^\epsilon*V^\epsilon)\) is a potential 
\begin{equation*}
  {\mathcal K}_{N}  (\rho_t^{N,\epsilon}\vert \rho_t^{\otimes N, \epsilon})
  := \frac{1}{\sigma^2} \E \bigg( \int_{\R^2} (W^\epsilon*V^\epsilon)(x-y)\ \Id  (\mu_t^{N,\epsilon} -\rho^{\epsilon}_t)(x)\,\Id  (\mu_t^{N,\epsilon}-\rho^\epsilon_t)(y) \bigg)
\end{equation*}
is the expectation of the modulated energy. We refer  to~\cite{BreschDidierJabinWangZhenfu2019} and the references therein for more details on the modulated free energy.
  
Let us now explore some connections between the relative entropy and the structure presented by Oelschl\"ager~\cite{Oeschlager1987}. We start by rewriting the expectation of the free energy by using our convolution structure. A straightforward calculation shows
\begin{align}\label{eq: expextationmodulatedenergy_connection}
  &{\mathcal K}_{N}  (\rho_t^{N,\epsilon}\vert \rho_t^{\otimes N, \epsilon}) =\; \frac{1}{\sigma^2} \E \bigg( \Big\langle \hat{W}^\epsilon*(\mu_t^{N,\epsilon}-\rho_t^\epsilon), V^\epsilon*(\mu_t^{N,\epsilon}-\rho_t^\epsilon)\Big\rangle \bigg) ,
\end{align}     
where \(\hat{W}(x) = W(-x)\) is the reflection. 
Applying Young's inequality we see that it is enough to control a term of the form 
\begin{align*} 
  \E \bigg( \norm{ V^\epsilon*(\mu_t^{N,\epsilon}-\rho_t^\epsilon)}_{ L^2(\R)}^2 \bigg)
\end{align*} 
for some function \(V^\epsilon\), where we just write \(V^\epsilon\) for simplicity and understand that we can chose \(V^\epsilon=\hat{W}^\epsilon\) in all calculations below. Hence, in order to estimate \( {\mathcal K}_{N}  (\rho_t^{N,\epsilon}\vert \rho_t^{\otimes N, \epsilon})\) we can estimate the \(L^2\)-difference between the convoluted empirical measure and the solution the law of the mean-field limit~\eqref{eq: mean_field_trajectories}. This will be accomplished in Theorem~\ref{theorem: emp_measure_l2_estimate}.

But let us recall that our initial goal is to estimate the relative entropy \( {\mathcal H}_N(\rho_t^{N,\epsilon}\vert \rho_t^{\otimes N, \epsilon})\) and not \( {\mathcal K}_{N}  (\rho_t^{N,\epsilon}\vert \rho_t^{\otimes N, \epsilon})\). Therefore, let us connect the relative free energy to the \(L^2\)-norm of the derivative \(V^\epsilon_x*(\mu_t^{N,\epsilon}-\rho_t^\epsilon)\).

\begin{lemma}\label{lemma: relative_entropy_L_2_estimate}
  Let \(W^\epsilon,V^\epsilon\) be admissible and \(k^\epsilon=W^\epsilon*V^\epsilon\). Then for the non-negative solutions \(\rho^{N,\epsilon} \) of~\eqref{eq: regularized_Liouville_equation} and \(\rho^{\epsilon}\) of~\eqref{eq: regularized_aggregation_diffusion_pde}, it holds $\forall t >0$ that
  \begin{align} 
    \nonumber & {\mathcal H}_N(\rho_t^{N,\epsilon} \vert \rho_t^{\otimes N , \epsilon})
    + \frac{\sigma^2}{4N}  \int^t_0\sum\limits_{i=1}^N \int_{\R^{N}}  \Bigg|  \partial_{x_{i}} \log\Bigl(\frac{\rho_s^{N,\epsilon}(\mathrm{X}^N )}{ \rho_s^{\otimes N,\epsilon}(\mathrm{X}^N )} \Bigr) \Bigg|^2  \rho_s^{N,\epsilon}(\mathrm{X}^N)  \Id  \mathrm{X}^N \Id s \\
     &\quad \le \; \frac{\norm{W^\epsilon}^2_{L^2(\R)}}{\sigma^2}  \E\bigg( \int^t_0 \bigg(   \norm{ V^\epsilon*(\mu_s^{N,\epsilon}-\rho_s^\epsilon)}_{L^2(\R)}^2 \bigg)\Id s\bigg), \label{entropyW}
  \end{align}
\end{lemma}

\begin{proof}
  Let us compute the time derivative of the relative entropy
  \begin{align*}
    &\; \frac{\Id}{\Id t}{\mathcal H}_N (\rho_t^{N,\epsilon} \vert \rho_t^{\otimes N , \epsilon}) \\
    =&\;   \frac{1}{N} \int_{\R^{N}}  \partial_t \rho_t^{N,\epsilon}(\mathrm{X}^N ) \log\Bigl(\frac{\rho_t^{N,\epsilon}(\mathrm{X}^N )}{ \rho_t^{\otimes N,\epsilon}(\mathrm{X}^N )}\Bigr)
    + \partial_t \rho_t^{N,\epsilon}(\mathrm{X}^N )
    - \frac{\rho_t^{N,\epsilon}(\mathrm{X}^N )}{\rho_t^{\otimes N,\epsilon}(\mathrm{X}^N )} \partial_t \rho_t^{\otimes N,\epsilon}(\mathrm{X}^N )\Id \mathrm{X}^N   \\
    =&\;   \frac{1}{N} \int_{\R^{N}} \bigg( \frac{\sigma^2}{2}\sum\limits_{i=1}^N  \partial_{x_ix_i} \rho_t^{N,\epsilon}(\mathrm{X}^N) + \sum\limits_{i=1}^N
    \partial_{x_i} \bigg ( \rho^{N,\epsilon}_t(\mathrm{X}^N) \frac{1}{N} \sum\limits_{j=1}^N  k^{\epsilon} (x_{i} -x_{j} ) \bigg)  \bigg) \log\Bigl(\frac{\rho_t^{N,\epsilon}(\mathrm{X}^N )}{ \rho_t^{\otimes N,\epsilon}(\mathrm{X}^N )}\Bigr)   \\
    & \quad-  \frac{\rho_t^{N,\epsilon}(\mathrm{X}^N )}{\rho_t^{\otimes N,\epsilon}(\mathrm{X}^N )} \sum\limits_{i=1}^N  \frac{\sigma^2}{2}  \partial_{x_ix_i} \rho_{t}^{\otimes N, \epsilon} (\mathrm{X}^N) -  \sum\limits_{i=1}^N  \partial_{x_i} ((k^{\epsilon}*\rho_t^{\epsilon})(x_{i}) \rho_t^{\otimes N, \epsilon}(\mathrm{X}^N)) \Id  \mathrm{X}^N   \\
    =& \;  - \frac{\sigma^2}{2N} \sum\limits_{i=1}^N  \int_{\R^{N}}  \Bigg|  \partial_{x_{i}} \log\Bigl(\frac{\rho_t^{N,\epsilon}(\mathrm{X}^N )}{ \rho_t^{\otimes N,\epsilon}(\mathrm{X}^N )} \Bigr) \Bigg|^2  \rho_t^{N,\epsilon}(\mathrm{X}^N)  \Id  \mathrm{X}^N   \\
    &- \frac{1}{N^2}    \sum\limits_{i,j=1}^N   \int_{\R^{N}}
    \Big(  k^{\epsilon} (x_{i} -x_{j} )  - k^\epsilon*\rho_t^\epsilon(x_{i}) \Big)
    \rho^{N,\epsilon}_t(\mathrm{X}^N)
    \partial_{x_{i}} \log\Bigl(\frac{\rho_t^{N,\epsilon}(\mathrm{X}^N )}{ \rho_t^{\otimes N,\epsilon}(\mathrm{X}^N )} \Bigr)     \Id  \mathrm{X}^N  \\
    \le  &\;   - \frac{\sigma^2}{4N}  \sum\limits_{i=1}^N \int_{\R^{N}}  \Bigg|  \partial_{x_{i}} \log\Bigl(\frac{\rho_t^{N,\epsilon}(\mathrm{X}^N )}{ \rho_t^{\otimes N,\epsilon}(\mathrm{X}^N )} \Bigr) \Bigg|^2  \rho_t^{N,\epsilon}(\mathrm{X}^N)  \Id  \mathrm{X}^N   \\
    &+ \frac{1}{\sigma^2N}    \sum\limits_{i=1}^N   \int_{\R^{N}}
    \bigg| \frac{1}{N}\sum\limits_{j=1}^N  k^{\epsilon} (x_{i} -x_{j} )  - k^\epsilon*\rho_t^\epsilon(x_{i}) \bigg|^2
    \rho^{N,\epsilon}_t(\mathrm{X}^N)
    \Id  \mathrm{X}^N \\
    \le &\;    - \frac{\sigma^2}{4N}  \sum\limits_{i=1}^N \int_{\R^{N}}  \Bigg|  \partial_{x_{i}} \log\Bigl(\frac{\rho_t^{N,\epsilon}(\mathrm{X}^N )}{ \rho_t^{\otimes N,\epsilon}(\mathrm{X}^N )} \Bigr) \Bigg|^2  \rho_t^{N,\epsilon}(\mathrm{X}^N)  \Id  \mathrm{X}^N   + \frac{1}{\sigma^2} \E \bigg( \langle \mu_t^{N,\epsilon}, | k^\epsilon*(\mu_t^{N,\epsilon}-\rho_t^\epsilon)|^2 \rangle\bigg). 
  \end{align*}
  For $k^\epsilon=W^\epsilon*V^\epsilon$ we have further estimates
  \begin{align*}
    &\; \frac{1}{\sigma^2} \E \bigg( \langle \mu_t^{N,\epsilon}, | k^\epsilon*(\mu_t^{N,\epsilon}-\rho_t^\epsilon)|^2 \rangle\bigg)=\frac{1}{\sigma^2} \E \bigg( \langle \mu_t^{N,\epsilon}, |W^\epsilon*V^\epsilon*(\mu_t^{N,\epsilon}-\rho_t^\epsilon)|^2 \rangle\bigg)\\
    &\quad\leq \; \frac{\norm{W^\epsilon}^2_{L^2(\R)}}{\sigma^2}  \E \bigg(   \norm{ V^\epsilon*(\mu_t^{N,\epsilon}-\rho_t^\epsilon)}_{L^2(\R)}^2 \bigg)  \\
  \end{align*}
  Substituting the above estimate into the first inequality, while recalling that \( {\mathcal H}_N (\rho_0^{N,\epsilon} \vert \rho_0^{\otimes N , \epsilon}) =0\), proves the lemma.
\end{proof}

\begin{remark}
  Depending on the regularity of \(V^\epsilon\) and \(W^\epsilon\) one may choose to interchange the roles in the estimate. Generally, one should choose the more regular function to be \(V^\epsilon\). Indeed, in the above estimate we need only the \(L^2\)-norm of \(W^\epsilon\), while later on in Theorem~\ref{theorem: emp_measure_l2_estimate} we need the \(L^\infty\)-norm as well as the \(L^2\)-norm of not only the function \(V^\epsilon\) but also of its derivatives. Moreover, if the force \(k^\epsilon\) is a potential field, the last term has the following structure
  \begin{equation*}
    \E \bigg(   \norm{ V^\epsilon_x*(\mu_t^{N,\epsilon}-\rho_t^\epsilon)}_{L^2(\R)}^2 \bigg),
  \end{equation*}
  which will also be estimated by Theorem~\ref{theorem: emp_measure_l2_estimate}. Hence, we do not lose convergence rates in the case \(k^\epsilon= W^\epsilon*V^\epsilon_x\), but as already mentioned, we obtained an additional estimate on the modulated energy \({\mathcal K}_{N}  (\rho_t^{N,\epsilon}\vert \rho_t^{\otimes N, \epsilon}) \).
\end{remark}

Consequently, by the above discussion, in order to control the relative entropy and the modulated energy in the case \(k^\epsilon\) is a potential field, we need to find an estimate for the \(L^2\)-norm~\eqref{eq: l2_norm_intro_sec}, which was studied in the moderated regime by Oelschläger~\cite{Oeschlager1987} nearly forty years ago.

\subsection{\texorpdfstring{\(L^2\)}{L2}-estimate}

In this section we concentrate on estimating the rest term in the entropy estimate \eqref{entropyW}. 

\medskip

We present the main theorem of the article, which is formulated for a function \(V^\epsilon\), which depends on \(\epsilon\). This presentation is motivation by our case \(k^\epsilon=W^\epsilon*V^\epsilon\). We emphasize that the function in the following Theorem can be chosen independent of \(\epsilon\), but than the estimate has no connection to the modulated energy or the relative entropy (see Lemma~\ref{lemma: relative_entropy_L_2_estimate}). 

\begin{theorem}\label{theorem: emp_measure_l2_estimate}
  Suppose~\eqref{eq: initial_condition_integrability}, the convergence in probability, Assumption~\ref{ass: convergence_in_probability}, and the law of large numbers, Assumption~\ref{ass: law_of_large_numbers} hold both with rates \(\beta,\beta_\alpha,\alpha\) specified therein. Then for any \(V^\epsilon \in H^2(\R)\) the following \(L^2\)-estimate holds
  \begin{align*}
    & \E\bigg(    \sup\limits_{0 \le  t \le T} \norm{V^\epsilon*\mu_t^{N,\epsilon}  -V^\epsilon*\rho_t^{\epsilon}}_{L^2(\R)}^2\bigg) +   \frac{\sigma^2}{8}  \E  \bigg(   \int\limits_{0}^T  \norm{ V_x^\epsilon *\mu_s^{N,\epsilon} - V_x^\epsilon*\rho_s^{\epsilon}}_{L^2(\R)}^2 \Id s \bigg)  \\
    &\quad\le \;  \frac{C}{N} ( \norm{V^\epsilon}_{H^1(\R)}^2 \norm{k^\epsilon}_{L^\infty}^2 + \norm{V^\epsilon_{xx}}_{L^2(\R)}^2) +\frac{C\norm{V^\epsilon}_{H^1(\R)}^2 (1+ \norm{k^\epsilon}_{L^\infty(\R)}^2)}{N^\gamma} \\
    &\quad\quad+\frac{\norm{k^\epsilon_x}_{L^\infty(\R)} \norm{V^\epsilon}_{L^2(\R)}+ \norm{k^\epsilon}_{L^\infty(\R)}^2 \norm{V_x^\epsilon}_{L^2(\R)}^2
    + \norm{V^\epsilon}_{L^2(\R)}^2}{N^{2\alpha }} \\
    &\quad\quad+C\frac{\norm{V_x^\epsilon}_{L^2(\R)}^2 (1+ \norm{k_x^\epsilon}_{L^\infty(\R)}) +  \norm{V_x^\epsilon}_{L^2(\R)} \norm{V^\epsilon_{xx}}_{L^2(\R)} \norm{k^\epsilon}_{L^\infty(\R^d)} }{N^{\alpha+\frac12 }},
  \end{align*}
  where \(C\) depends on $T$, $\sigma$, $\gamma$, $C_{\mathrm{BDG}}$.
\end{theorem}
	
\begin{remark}\label{remark: general_kernels_l2main_thm}
  The only ingredients we need for completing the proof of theorem \ref{theorem: emp_measure_l2_estimate} are the convergence in probability of the particle system \(\mathbf{X}^{N,\epsilon}\) to the mean-field limit \(\mathbf{Y}^{N,\epsilon}\)~\eqref{eq: convergence_in_probability} as well as the law of large numbers~\eqref{eq: lln_small_b_set}. But the convergence in probability and the law of large numbers are known for a variety of interaction force kernels, see for instance~\cite{lazarovici2017mean,Fetecau_2019,HuangLiuLu2019,HuangLiuPickl2020}. Hence, this result can be extended for a variety of interaction force kernels. Moreover, the kernels can also be \(d\)-dimensional since the estimates we used are dimension-free. We refer to Section~\ref{sec: application} for applicable models such as the case with Coulomb force. Actually, the estimates become dimension dependent by the choice of $\beta$. Consequently, the rate of convergence becomes dependent on the dimension. Nevertheless, the steps of the proof work analogously in multi-dimensional setting by replacing the multiplication with the scalar product, the absolute value with the Euclidean norm and the It\^{o}'s formula with its multidimensional counter part.
\end{remark}

\begin{remark}
  The results in Theorem~\ref{theorem: emp_measure_l2_estimate} state that \(\mu_t^{N,\epsilon}\) is close to \(\rho_t^\epsilon\) in the mollified \(L^2\)-norm. By the propagation of chaos we expect that this quantity should be small since \(\mu_t^{N,\epsilon}-\rho_t^\epsilon \) should ideally vanish in the limit. The majority of work, which lies ahead, is to estimate this \(L^2\)-norm with a good rate. In the process we will also obtain an estimate on the derivative \(V^\epsilon_x*(\mu_t^{N,\epsilon}-\rho_t^\epsilon)\). This is no surprise, since the estimate follows the structure of the classic a priori \(L^2\)-estimate for the parabolic equation~\cite[Chapter~3]{WuZhouqunYin2006}. As a result, we obtain in the \(L^2(\P)\)-norm an \(L^\infty([0,T];L^2(\R))\)-bound and as usual an \(L^2([0,T];L^2(\R))\)-bound for the derivative. In combination with Lemma~\ref{lemma: relative_entropy_L_2_estimate} this will allow us to obtain a bound on the relative entropy \({\mathcal{H}}(\rho_t^{N,\epsilon} \vert \rho_t^{\otimes N , \epsilon})\). Additionally, if the interaction force is a potential field we obtain an estimate for \( {\mathcal K}_{N}  (\rho_t^{N,\epsilon}\vert \rho_t^{\otimes N, \epsilon})\) by equality~\eqref{eq: expextationmodulatedenergy_connection}.
\end{remark}

Let us start by describing the dynamic of the empirical measure \(\mu_t^{N,\epsilon}\). Applying It\^{o}'s formula to a sufficiently smooth function \(f\), we obtain
\begin{align*}
  \qv{f, \mu_t^{N,\epsilon}} =&\;  \frac{1}{N} \sum\limits_{i=1}^N f(X_t^{i,\epsilon}) \\
  =  & \; \qv{f, \mu^N_0}
  -  \frac{1}{N} \sum\limits_{i=1}^N \frac{1}{N} \sum\limits_{j=1 }^N \int\limits_0^t f_x(X_s^{i,\epsilon}) k(X_s^{i,\epsilon}-X_s^{j,\epsilon}) \Id s
  + \frac{\sigma}{N} \sum\limits_{i=1}^N  \int\limits_0^t f_x(X_s^{i,\epsilon}) \Id B_s   \\
  &+ \frac{\sigma^2}{2N} \sum\limits_{i=1}^N  \int\limits_0^t f_{xx}(X_s^{i,\epsilon}) \Id s .
\end{align*}
Taking the expectation and using the fact that we have a density of \(\mathbf{X}_s^{N, \epsilon}\), provides a weak formulation of the Lioville equation~\eqref{eq: regularized_Liouville_equation}. If we want to compare it to the mean-field law, we need to make the crucial observation that the stochastic integral in the above equation should vanish after taking the expectation. In other words, we have no term in the regularized PDE~\eqref{eq: regularized_aggregation_diffusion_pde}, which corresponds to the stochastic integral. If the integrand is smooth enough then obviously the stochastic integral vanishes. However, we need to compute the following difference
\begin{equation*}
  \E\bigg( \sup\limits_{0 \le t \le T} \norm{V^\epsilon*\mu_t^{N,\epsilon}  -V^\epsilon
  *\rho_t^\epsilon}_{L^2(\R)}^2\bigg).
\end{equation*}
Therefore, we need somehow transfer the naive approach to the more complex expected value. Applying the above dynamic we prove the following lemma, which allows us to treat the convolution \(V^\epsilon * \mu_t^{N,\epsilon}\) as if the stochastic integral vanishes.

\begin{lemma}\label{lem: aux_reduction_empirical_measure}
  Let \(\mu_t^{N,\epsilon}\) defined by~\eqref{eq: def_emprirical_measure_reg}. Then, we have the following inequality
  \begin{align*}
    &\;  \E\bigg( \sup\limits_{0 \le t \le T} \norm{V^\epsilon*\mu_t^{N,\epsilon}  -V^\epsilon
    *\rho_t^\epsilon}_{L^2(\R)}^2\bigg) \\
    &\quad\le \;   2\E \bigg( \sup\limits_{0 \le t \le T}\int_{\R}  \bigg| \frac{1}{N}  \sum\limits_{i=1}^N \bigg(V^\epsilon(y-X_0^{i})
    + \frac{1}{N} \sum\limits_{j=1}^N \int\limits_{0}^t V^\epsilon_x(y-X_s^{i,\epsilon}) k^\epsilon(X_s^{i,\epsilon}-X_s^{j,\epsilon}) \Id s \\
    &\quad\quad + \frac{\sigma^2}{2} \int\limits_0^t V^\epsilon_{xx}(y-X_t^{i,\epsilon}) \Id s \bigg)  -V^\epsilon*\rho_t^\epsilon(y) \bigg|^2 \Id y \bigg)
    +\frac{2T \sigma^2  C_{\mathrm{BDG}}}{N} \norm{V^\epsilon_x}_{L^2(\R)}^2.
  \end{align*}
\end{lemma}

\begin{proof}
  We use It\^{o}'s formula, the dynamics~\eqref{eq: regularized_particle_system} and the Burkholder--Davis--Gundy inequality to find
  \begin{align*}
    & \; \E \bigg( \sup\limits_{0 \le t \le T} \norm{V^\epsilon*\mu_t^{N,\epsilon}  -V^\epsilon*\rho_t^\epsilon}_{L^2(\R)}^2 \bigg) \\
     =& \;  \E \bigg( \sup\limits_{0 \le t \le T} \int_{\R} \bigg| \frac{1}{N}  \sum\limits_{i=1}^N V^\epsilon(y-X_s^{i,\epsilon})  -V^\epsilon*\rho_t^\epsilon(y) \bigg|^2 \Id y  \bigg) \\
     =& \; \E \bigg( \sup\limits_{0 \le t \le T}\int_{\R}\bigg| \frac{1}{N}  \sum\limits_{i=1}^N \bigg(V^\epsilon(y-X_0^{i})
    + \frac{1}{N} \sum\limits_{j=1}^N  \int\limits_{0}^t V^\epsilon_x(y-X_s^{i,\epsilon}) k^\epsilon(X_s^{i,\epsilon}-X_s^{j,\epsilon}) \Id s \\
    &\quad + \frac{\sigma^2}{2} \int\limits_0^t V^\epsilon_{xx}(y-X_t^{i,\epsilon}) \Id s
    - \sigma \int\limits_0^t V^\epsilon_x(y-X_s^{i,\epsilon}) \Id B_s^{i} \bigg)
    -V^\epsilon*\rho_t^\epsilon(y) \bigg|^2 \Id y  \bigg) \\
     \le &  \; 2\E \bigg( \sup\limits_{0 \le t \le T}\int_{\R}  \bigg| \frac{1}{N}  \sum\limits_{i=1}^N \bigg(V^\epsilon(y-X_0^{i}) + \frac{1}{N} \sum\limits_{j=1}^N \int\limits_{0}^t V^\epsilon_x(y-X_s^{i,\epsilon}) k^\epsilon(X_s^{i,\epsilon}-X_s^{j,\epsilon}) \Id s  \\
    &\quad + \frac{\sigma^2}{2} \int\limits_0^t V^\epsilon_{xx}(y-X_t^{i,\epsilon}) \Id s \bigg)
    -V^\epsilon*\rho_t^\epsilon(y) \bigg|^2 \Id y \bigg) \\
    & +2 \sigma^2 \E\bigg( \sup\limits_{0 \le t \le T} \int_{\R} \bigg| \frac{1}{N} \sum\limits_{i=1}^N  \int\limits_0^t V^\epsilon_x(y-X_s^{i,\epsilon}) \Id B_s^{i} \bigg|^2 \Id y  \bigg).
  \end{align*}
  It remains to estimate the last term by the Burkholder--Davis--Gundy (BDG) inequality, 
  \begin{align*}
    & 2 \sigma^2 \E\bigg( \sup\limits_{0 \le t \le T} \int_{\R} \bigg| \frac{1}{N} \sum\limits_{i=1}^N  \int\limits_0^t V^\epsilon_x(y-X_s^{i,\epsilon}) \Id B_s^{i} \bigg|^2 \Id y  \bigg) \\
     \le &\;  2  \sigma^2  \int_{\R} \E \bigg( \sup\limits_{0 \le t \le T} \bigg| \frac{1}{N} \sum\limits_{i=1}^N  \int\limits_0^t V^\epsilon_x(y-X_s^{i,\epsilon}) \Id B_s^{i} \bigg|^2  \bigg)  \Id y \\
    \le &\;  2 \sigma^2  C_{\mathrm{BDG}} \int_{\R} \E \bigg( \bigg \langle \frac{1}{N} \sum\limits_{i=1}^N  \int\limits_0^\cdot V^\epsilon_x(y-X_s^{i,\epsilon}) \Id B_s^{i} \bigg\rangle_T  \bigg) \Id y \\
     \le &\;  \frac{2  \sigma^2  C_{\mathrm{BDG}}}{N^2} \int_{\R} \E \bigg(  \sum\limits_{i=1}^N  \int\limits_0^T |V^\epsilon_x(y-X_s^{i,\epsilon})|^2 \Id s  \bigg) \Id y  \le \;  \frac{2T \sigma^2  C_{\mathrm{BDG}}}{N} \norm{V^\epsilon_x}_{L^2(\R)}^2. 
  \end{align*}
  Inserting this calculation into the previous inequality proves the lemma.
\end{proof}

\begin{proof}[Proof of Theorem~\ref{theorem: emp_measure_l2_estimate}]\let\qed\relax
  By Lemma~\ref{lem: aux_reduction_empirical_measure} we can ignore the stochastic integral in the processes \((X^{i},\epsilon, t \ge 0 )\), which determine the empirical measure \(\mu_t^{N,\epsilon}\). Hence, let us write
  \begin{align*}
    &\;  V^\epsilon \tilde{*} \mu_t^{N,\epsilon} (y)  \\
    :=&\;  \frac{1}{N}  \sum\limits_{i=1}^N \bigg(V^\epsilon(y-X_0^{i})+\frac{1}{N} \sum\limits_{j=1}^N \int\limits_{0}^t V^\epsilon_x(y-X_s^{i,\epsilon}) k^\epsilon(X_s^{i,\epsilon}-X_s^{j,\epsilon}) \Id s + \frac{\sigma^2}{2} \int\limits_0^t V^\epsilon_{xx}(y-X_t^{i,\epsilon}) \Id s \bigg)
  \end{align*}
  for the convolution \(V^\epsilon*\mu_t^{N,\epsilon}\) after applying It\^{o}'s formula but without the stochastic integral. Then, we have
  \begin{align*}
    & \;\norm{V^\epsilon \tilde{*} \mu_t^{N,\epsilon}  -V^\epsilon*\rho_t^{\epsilon}}_{L^2(\R)}^2 \\
    = &\;  \norm{V^\epsilon *\mu^N_0  -V^\epsilon*\rho_0}_{L^2(\R)}^2
    + 2 \int\limits_{0}^t   \qv{\partial_s (V^\epsilon \tilde{*}\mu_s^{N,\epsilon}  -V^\epsilon*\rho_s^{\epsilon}), V^\epsilon\tilde{*}\mu_s^{N,\epsilon}  -V^\epsilon*\rho_s^{\epsilon}}_{L^2(\R)} \Id s   ,
  \end{align*}
  where we notice that for the initial time \(t=0\), we have \(V^\epsilon\tilde{*} \mu_0^N =  V^\epsilon* \mu_0^N\) by definition. Let us remark that since all integrands are smooth enough we have \((V^\epsilon \tilde{*} \mu_t^{N,\epsilon})_x = V^\epsilon_x \tilde{*} \mu_t^{N,\epsilon}\). Next, plugging in \(V^\epsilon\tilde{*}\mu_s^{N,\epsilon}\) and differentiate we obtain
  \begin{align*}
    &\;  \qv{\partial_s V^\epsilon\tilde{*} \mu_s^{N,\epsilon}, V^\epsilon\tilde{*} \mu_s^{N,\epsilon}  -V^\epsilon*\rho_s^{\epsilon}}_{L^2(\R)}   \\
    =&\; \bigg \langle
    \frac{1}{N} \sum\limits_{i=1}^N \frac{1}{N} \sum\limits_{j=1}^N V^\epsilon_{x}(\cdot-X_s^{i,\epsilon}) k^\epsilon(X_s^{i,\epsilon}-X_s^{j,\epsilon})  \\
    & \quad + \frac{\sigma^2}{2N} \sum\limits_{i=1}^N  V^\epsilon_{xx}(\cdot-X_s^{i,\epsilon}) , V^\epsilon\tilde{*} \mu_s^{N,\epsilon}  -V^\epsilon*\rho_s^{\epsilon} \bigg \rangle_{L^2(\R)} \\
    = &\;    \bigg \langle \frac{1}{N} \sum\limits_{i=1}^N \frac{1}{N} \sum\limits_{j=1}^N V^\epsilon_{x}(\cdot-X_s^{i,\epsilon}) k^\epsilon(X_s^{i,\epsilon}-X_s^{j,\epsilon}), V^\epsilon\tilde{*} \mu_s^{N,\epsilon}  -V^\epsilon*\rho_s^{\epsilon}  \bigg \rangle_{L^2(\R)} \\
    &  - \frac{\sigma^2}{2} \qv{  V^\epsilon_{x}* \mu_s^{N,\epsilon} , V^\epsilon_x\tilde{*}\mu_s^{N,\epsilon}  -V^\epsilon_x*\rho_s^{\epsilon}}_{L^2(\R)} .
  \end{align*}
  Similar, \(\rho_s^{\epsilon}\) is a weak solution to our PDE~\eqref{eq: regularized_aggregation_diffusion_pde}, which implies
  \begin{align*}
    &\; \qv{\partial_s V^\epsilon*\rho_s^{\epsilon}), V^\epsilon\tilde{*} \mu_s^{N,\epsilon}  -V^\epsilon*\rho_s^{\epsilon}}_{L^2(\R)} \\
    =&\;   \bigg\langle V^\epsilon*\bigg( \frac{\sigma^2}{2} (\rho_s^{\epsilon})_{xx} +  ((k^\epsilon*\rho_s^{\epsilon}) \rho_s^{\epsilon})_x \bigg) , V^\epsilon\tilde{*} \mu_s^{N,\epsilon}  -V^\epsilon*\rho_s^{\epsilon}\bigg\rangle_{L^2(\R)}  \\
    =&\;   - \frac{\sigma^2}{2}  \qv{ V^\epsilon *(\rho_s^{\epsilon})_{x}, V^\epsilon_x\tilde{*} \mu_s^{N,\epsilon}  -V^\epsilon_x*\rho_s^{\epsilon}}_{L^2(\R)}\\
    &+  \qv{ V^\epsilon*((k^\epsilon*\rho_s^{\epsilon}) \rho_s^{\epsilon})_x  , V^\epsilon\tilde{*} \mu_s^{N,\epsilon}  -V^\epsilon*\rho_s^{\epsilon}}_{L^2(\R)}.
  \end{align*}
  Combing the last two calculations, we find
  \begin{align*}
    & \; \norm{V^\epsilon\tilde{*}\mu_t^{N,\epsilon}  -V^\epsilon*\rho_t^{\epsilon}}_{L^2(\R)}^2 \\
    =&\; \norm{V^\epsilon*\mu^N_0  -V^\epsilon*\rho_0}_{L^2(\R)}^2
    - 2 \int\limits_{0}^t  \frac{\sigma^2}{2}  \qv{ V^\epsilon_x *\mu_s^{N,\epsilon} - V^\epsilon_x*\rho_s^{\epsilon} , V^\epsilon_x\tilde{*}\mu_s^{N,\epsilon}  -V^\epsilon_x*\rho_s^{\epsilon}}_{L^2(\R)} \Id s \\
    & \hspace{-0.3em} +  \int\limits_{0}^t \bigg \langle   \frac{2}{N^2} \sum\limits_{i,j=1}^N  V^\epsilon_{x}(\cdot-X_s^{i,\epsilon}) k^\epsilon(X_s^{i,\epsilon}-X_s^{j,\epsilon})  - V^\epsilon_x*((k^\epsilon*\rho_s^{\epsilon}) \rho_s^{\epsilon}) ,
      V^\epsilon\tilde{*}\mu_s^{N,\epsilon}  -V^\epsilon*\rho_s^{\epsilon} \bigg\rangle_{L^2(\R)} \Id s .
  \end{align*}
  The goal is now to insert \(V^\epsilon*\mu_s^{N,\epsilon}\) back into the equation. Hence, for the absorption term we have
  \begin{align*}
    &\; - \int\limits_{0}^t  \frac{\sigma^2}{2}  \qv{ V^\epsilon_x *\mu_s^{N,\epsilon} - V^\epsilon_x*\rho_s^{\epsilon} , V^\epsilon_x\tilde{*}\mu_s^{N,\epsilon}  -V^\epsilon_x*\rho_s^{\epsilon}}_{L^2(\R)} \Id s \\
    =&\; - \int\limits_{0}^t  \frac{\sigma^2}{2}  \norm{ V^\epsilon_x *\mu_s^{N,\epsilon} - V^\epsilon_x*\rho_s^{\epsilon}}_{L^2(\R)}^2 \Id s  \\
    &+ \int\limits_{0}^t  \frac{\sigma^2}{2}  \bigg\langle V^\epsilon_x *\mu_s^{N,\epsilon} - V^\epsilon_x*\rho_s^{\epsilon} ,  \frac{\sigma}{N} \sum\limits_{i=1}^N \int\limits_0^s V^\epsilon_{xx}(\cdot-X_u^{i}) \Id B_u \bigg\rangle_{L^2(\R)} \Id s \\
    \le &\;  -  \int\limits_{0}^t  \frac{\sigma^2}{2}  \norm{ V^\epsilon_x *\mu_s^{N,\epsilon} - V^\epsilon_x*\rho_s^{\epsilon}}_{L^2(\R)}^2 \Id s  + \int\limits_{0}^t  \frac{\sigma^2}{16}  \norm{ V^\epsilon_x *\mu_s^{N,\epsilon} - V^\epsilon_x*\rho_s^{\epsilon}}_{L^2(\R)}^2 \\
    & + 2\sigma^2 \norm{ \frac{\sigma}{N} \sum\limits_{i=1}^N \int\limits_0^s V^\epsilon_{xx}(\cdot-X_u^{i}) \Id B_u }_{L^2(\R)}^2 \Id s \\
    =&\;  -  \int\limits_{0}^t  \frac{7\sigma^2}{16}  \norm{ V^\epsilon_x *\mu_s^{N,\epsilon} - V^\epsilon_x*\rho_s^{\epsilon}}_{L^2(\R)}^2 \Id s  + 2 \sigma^2 \int\limits_{0}^t  \norm{ \frac{\sigma}{N} \sum\limits_{i=1}^N  \int\limits_0^s V^\epsilon_{xx}(\cdot-X_u^{i}) \Id B_u }_{L^2(\R)}^2 \Id s
  \end{align*}
  and for the last term
  \begin{align*}
    & \E \bigg( \sup\limits_{0 \le t \le T}  \int\limits_{0}^t  \bigg \langle  \frac{1}{N^2} \sum\limits_{i,j=1}^N  V^\epsilon_{x}(\cdot-X_s^{i,\epsilon}) k^\epsilon(X_s^{i,\epsilon}-X_s^{j,\epsilon}) - V^\epsilon_x*((k^\epsilon*\rho_s^{\epsilon}) \rho_s^{\epsilon}), \\
    &\quad V^\epsilon\tilde{*}\mu_s^{N,\epsilon}  -V^\epsilon*\rho_s^{\epsilon} \bigg\rangle_{L^2(\R)} \Id s  \bigg)  \\
    \le &\,  \E \bigg( \sup\limits_{0 \le t \le T}  \int\limits_{0}^t \bigg| \bigg \langle  \frac{1}{N^2} \sum\limits_{i,j=1}^N  V^\epsilon_{x}(\cdot-X_s^{i,\epsilon}) k^\epsilon(X_s^{i,\epsilon}-X_s^{j,\epsilon}) - V^\epsilon_x*((k^\epsilon*\rho_s^{\epsilon}) \rho_s^{\epsilon}) , \\
    &\quad   V^\epsilon*\mu_s^{N,\epsilon}  -V^\epsilon*\rho_s^{\epsilon} \bigg\rangle_{L^2(\R)}\bigg| \Id s  \bigg)  + \E \bigg( \sup\limits_{0 \le t \le T}  \int\limits_{0}^t \bigg| \bigg \langle  \frac{1}{N^2} \sum\limits_{i,j=1}^N  V^\epsilon_{x}(\cdot-X_s^{i,\epsilon}) k^\epsilon(X_s^{i,\epsilon}-X_s^{j,\epsilon}) \\
    & \quad - V^\epsilon_x*((k^\epsilon*\rho_s^{\epsilon}) \rho_s^{\epsilon})  ,
    \frac{\sigma}{N}\sum\limits_{l=1}^N  -\int\limits_0^s V^\epsilon_x(\cdot-X_u^{l}) \Id B_u^{l} \bigg\rangle_{L^2(\R)}\bigg| \Id s  \bigg).
  \end{align*}
  Applying Lemma~\ref{lem: aux_reduction_empirical_measure} and put together the above estimates we have shown
  \begin{align}\label{maininequality}
    & \E\bigg(    \sup\limits_{0 \le  t \le T} \norm{V^\epsilon*\mu_t^{N,\epsilon}  -V^\epsilon*\rho_t^{\epsilon}}_{L^2(\R)}^2\bigg)  \\
    \le & \E\bigg(    \sup\limits_{0 \le  t \le T} \norm{V^\epsilon\tilde{*}\mu_t^{N,\epsilon}  -V^\epsilon*\rho_t^{\epsilon}}_{L^2(\R)}^2\bigg) 
    + \frac{2T \sigma^2  C_{\mathrm{BDG}}}{N} \norm{V^\epsilon_x}_{L^2(\R)}^2 \nonumber  \\
    \nonumber \le & \; 2  \E\bigg(    \sup\limits_{0 \le  t \le T} \bigg(  \norm{V^\epsilon*\mu^N_0  -V^\epsilon*\rho_0}_{L^2(\R)}^2
    - \int\limits_{0}^t  \frac{7\sigma^2}{16}  \norm{ V^\epsilon_x *\mu_s^{N,\epsilon} - V^\epsilon_x*\rho_s^{\epsilon}}_{L^2(\R)}^2 \Id s \\
    \nonumber    & +  \int\limits_{0}^t  \bigg|\bigg \langle  \frac{1}{N^2} \sum\limits_{i,j=1}^N  V^\epsilon_{x}(\cdot-X_s^{i,\epsilon}) k^\epsilon(X_s^{i,\epsilon}-X_s^{j,\epsilon}) - V^\epsilon_x*((k^\epsilon*\rho_s^{\epsilon}) \rho_s^{\epsilon}) , \\
    \nonumber    &\quad  V^\epsilon*\mu_s^{N,\epsilon}  -V^\epsilon*\rho_s^{\epsilon} \bigg \rangle_{L^2(\R)} \bigg| \Id s \bigg) \bigg) \\
    \nonumber  &+ 2 \E \bigg( \sup\limits_{0 \le t \le T}  \int\limits_{0}^t \bigg| \bigg \langle  \frac{1}{N^2} \sum\limits_{i,j=1}^N  V^\epsilon_{x}(\cdot-X_s^{i,\epsilon}) k^\epsilon(X_s^{i,\epsilon}-X_s^{j,\epsilon}) - V^\epsilon_x*((k^\epsilon*\rho_s^{\epsilon}) \rho_s^{\epsilon}) , \\
    \nonumber    &\quad \frac{\sigma}{N}\sum\limits_{l=1}^N  -\int\limits_0^s V^\epsilon_x(\cdot-X_u^{l}) \Id B_u^{l} \bigg\rangle_{L^2(\R)}\bigg| \Id s  \bigg)  \\
    \nonumber    &+  4 \sigma^2 \E \bigg( \sup\limits_{0 \le t \le T}   \int\limits_{0}^t  \norm{ \frac{\sigma}{N} \sum\limits_{i=1}^N \int\limits_0^s V^\epsilon_{xx}(\cdot-X_u^{i}) \Id B_u }_{L^2(\R)}^2 \Id s\bigg)
    +\frac{2T \sigma^2  C_{\mathrm{BDG}}}{N} \norm{V^\epsilon_x}_{L^2(\R)}^2 .
  \end{align}
  Now we want to estimate each term on its own. We will split the fourth terms into fourth separate lemmata to keep a readable structure. The theorem follows immediately by combining Lemma~\ref{lemma: l2_theorem_absobation_lemma} and the inequalities~\eqref{eq: main_thm_initial_data_inequality},~\eqref{eq: main_theorem_GNS_absorbation},~\eqref{eq: main_theorem_stochastic_integral} in the lemmata below. We will summarize the estimate after we prove the following lemmata.
\end{proof}

\begin{lemma}[Initial Value Inequality]
  Let the assumptions of Theorem~\ref{theorem: emp_measure_l2_estimate} hold true. Then
  \begin{equation}\label{eq: main_thm_initial_data_inequality}
    \E\bigg( \norm{V^\epsilon*\mu^N_0  -V^\epsilon*\rho_0}_{L^2(\R)}^2 \bigg)  \le \frac{2}{N}  \norm{V^\epsilon}_{L^2(\R)}^2 .
  \end{equation}
\end{lemma}

\begin{proof}
  We compute
  \begin{align*}
    &\;   \E\bigg(    \norm{V^\epsilon*\mu^N_0  -V^\epsilon*\rho_0}_{L^2(\R)}^2  \bigg) \\
    =& \; \int_{\R}  \E \bigg(  (V^\epsilon*\mu^N_0(y))^2  -2 V^\epsilon*\mu^N_0(y) V^\epsilon*\rho_0(y) + (V^\epsilon*\rho_0(y))^2 \bigg) \Id y  \\
    =&\;  \int_{\R}  \frac{1}{N^2}\sum\limits_{i,j=1}^N \E \bigg(  V^\epsilon(y-X_0^{i}) V^\epsilon(y-X_0^{j}) \bigg) - \frac{2}{N} \sum\limits_{i=1}^N \E \bigg(  V^\epsilon(y-X_0^{i}) \bigg) V^\epsilon*\rho_0(y) \\
    &\quad +  (V^\epsilon*\rho_0(y))^2 \Id y\\
    =& \int_{\R}  \frac{N^2-N}{N^2} (V^\epsilon*\rho_0(y))^2 + \frac{1}{N} (V^\epsilon)^2*\rho_0(y)  - (V^\epsilon*\rho_0(y))^2 \Id y\\
    =&\; \frac{1}{N}\int_{\R} (V^\epsilon)^2*\rho_0(y)  - (V^\epsilon*\rho_0(y))^2 \Id y\\
    \le & \;  \frac{1}{N} \big( \norm{(V^\epsilon)^2*\rho_0}_{L^1(\R)} + \norm{V^\epsilon*\rho_0}_{L^2(\R)}^2 \big) \le \;  \frac{2}{N}  \norm{V^\epsilon}_{L^2(\R)}^2 \norm{\rho_0}_{L^1(\R)},
  \end{align*}
  where we used the fact that the initial particles are i.i.d. and Young's inequality for convolutions in the last step.
\end{proof}

\begin{lemma}[Absorbation Inequality]\label{lemma: l2_theorem_absobation_lemma}
  Let the assumptions of Theorem~\ref{theorem: emp_measure_l2_estimate} hold true. Then
  \begin{align*}
  \begin{split}
    &\; \E \bigg(    \sup\limits_{0 \le  t \le T}  \int\limits_{0}^t - \frac{7\sigma^2}{16}  \norm{ V^\epsilon_x *\mu_s^{N,\epsilon} - V^\epsilon_x*\rho_s^{\epsilon}}_{L^2(\R)}^2   \\
    &\quad+ \bigg| \bigg \langle  \frac{1}{N^2} \sum\limits_{i,j=1}^N  V^\epsilon_{x}(\cdot-X_s^{i,\epsilon}) k^\epsilon(X_s^{i,\epsilon}-X_s^{j,\epsilon}) - V^\epsilon_x*((k^\epsilon*\rho_s^{\epsilon}) \rho_s^{\epsilon}) , V^\epsilon*(\mu_s^{N,\epsilon} -\rho_s^{\epsilon} ) \bigg \rangle_{L^2(\R)} \bigg|\Id s \bigg)  \\
    &\le \;    \frac{16T\norm{k_x^\epsilon}_{L^\infty(\R)}^2 \norm{V^\epsilon}_{L^2(\R)}^2}{\sigma^2 N^{2\alpha}}   +\frac{4T\norm{V^\epsilon}_{L^2(\R)}^2}{\sigma^2N^{2(\alpha+\delta)}} \\
    &\quad+\Big(\frac{4\norm{   V^\epsilon_x}_{L^2(\R)}^2}{N^{2\alpha} \sigma^2}  +  \frac{16  \norm{V^\epsilon}_{L^2(\R)}^2 }{N\sigma^2} \Big)  \int\limits_0^T \norm{k^\epsilon*\rho_s^{\epsilon}}_{L^\infty(\R)}^2  \Id s \\
    & \quad+  \frac{C(\gamma)T}{N^\gamma}\Big(\norm{ V^\epsilon }_{L^2(\R)}^2  \norm{k^\epsilon}_{L^\infty(\R)}^2 +  \norm{ V^\epsilon_x}_{L^2(\R)}^2 \Big) \\
    &\quad -   \frac{\sigma^2}{8}  \E  \bigg(    \sup\limits_{0 \le  t \le T}  \int\limits_{0}^t  \norm{ V_x^\epsilon *\mu_s^{N,\epsilon} - V_x^\epsilon*\rho_s^{\epsilon}}_{L^2(\R)}^2 \Id s \bigg)  .
  \end{split}
  \end{align*}
\end{lemma}

\begin{proof}
  Before we begin the proof of this lemma, we will provide an overview of our approach. Our main strategy is to utilize  the convergence in probability of the particle \(X_t^{i,\epsilon}\) to their mean-field limit \(Y_t^{i,\epsilon}\) (Assumption~\ref{ass: convergence_in_probability}) in combination with the law of large numbers (Assumption~\ref{ass: law_of_large_numbers}). This implies that the "bad set", where the particles are apart is small in probability with arbitrary algebraic convergence rate. Therefore, we may assume that \(X_t^{i,\epsilon}\) is close to \(Y_t^{i,\epsilon}\), and we formally replace the empirical measure of \((X_t^{i,\epsilon}, i =1, \ldots, N)\) with the empirical measure associated with \((Y_t^{i,\epsilon},i=1,\ldots,N)\). However, \((Y_t^{i,\epsilon},i=1,\ldots,N)\) has more desirable properties. For instance, the particles are independent and have density \(\rho_t^{\epsilon} \in L^1(\R)\) and often even \(\rho_t^\epsilon \in L^\infty(\R)\). This allows us to apply the law of large numbers~\eqref{eq: lln_small_b_set}, which ultimately proves the claim.
 
  Let us start by splitting our probability space \(\Omega\) into two sets. On one set \( B_s^{\alpha}\) the particles are close to the mean-field particles in probability and ``satisfy'' the law of large numbers. The other set we take as the complement \((B_s^{\alpha})^{\mathrm{c}}\), which has small probability by inequalities~\eqref{eq: convergence_in_probability} and~\eqref{eq: lln_small_b_set}.

  More precisely, we have
  \begin{align}\label{Balpha}
    B_s^{\alpha}
    &=  \bigg\{ \omega \in \Omega \colon \max\limits_{i=1,\ldots,N} | X^{i,\epsilon}_s(\omega) - Y^{i,\epsilon}_s(\omega) | \le N^{-\alpha} \bigg \} \nonumber \\
    &\quad \cap \bigg\{ \omega \in \Omega \colon \max\limits_{i=1,\ldots,N} \bigg|\frac{1}{N} \sum\limits_{j=1}^N k^\epsilon(Y_s^{i,\epsilon}(\omega) -Y_s^{j,\epsilon}(\omega)) - (k^\epsilon*\rho_s^{\epsilon})(Y_s^{i,\epsilon}(\omega)) \bigg| \le N^{-(\alpha+\delta)}  \bigg\}
  \end{align}
  for some \(\delta > 0\) such that \(0< \alpha + \delta < 1/2 \) and we have the estimate \(\P((B_s^{\alpha})^{\mathrm{c}}) \le C(\gamma) N^{-\gamma} \) for all \(\gamma > 0\) by~~\eqref{eq: convergence_in_probability} and~\eqref{eq: lln_small_b_set}. Let us rewrite the last Lebesgue integral on the left-hand side of our claim as follows
  \begin{align*}
    & \E\bigg( \sup\limits_{0 \le t \le T} \int\limits_0^t  \bigg|\bigg \langle  \frac{1}{N^2} \sum\limits_{i,j=1}^N  V^\epsilon_{x}(\cdot-X_s^{i,\epsilon}) k^\epsilon(X_s^{i,\epsilon}-X_s^{j,\epsilon}) - V^\epsilon_x*((k^\epsilon*\rho_s^{\epsilon}) \rho_s^{\epsilon}) , \\
    &\qquad  V^\epsilon*\mu_s^{N,\epsilon}  -V^\epsilon*\rho_s^{\epsilon} \bigg \rangle_{L^2(\R)}\bigg| \Id s \bigg)  \\
    &\quad\le \; \E\bigg( \sup\limits_{0 \le t \le T} \int\limits_0^t  \bigg|\bigg \langle  \frac{1}{N^2} \sum\limits_{i,j=1}^N  V^\epsilon_{x}(\cdot-X_s^{i,\epsilon}) k^\epsilon(X_s^{i,\epsilon}-X_s^{j,\epsilon}) - V^\epsilon_x*((k^\epsilon*\rho_s^{\epsilon}) \rho_s^{\epsilon})   , \\
    &\qquad V^\epsilon*\mu_s^{N,\epsilon}  -V^\epsilon*\rho_s^{\epsilon} \bigg \rangle_{L^2(\R)}  \bigg|\Big(\indicator{(B_s^\alpha)}+\indicator{(B_s^\alpha)^{\mathrm{c}}}\Big)\Id s  \bigg).
  \end{align*}
  We are going to estimate each term by itself.

  \textit{On the set \(B_s^\alpha\):} In order to estimate the first term above we let \(\omega \in B_s^\alpha\) and will not write the indicator function. Then we have
  \begin{align*}
    &  \hspace{-0.3em}\frac{1}{	N^2} \sum\limits_{i,j=1}^N \Big\langle  V^\epsilon_x ( \cdot-X_s^{i,\epsilon}(\omega)) k^\epsilon(X_s^{i,\epsilon}(\omega)-X_s^{j,\epsilon}(\omega))- V^\epsilon_x*((k^\epsilon*\rho_s^{\epsilon}) \rho_s^{\epsilon}))  ,   V^\epsilon*(\mu_s^{N,\epsilon}(\omega) -\rho_s^{\epsilon}) \Big\rangle_{L^2} \\
    = &\;  \frac{1}{N^2} \sum\limits_{i,j=1}^N \Big\langle  V^\epsilon_x ( \cdot-X_s^{i,\epsilon}(\omega)) (k^\epsilon(X_s^{i,\epsilon}(\omega)-X_s^{j,\epsilon}(\omega))-k^\epsilon(Y_s^{i,\epsilon}(\omega)-Y_s^{j,\epsilon}(\omega)))  ,\\
    &\quad\quad  V^\epsilon*(\mu_s^{N,\epsilon}(\omega) -\rho_s^{\epsilon}) \Big\rangle_{L^2(\R)} \\
    &  + \frac{1}{	N^2} \sum\limits_{i,j=1}^N \Big\langle  V^\epsilon_x ( \cdot-X_s^{i,\epsilon}(\omega))   k^\epsilon(Y_s^{i,\epsilon}(\omega)-Y_s^{j,\epsilon}(\omega))
    - V^\epsilon_x*((k^\epsilon*\rho_s^{\epsilon}) \rho_s^{\epsilon}))  , \\
    &\quad\quad  V^\epsilon*(\mu_s^{N,\epsilon}(\omega) -\rho_s^{\epsilon}) \Big\rangle_{L^2(\R)} \\
    = &\; I_s^1(\omega) + I_s^2(\omega) .
  \end{align*}

  For the first term we obtain
  \begin{align}
   \nonumber |I_s^1(\omega)| =
    &\; \bigg| \frac{1}{N} \sum\limits_{j=1}^N \bigg\langle   \frac{1}{	N} \sum\limits_{i=1}^N  V^\epsilon ( \cdot-X_s^{i,\epsilon}(\omega)) (k^\epsilon(X_s^{i,\epsilon}(\omega)-X_s^{j,\epsilon}(\omega))-k^\epsilon(Y_s^{i,\epsilon}(\omega)-Y_s^{j,\epsilon}(\omega))) ,  \\
    \nonumber  & \qquad V^\epsilon_x*( \mu_s^{N,\epsilon}(\omega) -\rho_s^{\epsilon}) \bigg\rangle_{L^2(\R)}\bigg| \\
    \nonumber  \le & \; \frac{1}{N^2} \sum\limits_{i,j=1}^N \bigg\langle   | (V^\epsilon ( \cdot-X_s^{i,\epsilon}(\omega))| \max\limits_{1\le i \le N} |k^\epsilon(X_s^{i,\epsilon}(\omega)-X_s^{j,\epsilon}(\omega))-k^\epsilon(Y_s^{i,\epsilon}(\omega)-Y_s^{j,\epsilon}(\omega))| ,  \\
    \nonumber  & \;  |V^\epsilon_x*( \mu_s^{N,\epsilon}(\omega) -\rho_s^{\epsilon}| \bigg\rangle_{L^2(\R)}   \\
    \nonumber    \le &  \frac{2 }{N} \sum\limits_{i=1}^N \bigg\langle \norm{k_x^\epsilon}_{L^\infty(\R)} | V^\epsilon ( \cdot-X_s^{i,\epsilon}(\omega))| \max\limits_{1\le i \le N} |X^{i,N}_s -Y_s^{i,N}|  ,  | V^\epsilon_x *(\mu_s^{N,\epsilon}(\omega) -\rho_s^{\epsilon})| \bigg\rangle_{L^2}   \\
    \nonumber  \le &\;    \frac{2 }{N}\sum\limits_{i=1}^N  \langle    N^{-\alpha} \norm{k_x^\epsilon}_{L^\infty(\R)}   | V^\epsilon ( \cdot-X_s^{i,\epsilon}(\omega))|,            |V^\epsilon_x*( \mu_s^{N,\epsilon}(\omega) -\rho_s^{\epsilon}| \rangle_{L^2(\R)}   \\
    \nonumber   \le & \;     \frac{2}{N} \sum\limits_{i=1}^N \int_{\R}  \frac{8\norm{k_x^\epsilon}_{L^\infty(\R)}^2 }{\sigma^2N^{2\alpha}}| V^\epsilon ( y-X_s^{i,\epsilon}(\omega))|^2 \Id y  + \frac{\sigma^2}{32} \int_{\R} |V^\epsilon_x * (\mu_s^{N,\epsilon}(\omega)-\rho_s^{\epsilon})(y)|^2 \Id y  \\
    \label{I1}  \le & \;     \frac{16}{\sigma^2 N^{2\alpha}}  \norm{k_x^\epsilon}_{L^\infty(\R)}^2 \norm{V^\epsilon}_{L^2(\R)}^2 +\frac{\sigma^2}{16} \norm{ V^\epsilon_x * (\mu_s^{N,\epsilon}(\omega)-\rho_s^{\epsilon})}_{L^2(\R)}^2.
  \end{align}
  Here we used integration by parts in the first step, the property of the set \(B_s^\alpha\) in the fourth step. As always, we neglect the last term by absorbing it into the diffusion in our statement.

  We treat the term \(I_s^2(\omega)\) using the law of large numbers property of the second term in \(B_s^\alpha\). For \(\omega \in B_s^\alpha \) we rewrite
  \begin{align}\label{I2}
    | I_s^2(\omega)|
    &= \Big|\frac{1}{	N^2} \sum\limits_{i,j=1}^N \Big\langle
    V^\epsilon_x ( \cdot-X_s^{i,\epsilon}(\omega))   k^\epsilon(Y_s^{i,\epsilon}(\omega)-Y_s^{j,\epsilon}(\omega))
    - V^\epsilon_x*((k^\epsilon*\rho_s^{\epsilon}) \rho_s^{\epsilon}))  , \nonumber  \\
    &\qquad  V^\epsilon*(\mu_s^{N,\epsilon}(\omega) -\rho_s^{\epsilon}) \Big\rangle_{L^2(\R)} \Big|\nonumber \\
    &= \Big|\frac{1}{	N^2} \sum\limits_{i,j=1}^N \Big\langle
    V^\epsilon ( \cdot-X_s^{i,\epsilon}(\omega))   k^\epsilon(Y_s^{i,\epsilon}(\omega)-Y_s^{j,\epsilon}(\omega))
    - V^\epsilon*((k^\epsilon*\rho_s^{\epsilon}) \rho_s^{\epsilon}))  ,  \nonumber \\
    &\qquad V^\epsilon_x*(\mu_s^{N,\epsilon}(\omega) -\rho_s^{\epsilon}) \Big\rangle_{L^2(\R)} \Big|\nonumber \\
    &= \Big| \frac{1}{	N^2} \sum\limits_{i,j=1}^N \Big\langle
    V^\epsilon ( \cdot-X_s^{i,\epsilon}(\omega))  ( k^\epsilon(Y_s^{i,\epsilon}(\omega)-Y_s^{j,\epsilon}(\omega))- (k^\epsilon*\rho_s^{\epsilon})(Y_s^{i,\epsilon}(\omega)) ) \nonumber \\
    & + ( V^\epsilon ( \cdot-X_s^{i,\epsilon}(\omega)) -V^\epsilon( \cdot-Y_s^{i,\epsilon}(\omega))) (k^\epsilon*\rho_s^{\epsilon})(Y_s^{i,\epsilon}(\omega))   \nonumber  \\
    & + V^\epsilon( \cdot-Y_s^{i,\epsilon}(\omega))(k^\epsilon*\rho_s^{\epsilon})(Y_s^{i,\epsilon}(\omega))
    - V^\epsilon*((k^\epsilon*\rho_s^{\epsilon}) \rho_s^{\epsilon}))  , V^\epsilon_x*(\mu_s^{N,\epsilon}(\omega) -\rho_s^{\epsilon})\Big\rangle_{L^2(\R)}\Big| \nonumber \\
    &= |I_s^{21}(\omega)| + |I_s^{22}(\omega)|+ |I_s^{23}(\omega)| .
  \end{align}

  For the first term \(I_s^{21}(\omega)\) we obtain
  \begin{align}\label{I21}
    \nonumber | I_s^{21}(\omega) |&\le \frac{1}{N} \sum\limits_{i=1}^N \bigg\langle
    |V^\epsilon ( \cdot-X_s^{i,\epsilon}(\omega))|  \bigg| \frac{1}{N} \sum\limits_{j=1}^N k^\epsilon(Y_s^{i,\epsilon}(\omega)-Y_s^{j,\epsilon}(\omega))- (k^\epsilon*\rho_s^{\epsilon})(Y_s^{i,\epsilon}(\omega)) \bigg|  , \\
    \nonumber &\quad |V^\epsilon_x*(\mu_s^{N,\epsilon}(\omega) -\rho_s^{\epsilon})| \bigg\rangle_{L^2(\R)} \\
    \nonumber &\le  \frac{1}{N} \sum\limits_{i=1}^N \langle
    N^{-(\alpha+\delta)}|V^\epsilon ( \cdot-X_s^{i,\epsilon}(\omega))|   , |V^\epsilon_x*(\mu_s^{N,\epsilon}(\omega) -\rho_s^{\epsilon})| \rangle_{L^2(\R)} \\
    &\le \frac{4N^{-2(\alpha+\delta)}}{\sigma^2} \norm{V^\epsilon}_{L^2(\R)}^2 +\frac{\sigma^2}{16} \norm{ V^\epsilon_x * (\mu_s^{N,\epsilon}(\omega)-\rho_s^{\epsilon})}_{L^2(\R)}^2,
  \end{align}
  where we used the property of the set \(B_s^\alpha\) in the second step and Young's inequality.

  Using the fact that we are still on the set \(B_s^\alpha\) we obtain for the second term \(I_s^{22}(\omega)\) the following estimate
  \begin{align}\label{I22}
    \nonumber |I_s^{22}(\omega)|
    \le &\; \frac{4\norm{k^\epsilon*\rho_s^{\epsilon}}_{L^\infty(\R)}^2}{N \sigma^2}   \sum\limits_{i=1}^N \int_{\R} |V^\epsilon ( y-X_s^{i,\epsilon}(\omega)) -  V^\epsilon ( y-Y_s^{i,\epsilon}(\omega)) )|^2  \Id y  \\
    \nonumber  &\qquad+ \frac{\sigma^2}{16} \int_{\R} |V^\epsilon_x*(\mu_s^{N,\epsilon}(\omega) -\rho_s^{\epsilon}) (y)|^2 \Id y  \\
    \nonumber  =&\; \frac{4\norm{k^\epsilon*\rho^\epsilon_s}_{L^\infty(\R)}^2}{N \sigma^2}   \sum\limits_{i=1}^N \int_{\R} \bigg|\int\limits_0^1 \frac{\Id}{\Id r}  V^\epsilon(y-Y_s^{i,\epsilon}(\omega) + r(Y_s^{i,\epsilon}(\omega)-X_s^{i,\epsilon}(\omega)))  \Id r \bigg|^2  \Id y \\
    \nonumber   &\qquad + \frac{\sigma^2}{16} \int_{\R} |V^\epsilon_x*(\mu_s^{N,\epsilon}(\omega) -\rho_s^{\epsilon}) (y)|^2 \Id y  \\
    \nonumber   &\le\; \frac{4\norm{k^\epsilon*\rho_s^\epsilon}_{L^\infty(\R)}^2}{ \sigma^2} \max_{1\leq i\leq N}| Y_s^{i,\epsilon}(\omega)-X_s^{i,\epsilon}(\omega)|^2   \\
    \nonumber   &\qquad\qquad\cdot \frac{1}{N}\sum\limits_{i=1}^N  \int\limits_0^1 \int_{\R}  \Big|\frac{\Id }{\Id x}  V^\epsilon(y-Y_s^{i,\epsilon}(\omega) + r(Y_s^{i,\epsilon}(\omega)-X_s^{i,\epsilon}(\omega)))\Big|^2  \Id y   \Id r   \\
    \nonumber   &\qquad + \frac{\sigma^2}{16} \int_{\R} |V^\epsilon_x*(\mu_s^{N,\epsilon}(\omega) -\rho_s^{\epsilon}) (y)|^2 \Id y  \\
    \nonumber    &= \; \frac{4\norm{k^\epsilon*\rho_s^\epsilon}_{L^\infty(\R)}^2}{\sigma^2}\max_{1\leq i\leq N} | Y_s^{i,\epsilon}(\omega)-X_s^{i,\epsilon}(\omega)|^2  \int\limits_0^1 \int_{\R}  |  V^\epsilon_z(z)|^2  \Id z   \Id r  \\
    \nonumber    &\quad+  \frac{\sigma^2}{16} \int_{\R} |V^\epsilon_x*(\mu_s^{N,\epsilon}(\omega) -\rho_s^{\epsilon}) (y)|^2 \Id y  \\
    &\le\; \frac{4\norm{k^\epsilon*\rho_s^\epsilon}_{L^\infty(\R)}^2}{N^{2\alpha} \sigma^2}  \norm{   V^\epsilon_x}_{L^2(\R)}^2  + \frac{\sigma^2}{16} \norm{ V^\epsilon_x*(\mu_s^{N,\epsilon}(\omega) -\rho_s^{\epsilon})}_{L^2(\R)}^2.
  \end{align}
  In the above calculations we used Young's inequality in the first step, Jensen inequality in the second estimate, the property of the set \(B_s^{\alpha}\) in the third estimate.

  In order to estimate the last term \(I_s^{23}(\omega)\) in~\eqref{I2} we use the independence of our mean-field particles \((Y_t^{i,\epsilon}, i = 1, \ldots, N )\). Hence, we can no longer do the estimates pathwise and need to take advantage of the expectation. First, applying Young's inequality we find
  \begin{align*}
     |I_s^{23}(\omega)| & \le \;  \frac{4}{\sigma^2}  \int_{\R} \frac{1}{N^2}
    \bigg|\sum\limits_{i=1}^N V^\epsilon( y-Y_s^{i,\epsilon}(\omega))(k^\epsilon*\rho_s^{\epsilon})(Y_s^{i,\epsilon}(\omega))
    - V^\epsilon*((k^\epsilon*\rho_s^{\epsilon}) \rho_s^{\epsilon}))(y) \bigg|^2 \Id y \\
    &\quad \quad +  \frac{\sigma^2}{16} \norm{ V^\epsilon_x*(\mu_s^{N,\epsilon}(\omega) -\rho_s^{\epsilon})}_{L^2(\R)}^2 .
  \end{align*}
  As always, the last term is going to be absorbed. For the first term, we recall that our statement has an supremum over all \(0 \le t \le T\) and an expectation. Hence, it is enough to estimate
  \begin{align*}
    &\;  \E\bigg( \sup\limits_{0 \le t \le T} \int\limits_0^t \frac{4}{\sigma^2}  \int_{\R} \frac{1}{N^2}
    \bigg|\sum\limits_{i=1}^N V^\epsilon( y-Y_s^{i,\epsilon}(\omega))(k^\epsilon*\rho_s^{\epsilon})(Y_s^{i,\epsilon}(\omega))
    -V^\epsilon*((k^\epsilon*\rho_s^{\epsilon}) \rho_s^{\epsilon}))(y) \bigg|^2 \Id y \bigg)\\
    =&  \; \int\limits_0^T \frac{4}{N^2\sigma^2}  \int_{\R} \E \bigg(
    \bigg|\sum\limits_{i=1}^N V^\epsilon( y-Y_s^{i,\epsilon}(\omega))(k^\epsilon*\rho_s^{\epsilon})(Y_s^{i,\epsilon}(\omega))
    - V^\epsilon*((k^\epsilon*\rho_s^{\epsilon}) \rho_s^{\epsilon}))(y) \bigg|^2  \bigg) \Id y   .
  \end{align*}
  Let us denote for fix \(y \in \R\)
  \begin{equation*}
    Z^{i}_s(\omega) := V^\epsilon( y-Y_s^{i,\epsilon}(\omega))(k^\epsilon*\rho_s^{\epsilon})(Y_s^{i,\epsilon}(\omega))
    - V^\epsilon*((k^\epsilon*\rho_s^{\epsilon}) \rho_s^{\epsilon}))(y) .
  \end{equation*}
  Then we notice that
  \begin{align*}
    \E(Z^{i}_s)
    &= \E(V^\epsilon( y-Y_s^{i,\epsilon})(k^\epsilon*\rho_s^{\epsilon})(Y_s^{i,\epsilon}) ) - V^\epsilon*((k^\epsilon*\rho_s^{\epsilon}) \rho_s^{\epsilon}))(y) \\
    &= \int_{\R} V^\epsilon( y-z)(k^\epsilon*\rho_s^{\epsilon})(z ) \rho_s^{\epsilon}(z) \Id z
    - V^\epsilon*((k^\epsilon*\rho_s^{\epsilon}) \rho_s^{\epsilon}))(y) = 0 .
  \end{align*}
  Furthermore, we have the random variables \((Z_s^{i}, i =1, \ldots, N)\) are pairwise independent. Hence, if \(i \neq j \) we find
  \begin{equation*}
    \E(Z^{i}_s Z^{j}_s)  = \E(Z^{i}_s)  \E(Z^{j}_s)  = 0.
  \end{equation*}

  We notice that we have
  \begin{align*}
    &\; \E \bigg(
    \bigg|\sum\limits_{i=1}^N V^\epsilon( y-Y_s^{i,\epsilon}(\omega))(k^\epsilon*\rho_s^{\epsilon})(Y_s^{i,\epsilon}(\omega))
    - V^\epsilon*((k^\epsilon*\rho_s^{\epsilon}) \rho_s^{\epsilon}))(y) \bigg|^2  \bigg)\\
    &\quad=\;  \E \bigg(
    \bigg|\sum\limits_{i=1}^N Z_s^{i} \bigg|^2  \bigg)
    = \sum\limits_{i=1}^N \E(|Z^{i}_s|^2 )  .
  \end{align*}
  On the other hand by using the trivial inequality \((a+b)^2 \le 2(a^2+b^2)\) and Young's inequality for convolution we obtain
  \begin{align*}
    \int_{\R} \E(|Z^{i}_s|^2 )
    &\le 2 \E \bigg( \int_{\R} |V^\epsilon( y-Y_s^{i,\epsilon}(\omega))(k^\epsilon*\rho_s^{\epsilon})(Y_s^{i,\epsilon}(\omega)) |^2
    +|V^\epsilon*((k^\epsilon*\rho_s^{\epsilon}) \rho_s^{\epsilon}))(y)|^2 \Id y \bigg) \\
    &\le 2 \norm{k^\epsilon*\rho_s^{\epsilon}}_{L^\infty(\R)}^2 \norm{V^\epsilon}_{L^2(\R)}^2 + 2 \norm{V^\epsilon}_{L^2(\R)}^2 \norm{(k^\epsilon*\rho_s^{\epsilon}) \rho_s^{\epsilon}}_{L^1(\R)}^2 \\
    &= 4\norm{k^\epsilon*\rho_s^{\epsilon}}_{L^\infty(\R)}^2  \norm{V^\epsilon}_{L^2(\R)}^2 .
  \end{align*}
  Hence, the estimate for $I^{23}$ follows by the previous law of large numbers argument and is obtained in the following
  \begin{align}\label{I23}
    & \; \E\bigg( \sup\limits_{0 \le t \le T} \int\limits_0^t |I^{23}_s(\omega)| \Id s \bigg)\\
    \nonumber &\quad\le  \; \frac{\sigma^2}{16} \int\limits_0^t \norm{ V^\epsilon_x*(\mu_s^{N,\epsilon}(\omega) -\rho_s^{\epsilon})}_{L^2(\R)}^2 + \frac{16  \norm{V^\epsilon}_{L^2(\R)}^2 }{N\sigma^2}   \int\limits_0^T \norm{k^\epsilon*\rho_s^{\epsilon}}_{L^\infty(\R)}^2  \Id s  .
  \end{align}

  By combining the estimates \eqref{I21}\eqref{I22}\eqref{I23} with \eqref{I2} and \eqref{I1} we obtain the estimate on the set $B^\alpha_s$
  \begin{align}\label{I2est}
    &\; \E\bigg( \sup\limits_{0 \le t \le T} \int\limits_0^t (|I^1_s(\omega)|+|I^2_s(\omega)|)\indicator{(B_s^\alpha)}\Id s\bigg) \\
    \nonumber  \le  &\;   - \frac{3\sigma^2}{16} \int\limits_0^T \norm{ V^\epsilon_x*(\mu_s^{N,\epsilon}(\omega) -\rho_s^{\epsilon})}_{L^2(\R)}^2\Id s  +\frac{16T\norm{k_x^\epsilon}_{L^\infty(\R)} \norm{V^\epsilon}_{L^2(\R)}^2}{\sigma^2 N^{2\alpha}}   +\frac{4T\norm{V^\epsilon}_{L^2(\R)}^2}{\sigma^2N^{2(\alpha+\delta)}} \\
    & \nonumber +\Big(\frac{4\norm{   V^\epsilon_x}_{L^2(\R)}^2}{N^{2\alpha} \sigma^2}  +  \frac{16  \norm{V^\epsilon}_{L^2(\R)}^2 }{N\sigma^2} \Big)  \int\limits_0^T \norm{k^\epsilon*\rho_s^{\epsilon}}_{L^\infty(\R)}^2  \Id s
  \end{align}
  It remains to obtain an estimate on the complement of \(B_s^\alpha\).

  \textit{On the set \((B_s^\alpha)^{\mathrm{c}}\):} Applying Young's inequality, multiple Hölder's inequalities, the fact that \(\P((B_s^\alpha)^{\mathrm{c}}) \le C(\gamma)N^{-\gamma}\), we obtain
  \begin{align*}
    & \E \bigg( \sup\limits_{0\le t \le T} \int\limits_0^t \indicator{(B_s^\alpha)^{\mathrm{c}}} \Big|\frac{1}{	N^2} \sum\limits_{i,j=1}^N \Big\langle  V^\epsilon_x ( \cdot-X_s^{i,\epsilon}(\omega)) k^\epsilon(X_s^{i,\epsilon}(\omega)-X_s^{j,\epsilon}(\omega))- V^\epsilon_x*((k^\epsilon*\rho_s^{\epsilon}) \rho_s^{\epsilon}))  , \\
    & \quad\;   V^\epsilon*(\mu_s^{N,\epsilon}(\omega) -\rho_s^{\epsilon})\Big\rangle_{L^2(\R)}\Big| \Id s  \bigg) \\
    \le &  \;   \frac{1}{N^2} \sum\limits_{i,j=1}^N \E \bigg( \sup\limits_{0\le t \le T} \int\limits_0^t \indicator{(B_s^\alpha)^{\mathrm{c}}} \Big|\Big\langle  V^\epsilon ( \cdot-X_s^{i,\epsilon}(\omega)) k^\epsilon(X_s^{i,\epsilon}(\omega)-X_s^{j,\epsilon}(\omega))- V^\epsilon*((k^\epsilon*\rho_s^{\epsilon}) \rho_s^{\epsilon})) , \\
    & \quad V^\epsilon_x*(\mu_s^{N,\epsilon}(\omega) -\rho_s^{\epsilon}) \Big\rangle_{L^2(\R)}\Big| \Id s  \bigg) \\
    \le &  \; \frac{1}{N^2} \sum\limits_{i,j=1}^N
    \E\bigg( \int\limits_0^T \indicator{(B_s^\alpha)^{\mathrm{c}}} \Big( \norm{ V^\epsilon ( \cdot-X_s^{i,\epsilon}(\omega)) k^\epsilon(X_s^{i,\epsilon}(\omega)-X_s^{j,\epsilon}(\omega))}_{L^2(\R)}^2 \\
    &\quad + \norm{ V^\epsilon*((k^\epsilon*\rho_s^{\epsilon}) \rho_s^{\epsilon})) }_{L^2(\R)}^2 \Big)\Id s  \bigg)
    +  \frac{1}{2} \E \bigg(  \int\limits_0^T \indicator{(B_s^\alpha)^{\mathrm{c}}} \norm{ V^\epsilon_x*(\mu_s^{N,\epsilon}(\omega) -\rho_s^{\epsilon})}_{L^2(\R)}^2 \Id s \bigg) \\
    \le &  \;  \frac{1}{ N} \sum\limits_{i=1}^N
    \E\bigg( \int\limits_0^T  \indicator{(B_s^\alpha)^{\mathrm{c}}} \Big( \norm{ V^\epsilon(\cdot-X_s^{i,\epsilon}(\omega)) }_{L^2(\R)}^2 \norm{k^\epsilon}_{L^\infty}^2 \\
    &\quad+ \norm{ V^\epsilon}_{L^2(\R)}^2 \norm{k^\epsilon*\rho_s^{\epsilon}}_{L^\infty(\R)}^2 \norm{ \rho_s^{\epsilon} }_{L^1(\R)}^2 \Big) \Id s \bigg)
    +  2 \int\limits_0^T \E \bigg( \indicator{(B_s^\alpha)^{\mathrm{c}}} \norm{ V^\epsilon_x}_{L^2(\R)}^2 \bigg) \Id s  \\
    \le & \; 2 \int\limits_0^T   \P \Big((B_s^\alpha)^{\mathrm{c}} \Big) \Big(\norm{ V^\epsilon }_{L^2(\R)}^2  \norm{k^\epsilon}_{L^\infty(\R)}^2 +  \norm{ V^\epsilon_x}_{L^2(\R)}^2 \Big) \Id s  \\
    \le & \;  \frac{C(\gamma)T}{N^\gamma}\Big(\norm{ V^\epsilon }_{L^2(\R)}^2  \norm{k^\epsilon}_{L^\infty(\R)}^2 +  \norm{ V^\epsilon_x}_{L^2(\R)}^2 \Big) .
  \end{align*}
  Combined with the estimate on the set $B^\alpha_s$ , we obtained the result.
\end{proof}

\begin{lemma}[Stochastic Remaining Term Inequality]
  Let the assumptions of Theorem~\ref{theorem: emp_measure_l2_estimate} hold true. Then
  \begin{align}\label{eq: main_theorem_GNS_absorbation}
    &  \E \bigg( \sup\limits_{0 \le t \le T}  \int\limits_{0}^t \bigg| \bigg \langle  \frac{1}{N^2} \sum\limits_{i,j=1}^N  V^\epsilon_{x}(\cdot-X_s^{i,\epsilon}) k^\epsilon(X_s^{i,\epsilon}-X_s^{j,\epsilon}) - V^\epsilon_x*((k^\epsilon*\rho_s^{\epsilon}) \rho_s^{\epsilon}) , \nonumber \\
    & \quad\quad \frac{\sigma}{N}\sum\limits_{l=1}^N  - \int\limits_0^s V^\epsilon_x(\cdot-X_u^{l}) \Id B_u^{l}  \bigg\rangle_{L^2(\R)}\bigg| \Id s  \bigg)  \nonumber \\
   \le & \;  \frac{2 \sigma T^\frac{3}{2} C_{\mathrm{BDG}}^{\frac{1}{2}}}{N^{\alpha+\frac12} }   \norm{V^\epsilon_x}_{L^2(\R)}^2  \norm{k^\epsilon_x}_{L^\infty(\R)}+\sigma\frac{C_{\mathrm{BDG}}^{\frac{1}{2}}T^\frac{3}{2} }{N^{\alpha+\delta+\frac12}}\norm{V^\epsilon_x}_{L^2(\R)}^2\\
    \nonumber &+ \Big(\sigma  \frac{C_{\mathrm{BDG}  }^{\frac{1}{2}}\norm{V^\epsilon_{xx} }_{L^2(\R)}}{N^{\alpha+\frac12}}  \norm{V^\epsilon_x}_{L^2(\R)}  +\sigma \frac{2 C_{\mathrm{BDG}}^{\frac{1}{2}}\norm{V^\epsilon_x}_{L^2(\R)}^2 }{N}\Big)\int\limits_0^T \norm{k^\epsilon*\rho_s^\epsilon}_{L^\infty(\R)} s^{\frac{1}{2}} \Id s.\nonumber \\
    & + \frac{2 C(\gamma) C_{\mathrm{BDG} }^{\frac{1}{2}}  \sigma}{N^{\frac12+\gamma}} \norm{V^\epsilon_x}_{L^2(\R)}^2\Big(\norm{k^\epsilon}_{L^\infty(\R)} \frac{2}{3} T^\frac{3}{2}+ \int\limits_0^T \norm{k^\epsilon*\rho_s^\epsilon}_{L^\infty(\R)} s^{\frac{1}{2}} \Id s \Big).  \nonumber
  \end{align}
\end{lemma}

\begin{proof}
  We carry out a similar strategy as in the previous Lemma~\ref{lemma: l2_theorem_absobation_lemma}. Again, we want to split \(\Omega\) into a good and bad set. Remember the definition of set $B_s^{\alpha}$ in \eqref{Balpha}, we do the estimates on $B_s^{\alpha}$ and its complement $(B_s^{\alpha})^c$ separately.

  \textit{On the set \(B_s^\alpha\):} Let \(\omega \in B_s^\alpha\), then we insert the i.i.d. process $\mathbf{Y}^{N,\epsilon}$ and split the estimate further into two terms
  \begin{align*}
    &\;  \frac{1}{	N^2} \sum\limits_{i,j=1}^N  \bigg\langle  V^\epsilon_x ( \cdot-X_s^{i,\epsilon}(\omega)) k^\epsilon(X_s^{i,\epsilon}(\omega)-X_s^{j,\epsilon}(\omega))- V^\epsilon_x*((k^\epsilon*\rho_s^{\epsilon}) \rho_s^{\epsilon}))  , \\
    &\quad\quad  \frac{\sigma}{N}\sum\limits_{l=1}^N - \int\limits_0^s V^\epsilon_x(\cdot-X_u^{l}) \Id B_u^{l}  \bigg\rangle_{L^2(\R)} \\
    =&\;  \frac{1}{N^2} \sum\limits_{i,j=1}^N \bigg\langle  V^\epsilon_x ( \cdot-X_s^{i,\epsilon}(\omega)) (k^\epsilon(X_s^{i,\epsilon}(\omega)-X_s^{j,\epsilon}(\omega))-k^\epsilon(Y_s^{i,\epsilon}(\omega)-Y_s^{j,\epsilon}(\omega)))  , \\
    &\quad\quad  \frac{\sigma}{N}\sum\limits_{l=1}^N  - \int\limits_0^s  V^\epsilon_x(\cdot-X_u^{l}) \Id B_u^{l}  \bigg\rangle_{L^2(\R)} \\
    & + \frac{1}{	N^2} \sum\limits_{i,j=1}^N \bigg\langle  V^\epsilon_x ( \cdot-X_s^{i,\epsilon}(\omega))   k^\epsilon(Y_s^{i,\epsilon}(\omega)-Y_s^{j,\epsilon}(\omega))
    - V^\epsilon_x*((k^\epsilon*\rho_s^{\epsilon}) \rho_s^{\epsilon}))  , \\
    &\quad\quad  \frac{\sigma}{N}\sum\limits_{l=1}^N - \int\limits_0^s V^\epsilon_x(\cdot-X_u^{l}) \Id B_u^{l}  \bigg\rangle_{L^2(\R)} \\
    =&\;   II_s^1(\omega) + II_s^2(\omega)  .
  \end{align*}
  Further, utilizing the property of the set \(B_s^\alpha\) and the Burkholder--Davis--Gundy inequality we obtain
  \begin{align}\label{II1}
    &\;  \E \bigg( \sup\limits_{0 \le t \le T} \int\limits_0^t |II_s^1(\omega)|\indicator{(B_s^\alpha)} \Id s \bigg) \\
    \nonumber  \le &\;   \E \bigg(  \int\limits_0^T \frac{1}{N^2} \sum\limits_{i,j=1}^N \bigg\langle  \Big|V^\epsilon_x ( \cdot-X_s^{i,\epsilon}(\omega)) (k^\epsilon(X_s^{i,\epsilon}(\omega)-X_s^{j,\epsilon}(\omega))-k^\epsilon(Y_s^{i,\epsilon}(\omega)-Y_s^{j,\epsilon}(\omega)))\Big|  ,
    \\
    \nonumber   & \quad \bigg|\frac{\sigma}{N}\sum\limits_{l=1}^N  \int\limits_0^s V^\epsilon_x(\cdot-X_u^{l}) \Id B_u^{l} \bigg| \bigg\rangle_{L^2(\R)} \indicator{(B_s^\alpha)}\Id s \bigg) \\
    \nonumber  \le &\;  2 \E \bigg(  \int\limits_0^T \frac{1}{N} \sum\limits_{i=1}^N \bigg\langle  \norm{k^\epsilon_x}_{L^\infty(\R)} |V^\epsilon_x ( \cdot-X_s^{i,\epsilon}(\omega))| \max\limits_{1 \le i \le N} |X_s^{i,\epsilon}(\omega)-Y_s^{i,\epsilon}(\omega)|,\\
    \nonumber   & \quad \bigg|\frac{\sigma}{N}\sum\limits_{l=1}^N  \int\limits_0^s V^\epsilon_x(\cdot-X_u^{l}) \Id B_u^{l} \bigg| \bigg\rangle_{L^2(\R)} \indicator{(B_s^\alpha)}\Id s \bigg) \\
    \nonumber  \le &\;   \frac{2 \sigma  \norm{k^\epsilon_x}_{L^\infty(\R)}}{N^{\alpha} }  \frac{1}{N} \sum\limits_{i=1}^N \int\limits_0^T  \E \bigg(  \bigg\langle   |V^\epsilon_x ( \cdot-X_s^{i,\epsilon}(\omega))| ,  \bigg|\frac{1}{N}\sum\limits_{l=1}^N  \int\limits_0^s V^\epsilon_x(\cdot-X_u^{l}) \Id B_u^{l} \bigg| \bigg\rangle_{L^2(\R)} \bigg) \Id s  \\
    \nonumber
    =&\;  \frac{2 \sigma  \norm{k^\epsilon_x}_{L^\infty(\R)}}{N^{\alpha} }  C_{\mathrm{BDG}}^{\frac{1}{2}}  \norm{V^\epsilon_x}_{L^2(\R)}^2 T^\frac{3}{2} \frac{1}{N^\frac{1}{2}}=\frac{2 \sigma T^\frac{3}{2} C_{\mathrm{BDG}}^{\frac{1}{2}}}{N^{\alpha+\frac12} }   \norm{V^\epsilon_x}_{L^2(\R)}^2  \norm{k^\epsilon_x}_{L^\infty(\R)},
  \end{align}
  where we have used the estimate
  \begin{align}\label{estBDG}
    &\; \frac{1}{N} \sum\limits_{i=1}^N \int\limits_0^T  \E \bigg(  \bigg\langle   |V^\epsilon_x ( \cdot-X_s^{i,\epsilon}(\omega))| ,  \bigg|\frac{1}{N}\sum\limits_{l=1}^N  \int\limits_0^s V^\epsilon_x(\cdot-X_u^{l}) \Id B_u^{l} \bigg| \bigg\rangle_{L^2(\R)} \bigg) \Id s  \\
    \nonumber\le &\; \frac{1}{N} \sum\limits_{i=1}^N \int\limits_0^T \int_{\R} \E (   |V^\epsilon_x ( y-X_s^{i,\epsilon}(\omega))|^2 ) ^{\frac{1}{2}} \E\bigg(  \bigg|\frac{1}{N}\sum\limits_{l=1}^N  \int\limits_0^s V^\epsilon_x(y-X_u^{l}) \Id B_u^{l} \bigg|^2  \bigg)^{\frac{1}{2}}  \Id y \Id s  \\
    \nonumber\le &\; \frac{1}{N} \sum\limits_{i=1}^N \int\limits_0^T  \int_{\R} \E (   |V^\epsilon_x ( y-X_s^{i,\epsilon}(\omega))|^2 ) ^{\frac{1}{2}} \frac{C_{\mathrm{BDG}}^{\frac{1}{2}} }{N} \E\bigg( \sum\limits_{l=1}^N  \int\limits_0^s |V^\epsilon_x(y-X_u^{l}) |^2 \Id u  \bigg)^{\frac{1}{2}}  \Id y \Id s  \\
    \nonumber\le  &\; C_{\mathrm{BDG}}^{\frac{1}{2}}   \frac{1}{N^2} \sum\limits_{i=1}^N   \int\limits_0^T  \bigg(\int_{\R} \E (   |V^\epsilon_x ( y-X_s^{i,\epsilon}(\omega))|^2 ) \Id y \bigg)^{\frac{1}{2}} \\
    \nonumber&\quad \cdot \bigg(\int_{\R} \E\bigg( \sum\limits_{l=1}^N  \int\limits_0^s |V^\epsilon_x(y-X_u^{l}) |^2 \Id u  \bigg) \Id y \bigg)^{\frac{1}{2}}  \Id s  \\
    \nonumber = &\;   C_{\mathrm{BDG}}^{\frac{1}{2}}  \norm{V^\epsilon_x}_{L^2(\R)}^2 T^\frac{3}{2} \frac{1}{N^\frac{1}{2}}.
  \end{align}
  This completes the estimate of \(II_s^1(\omega)\) on the set \(B_s^\alpha\). Next, for \(\omega \in B_s^\alpha \) we rewrite \(II_s^2(\omega)\) in the following way
  \begin{align*}
    &\; II_s^2(\omega) \\
    =&\; \frac{1}{	N^2} \sum\limits_{i,j=1}^N \bigg\langle
    V^\epsilon_x ( \cdot-X_s^{i,\epsilon}(\omega))   k^\epsilon(Y_s^{i,\epsilon}(\omega)-Y_s^{j,\epsilon}(\omega))
    - V^\epsilon_x*((k^\epsilon*\rho_s^{\epsilon}) \rho_s^{\epsilon}))  ,\\
    &\quad  \frac{\sigma}{N}\sum\limits_{l=1}^N -  \int\limits_0^s V^\epsilon_x(\cdot-X_u^{l}) \Id B_u^{l}  \bigg\rangle_{L^2(\R)} \nonumber \\
    =& \; - \frac{\sigma}{	N^2} \sum\limits_{i,j=1}^N \bigg\langle
    V^\epsilon_x( \cdot-X_s^{i,\epsilon}(\omega))  ( k^\epsilon(Y_s^{i,\epsilon}(\omega)-Y_s^{j,\epsilon}(\omega))- (k^\epsilon*\rho_s^{\epsilon})(Y_s^{i,\epsilon}(\omega)) ) \nonumber \\
    &\quad + ( V^\epsilon_x ( \cdot-X_s^{i,\epsilon}(\omega)) -V^\epsilon_x( \cdot-Y_s^{i,\epsilon}(\omega))) (k^\epsilon*\rho_s^{\epsilon})(Y_s^{i,\epsilon}(\omega))   \nonumber  \\
    &\quad + V^\epsilon_x( \cdot-Y_s^{i,\epsilon}(\omega))(k^\epsilon*\rho_s^{\epsilon})(Y_s^{i,\epsilon}(\omega))
    - V^\epsilon_x*((k^\epsilon*\rho_s^{\epsilon}) \rho_s^{\epsilon}))  , \frac{1}{N}\sum\limits_{l=1}^N  \int\limits_0^s V^\epsilon_x(\cdot-X_u^{l}) \Id B_u^{l}  \bigg\rangle_{L^2(\R)} \nonumber \\
    =&  \; \sigma (  II_s^{21}(\omega) + II_s^{22}(\omega)
    + II_s^{23}(\omega) ).
  \end{align*}
  For the first term  \(II_s^{21}(\omega)\), applying the the property of the set \(B_s^\alpha\), we find with the help of the estimate \eqref{estBDG} for the stochastic term that
  \begin{align}\label{II21}
    &\;  \E\bigg( \sup\limits_{0 \le t \le T} \int\limits_0^t
    |II_s^{21}(\omega)| \Id s \bigg)  \\
    \nonumber \le &  \;  \E\bigg( \sup\limits_{0 \le t \le T} \int\limits_0^t  \frac{1}{N} \sum\limits_{i=1}^N  \bigg\langle
    |V^\epsilon_x ( \cdot-X_s^{i,\epsilon}(\omega))|  \bigg| \frac{1}{N} \sum\limits_{j=1}^N k^\epsilon(Y_s^{i,\epsilon}(\omega)-Y_s^{j,\epsilon}(\omega))- (k^\epsilon*\rho_s^{\epsilon})(Y_s^{i,\epsilon}(\omega)) \bigg|  , \\
    \nonumber & \quad\; \bigg|\frac{1}{N}\sum\limits_{l=1}^N  \int\limits_0^s V^\epsilon_x(\cdot-X_u^{l}) \Id B_u^{l} \bigg|  \bigg\rangle_{L^2(\R)} \Id s \bigg) \\
     \nonumber \le & \;  N^{-(\alpha+\delta)} \int\limits_0^T   \frac{1}{N} \sum\limits_{i=1}^N  \int_{\R} \E\bigg(
    \bigg|V^\epsilon_x ( y-X_s^{i,\epsilon}(\omega))  \frac{1}{N}\sum\limits_{l=1}^N  \int\limits_0^s V^\epsilon_x(y-X_u^{l}) \Id B_u^{l} \bigg|   \bigg) \Id  y  \Id s \\
   \nonumber \le & \;  N^{-(\alpha+\delta)} C_{\mathrm{BDG}}^{\frac{1}{2}}  \norm{V^\epsilon_x}_{L^2(\R)}^2 T^\frac{3}{2} \frac{1}{N^\frac{1}{2}}=\frac{C_{\mathrm{BDG}}^{\frac{1}{2}}T^\frac{3}{2} }{N^{\alpha+\delta+\frac12}}\norm{V^\epsilon_x}_{L^2(\R)}^2,
  \end{align}
  where we used Fubini's Theorem in the second step. 
  For the term \(II_s^{22}(\omega)\) compute
  \begin{align*}
    &\;  |II_s^{22}(\omega)|  \\
    \le &\;  \frac{1}{	N} \sum\limits_{i=1}^N \bigg\langle | V^\epsilon_x ( \cdot-X_s^{i,\epsilon}(\omega)) -  V^\epsilon_x ( \cdot-Y_s^{i,\epsilon}(\omega)) ) (k^\epsilon*\rho_s^{\epsilon})(Y_s^{i,\epsilon}(\omega)|  , \\
    &\quad\quad   \bigg|\frac{1}{N}\sum\limits_{l=1}^N  \int\limits_0^s V^\epsilon_x(\cdot-X_u^{l}) \Id B_u^{l} \bigg|  \bigg\rangle_{L^2(\R)} \\
    \le & \;  \frac{1}{	N} \sum\limits_{i=1}^N \bigg\langle \bigg| \int\limits_0^1 \frac{\d}{\d x} V^\epsilon_{x}\Big(y- Y_s^{i,\epsilon}(\omega)- r(X_s^{i,\epsilon}(\omega)- Y_s^{i,\epsilon}(\omega))\Big) (Y_s^{i,\epsilon}(\omega) -X_s^{i,\epsilon}(\omega)) \Id r \bigg|, \\
    &\quad \;    \bigg|\frac{1}{N}\sum\limits_{l=1}^N  \int\limits_0^s V^\epsilon_x(\cdot-X_u^{l}) \Id B_u^{l} \bigg|  \bigg\rangle_{L^2(\R)} \norm{k^\epsilon*\rho_s^{\epsilon}}_{L^\infty(\R)} \\
    \le & \; \norm{k^\epsilon*\rho_s^\epsilon}_{L^\infty(\R)}  N^{-\alpha} \frac{1}{N} \sum\limits_{i=1}^N \bigg\langle   \int\limits_0^1  \Big|\frac{\d}{\d x} V^\epsilon_{x}\Big(y- Y_s^{i,\epsilon}(\omega)- r(X_s^{i,\epsilon}(\omega)- Y_s^{i,\epsilon}(\omega))\Big)\Big|  \Id r ,\\
    &\quad \;  \bigg| \frac{1}{N}\sum\limits_{l=1}^N  \int\limits_0^s V^\epsilon_x(\cdot-X_u^{l}) \Id B_u^{l} \bigg| \bigg\rangle_{L^2(\R)}  \\
    \le & \; \norm{k^\epsilon*\rho_s^\epsilon}_{L^\infty(\R)}  N^{-\alpha} \frac{1}{N} \sum\limits_{i=1}^N \bigg\|\int\limits_0^1 \Big|\frac{\d}{\d x} V^\epsilon_{x}\Big(y- Y_s^{i,\epsilon}(\omega)- r(X_s^{i,\epsilon}(\omega)- Y_s^{i,\epsilon}(\omega))\Big)\Big| \Id r \bigg\|_{L^2(\R)} \\
    &\quad \; \bigg\| \frac{1}{N}\sum\limits_{l=1}^N  \int\limits_0^s V^\epsilon_x(\cdot-X_u^{l}) \Id B_u^{l} \bigg\|_{L^2(\R)} \\
    \le & \; \norm{k^\epsilon*\rho_s^\epsilon}_{L^\infty(\R)}  N^{-\alpha}  \norm{V^\epsilon_{xx} }_{L^2(\R)}
    \bigg\| \frac{1}{N}\sum\limits_{l=1}^N  \int\limits_0^s V^\epsilon_x(\cdot-X_u^{l}) \Id B_u^{l} \bigg\|_{L^2(\R)} ,  
  \end{align*}
  where we utilized the property of \(B_s^\alpha\) in the third step, followed by the application of H\"older's inequality and Minkowski's inequality. Consequently, applying the Burkholder--Davis--Gundy inequality we obtain
  \begin{align}\label{II22}
    &\;  \E\bigg( \sup\limits_{0 \le t \le T} \int\limits_0^t
    |II_s^{22}(\omega)| \Id s \bigg)  \\
    \nonumber
    \le &\;    N^{-\alpha}  \norm{V^\epsilon_{xx} }_{L^2(\R)} \E\bigg( \sup\limits_{0 \le t  \le T} \int\limits_0^t   \norm{k^\epsilon*\rho_s^\epsilon}_{L^\infty(\R)} \bigg\|\frac{1}{N}\sum\limits_{l=1}^N  \int\limits_0^s V^\epsilon_x(\cdot-X_u^{l}) \Id B_u^{l}\bigg\|_{L^2(\R)} \Id s \bigg) \\
    \nonumber
    \le &\;   N^{-\alpha-1}  \norm{V^\epsilon_{xx} }_{L^2(\R)} \int\limits_0^T  \norm{k^\epsilon*\rho_s^\epsilon}_{L^\infty(\R)}  \bigg( \int_{\R}  \E\bigg( \bigg|\sum\limits_{l=1}^N  \int\limits_0^s V^\epsilon_x(y-X_u^{l}) \Id B_u^{l} \bigg|^2  \bigg)  \Id y \bigg)^{\frac{1}{2}} \Id s \\
    \nonumber
    \le &\;  C_{\mathrm{BDG} }^{\frac{1}{2}} N^{-\alpha-1}  \norm{V^\epsilon_{xx} }_{L^2(\R)} \int\limits_0^T  \norm{k^\epsilon*\rho_s^\epsilon}_{L^\infty(\R)} \bigg( \int_{\R}  \E\bigg( \sum\limits_{l=1}^N   \int\limits_0^s |V^\epsilon_x(y-X_u^{l})|^2 \Id u   \bigg) \Id y  \bigg)^{\frac{1}{2}} \Id s \\
    \nonumber
    \le &\;   \frac{C_{\mathrm{BDG}  }^{\frac{1}{2}}\norm{V^\epsilon_{xx} }_{L^2(\R)}}{N^{\alpha+\frac12}}  \norm{V^\epsilon_x}_{L^2(\R)} \int\limits_{0}^{T}  \norm{k^\epsilon*\rho_s^\epsilon}_{L^\infty(\R)} s^{\frac{1}{2}} \Id s    .
  \end{align}

  For \(II^{23}_s(\omega)\) we use again the Burkholder--Davis--Gundy inequality to estimate the stochastic integral and by the law of large number argument similar to the term $I^{23}$ in Lemma~\ref{lemma: l2_theorem_absobation_lemma}, and obtain
  \begin{align}\label{II23}
    &\;  \E\bigg( \sup\limits_{0\le t \le T } \int\limits_0^t   |II^{23}_s(\omega)|\Id s \bigg)  \\
    \nonumber
     \le&\;     \int\limits_0^T \int_{\R} \E\bigg(  \bigg|  \frac{1}{N}  \sum\limits_{i=1}^N   V^\epsilon_x( y-Y_s^{i,\epsilon}(\omega))(k^\epsilon*\rho_s^{\epsilon})(Y_s^{i,\epsilon}(\omega))
    - V^\epsilon_x*((k^\epsilon*\rho_s^{\epsilon}) \rho_s^{\epsilon}))(y) \bigg|  , \\
    \nonumber&\quad\quad  \bigg| \frac{1}{N}\sum\limits_{l=1}^N  \int\limits_0^s V^\epsilon_x(y-X_u^{l}) \Id B_u^{l}\bigg|   \bigg)  \Id y \Id s \\
    \nonumber
    \le &  \;   \int\limits_0^T \bigg( \int_{\R} \E\bigg( \bigg|  \frac{1}{N}  \sum\limits_{i=1}^N  V^\epsilon_x( y-Y_s^{i,\epsilon}(\omega))(k^\epsilon*\rho_s^{\epsilon})(Y_s^{i,\epsilon}(\omega))
    - V^\epsilon_x*((k^\epsilon*\rho_s^{\epsilon}) \rho_s^{\epsilon}))(y) \bigg|^2 \bigg)  \Id y  \bigg)^{\frac{1}{2}} \\
    \nonumber&\quad\quad \cdot \bigg( \int_{\R} \E \bigg( \bigg| \frac{1}{N}\sum\limits_{l=1}^N  \int\limits_0^s V^\epsilon_x(y-X_u^{l}) \Id B_u^{l} \bigg|^2 \bigg)  \Id y \bigg)^{\frac{1}{2}} \Id s \\
    \nonumber
    \le &  \;  2  C_{\mathrm{BDG}}^{\frac{1}{2}} \int\limits_0^T   N^{-1/2}  \norm{k^\epsilon*\rho_s^\epsilon}_{L^\infty(\R)} \norm{V^\epsilon_x}_{L^2(\R)} s^{\frac{1}{2}} N^{-1/2} \norm{V^\epsilon_x}_{L^2(\R)} \Id s  \\
    \nonumber
    \le & \;\frac{2 C_{\mathrm{BDG}}^{\frac{1}{2}}\norm{V^\epsilon_x}_{L^2(\R)}^2 }{N}\int\limits_0^T \norm{k^\epsilon*\rho_s^\epsilon}_{L^\infty(\R)} s^{\frac{1}{2}} \Id s .
  \end{align}
  This completes the estimate on the set \(B_s^\alpha\), namely
  \begin{align}\label{IIBalpha}
    & \; \E \bigg( \sup\limits_{0 \le t \le T} \int\limits_0^t (|II_s^1(\omega)|+|II_s^2(\omega)|)\indicator{(B_s^\alpha)} \Id s \bigg) \\
    \nonumber  \le &\;   \frac{2 \sigma T^\frac{3}{2} C_{\mathrm{BDG}}^{\frac{1}{2}}}{N^{\alpha+\frac12} }   \norm{V^\epsilon_x}_{L^2(\R)}^2  \norm{k^\epsilon_x}_{L^\infty(\R)}+\sigma\frac{C_{\mathrm{BDG}}^{\frac{1}{2}}T^\frac{3}{2} }{N^{\alpha+\delta+\frac12}}\norm{V^\epsilon_x}_{L^2(\R)}^2\\
    \nonumber &+ \Big(\sigma  \frac{C_{\mathrm{BDG}  }^{\frac{1}{2}}\norm{V^\epsilon_{xx} }_{L^2(\R)}}{N^{\alpha+\frac12}}  \norm{V^\epsilon_x}_{L^2(\R)}  +\sigma \frac{2 C_{\mathrm{BDG}}^{\frac{1}{2}}\norm{V^\epsilon_x}_{L^2(\R)}^2 }{N}\Big)\int\limits_0^T \norm{k^\epsilon*\rho_s^\epsilon}_{L^\infty(\R)} s^{\frac{1}{2}} \Id s.
  \end{align}

  \textit{On the set \((B_s^\alpha)^{\mathrm{c}}\):} Using \(\P( (B_s^\alpha)^{\mathrm{c}})\le C(\gamma) N^{-\gamma}\) for all \(\gamma > 0\) by Assumption~\ref{ass: law_of_large_numbers}, the Burkholder--Davis--Gundy inequality, H{\"o}lder's inequality, we obtain
  \begin{align*}
    &\; \E \bigg( \sup\limits_{0 \le t \le T} \int\limits_0^t  \bigg|\bigg \langle  \frac{1}{N^2} \sum\limits_{i,j=1}^N  V^\epsilon_{x}(\cdot-X_s^{i,\epsilon}) k^\epsilon(X_s^{i,\epsilon}-X_s^{j,\epsilon}) - V^\epsilon_x*((k^\epsilon*\rho_s^{\epsilon}) \rho_s^{\epsilon})  , \\
    & \quad\quad \frac{\sigma}{N}\sum\limits_{l=1}^N - \int\limits_0^s V^\epsilon_x(\cdot-X_u^{l}) \Id B_u^{l}  \bigg \rangle_{L^2(\R)} \bigg| \indicator{(B_s^\alpha)^{\mathrm{c}}} \Id s \bigg) \\
    \le  &\; \frac{\sigma}{N^2} \sum\limits_{i,j=1}^N \int\limits_0^T \int_{\R} \E \bigg( \indicator{(B_s^\alpha)^{\mathrm{c}}} \bigg| ( V^\epsilon_{x}(y-X_s^{i,\epsilon}) k^\epsilon(X_s^{i,\epsilon}-X_s^{j,\epsilon}) - V^\epsilon_x*((k^\epsilon*\rho_s^{\epsilon}) \rho_s^{\epsilon}) )(y) \\
    &\quad\cdot \; \frac{1}{N}\sum\limits_{l=1}^N  \int\limits_0^s V^\epsilon_x(y-X_u^{l}) \Id B_u^{l} \bigg| \bigg) \Id y \Id s  \\
    \le &\;  \frac{\sigma}{N^2} \sum\limits_{i,j=1}^N \int\limits_0^T \bigg( \int_{\R} \E \bigg( \indicator{(B_s^\alpha)^{\mathrm{c}}} | ( V^\epsilon_{x}(y-X_s^{i,\epsilon}) k^\epsilon(X_s^{i,\epsilon}-X_s^{j,\epsilon}) - V^\epsilon_x*((k^\epsilon*\rho_s^{\epsilon})(y) \rho_s^{\epsilon})(y) |^2 \bigg) \d y   \bigg)^{\frac{1}{2}} \\
    &\quad\cdot  \bigg( \int_{\R} \E\bigg(\bigg| \frac{1}{N}\sum\limits_{l=1}^N  \int\limits_0^s V^\epsilon_x(y-X_u^{l}) \Id B_u^{l} \bigg|^2 \bigg) \Id y  \bigg)^{\frac{1}{2}}  \Id s  \\
    \le &    \frac{  2 C_{\mathrm{BDG} }^{\frac{1}{2}}  \sigma}{N} \sum\limits_{i=1}^N\int\limits_0^T \bigg( \int_{\R} \E \bigg( \indicator{(B_s^\alpha)^{\mathrm{c}}} (| ( V^\epsilon_{x}(y-X_s^{i,\epsilon})|^2 \norm{k^\epsilon}_{L^\infty(\R)}^2+ | V^\epsilon_x*((k^\epsilon*\rho_s^{\epsilon}) \rho_s^{\epsilon})(y) |^2 )  \bigg) \Id y  \bigg)^{\frac{1}{2}} \\
    &\quad\cdot \frac{1}{N}  \bigg( \int_{\R} \E\bigg( \sum\limits_{l=1}^N  \int\limits_0^T |V^\epsilon_x(y-X_u^{l})|^2 \Id u  \bigg) \Id y  \bigg)^{\frac{1}{2}}  \Id s  \\
    \le  &\; \frac{2 C_{\mathrm{BDG} }^{\frac{1}{2}}  \sigma}{N^\frac12} \norm{V^\epsilon_x}_{L^2(\R)}   \int\limits_0^T \E \bigg( \indicator{(B_s^\alpha)^{\mathrm{c}}} \big( \norm{V^\epsilon_x}_{L^2(\R)}^2 \norm{k^\epsilon}_{L^\infty(\R)}^2 + \norm{ V^\epsilon_x*((k^\epsilon*\rho_s^{\epsilon}) \rho_s^{\epsilon})}_{L^2(\R)}^2 \big) \bigg)^{\frac{1}{2}}  s^{\frac{1}{2}} \Id s  \\
    \le &\; \frac{2 C(\gamma) C_{\mathrm{BDG} }^{\frac{1}{2}}  \sigma}{N^{\frac12+\gamma}} \norm{V^\epsilon_x}_{L^2(\R)}^2\Big(\norm{k^\epsilon}_{L^\infty(\R)} \frac{2}{3} T^\frac{3}{2}+ \int\limits_0^T \norm{k^\epsilon*\rho_s^\epsilon}_{L^\infty(\R)} s^{\frac{1}{2}} \Id s \Big).
   \end{align*}
  This completes the estimate on the set \((B_s^\alpha)^{\mathrm{c}}\) and we have shown our Lemma.
\end{proof}

\begin{lemma}[Stochastic Integral Inequality]
  Under the assumptions of Theorem~\ref{theorem: emp_measure_l2_estimate} we have the following \(L^2\)-estimate for the stochastic integral,
  \begin{align}\label{eq: main_theorem_stochastic_integral}
    4 \sigma^2 \E \bigg( \sup\limits_{0 \le t \le T}   \int\limits_{0}^t  \bigg\|\frac{\sigma}{N} \sum\limits_{i=1}^N \int\limits_0^s V^\epsilon_{xx}(\cdot-X_u^{i}) \Id B_u \bigg\|_{L^2(\R)}^2 \Id s\bigg)
    \le & \;  \frac{T^\frac{3}{2}}{N} \norm{V^\epsilon_{xx}}_{L^2(\R)}^2 .
  \end{align}
\end{lemma}

\begin{proof}
  An application of the Burkholder--Davis--Gundy inequality implies
  \begin{align*}
    & \; \E \bigg( \sup\limits_{0 \le t \le T}   \int\limits_{0}^t  \norm{ \frac{1}{N} \sum\limits_{i=1}^N \int\limits_0^s V^\epsilon_{xx}(\cdot-X_u^{i}) \Id B_u }_{L^2(\R)}^2 \Id s\bigg)  \\
    \le & \;   \int\limits_{0}^T \int_{\R} \E \bigg(\bigg|\frac{1}{N} \sum\limits_{i=1}^N \int\limits_0^s V^\epsilon_{xx}(y-X_u^{i}) \Id B_u\bigg|^2 \bigg) \Id y  \bigg) \Id s\\
    \le &\; \frac{1}{N^2}  \int\limits_{0}^T \int_{\R} \E \bigg( \sum\limits_{i=1}^N \int\limits_0^s | V^\epsilon_{xx}(y-X_u^{i}) |^2 \Id u  \bigg) \Id y  \bigg) \Id s\le \;  \frac{T^\frac{3}{2}}{N} \norm{V^\epsilon_{xx}}_{L^2(\R)}^2.
  \end{align*}
\end{proof}

\begin{proof}[Continuation of the proof of theorem \ref{theorem: emp_measure_l2_estimate}]
  We are ready to input the estimates from above lemmata in the the inequality \eqref{maininequality}. We find
  \begin{align*}
	& \E\bigg(    \sup\limits_{0 \le  t \le T} \norm{V^\epsilon*\mu_t^{N,\epsilon}  -V^\epsilon*\rho_t^{\epsilon}}_{L^2(\R)}^2\bigg) +   \frac{\sigma^2}{8}  \E  \bigg(   \int\limits_{0}^T  \norm{ V_x^\epsilon *\mu_s^{N,\epsilon} - V_x^\epsilon*\rho_s^{\epsilon}}_{L^2(\R)} \Id s \bigg)  \\
	\le & \; \frac{2}{N}  \norm{V^\epsilon}_{L^2(\R)}^2 
	+\frac{16T\norm{k_x^\epsilon}_{L^\infty(\R)}^2 \norm{V^\epsilon}_{L^2(\R)}^2}{\sigma^2 N^{2\alpha}}   +\frac{4T\norm{V^\epsilon}_{L^2(\R)}^2}{\sigma^2N^{2(\alpha+\delta)}} \\
	&+\Big(\frac{4\norm{   V^\epsilon_x}_{L^2(\R)}^2}{N^{2\alpha} \sigma^2}  +  \frac{16  \norm{V^\epsilon}_{L^2(\R)}^2 }{N\sigma^2} \Big)  \int\limits_0^T \norm{k^\epsilon*\rho_s^{\epsilon}}_{L^\infty(\R)}^2  \Id s \\
	& +  \frac{C(\gamma)T}{N^\gamma}\Big(\norm{ V^\epsilon }_{L^2(\R)}^2  \norm{k^\epsilon}_{L^\infty(\R)}^2 +  \norm{ V^\epsilon_x}_{L^2(\R)}^2 \Big) \\
	& +\frac{2 \sigma T^\frac{3}{2} C_{\mathrm{BDG}}^{\frac{1}{2}}}{N^{\alpha+\frac12} }   \norm{V^\epsilon_x}_{L^2(\R)}^2  \norm{k^\epsilon_x}_{L^\infty(\R)}+\sigma\frac{C_{\mathrm{BDG}}^{\frac{1}{2}}T^\frac{3}{2} }{N^{\alpha+\delta+\frac12}}\norm{V^\epsilon_x}_{L^2(\R)}^2\\
	\nonumber 
	&+ \Big(\sigma  \frac{C_{\mathrm{BDG}  }^{\frac{1}{2}}\norm{V^\epsilon_{xx} }_{L^2(\R)}}{N^{\alpha+\frac12}}  \norm{V^\epsilon_x}_{L^2(\R)}  +\sigma \frac{2 C_{\mathrm{BDG}}^{\frac{1}{2}}\norm{V^\epsilon_x}_{L^2(\R)}^2 }{N}\Big)\int\limits_0^T \norm{k^\epsilon*\rho_s^\epsilon}_{L^\infty(\R)} s^{\frac{1}{2}} \Id s.\nonumber \\
	&  + \frac{2 C(\gamma) C_{\mathrm{BDG} }^{\frac{1}{2}}  \sigma}{N^{\frac12+\gamma}} \norm{V^\epsilon_x}_{L^2(\R)}^2\Big(\norm{k^\epsilon}_{L^\infty(\R)} \frac{2}{3} T^\frac{3}{2}+ \int\limits_0^T \norm{k^\epsilon*\rho_s^\epsilon}_{L^\infty(\R)} s^{\frac{1}{2}} \Id s \Big).\\
	& +\frac{T^\frac{3}{2}}{N} \norm{V^\epsilon_{xx}}_{L^2(\R)}^2.
  \end{align*}
  The above estimate is the most general one we obtain. In the following we simplify it to derive a usable estimates. In the process we may loose some convergence rate, depending on the concrete problem at hand. Noticing that by mass conservation $$\norm{k^\epsilon*\rho_s^\epsilon}_{L^2(0,T;L^\infty(\R))}\leq \norm{k^\epsilon}_{L^\infty(\R)}\norm{\rho_s^\epsilon}_{L^2(0,T;L^1(\R))}\leq T \norm{k^\epsilon}_{L^\infty(\R)},$$ by keeping all the $N$ and $\epsilon$ dependent terms and put all the other constants into a universal constant $C$, which depends on $T$, $\sigma$, $\gamma$, $C_{\mathrm{BDG}}$, we obtain
  \begin{align*}
    & \; \E\bigg(    \sup\limits_{0 \le  t \le T} \norm{V^\epsilon*\mu_t^{N,\epsilon}  -V^\epsilon*\rho_t^{\epsilon}}_{L^2(\R)}^2\bigg) +   \frac{\sigma^2}{8}  \E  \bigg(   \int\limits_{0}^T  \norm{ V_x^\epsilon *\mu_s^{N,\epsilon} - V_x^\epsilon*\rho_s^{\epsilon}}_{L^2(\R)} \Id s \bigg)  \\
    \le & \;  \frac{C}{N} ( \norm{V^\epsilon}_{H^1(\R)}^2 \norm{k^\epsilon}_{L^\infty}^2 + \norm{V^\epsilon_{xx}}_{L^2(\R)}^2)
    +\frac{C\norm{V^\epsilon}_{H^1(\R)}^2 (1+ \norm{k^\epsilon}_{L^\infty(\R)}^2)}{N^\gamma} \\
    &+\frac{\norm{k^\epsilon_x}_{L^\infty(\R)}^2 \norm{V^\epsilon}_{L^2(\R)}^2+ \norm{k^\epsilon}_{L^\infty(\R)}^2 \norm{V_x^\epsilon}_{L^2(\R)}^2
    + \norm{V^\epsilon}_{L^2(\R)}^2}{N^{2\alpha }} \\
    &+C\frac{\norm{V_x^\epsilon}_{L^2(\R)}^2 (1+ \norm{k_x^\epsilon}_{L^\infty(\R)}) +  \norm{V_x^\epsilon}_{L^2(\R)} \norm{V^\epsilon_{xx}}_{L^2(\R)} \norm{k^\epsilon}_{L^\infty(\R^d)} }{N^{\alpha+\frac12 }} .
  \end{align*}
  In the above estimates, $\alpha \in (0,\frac12)$ and $\delta>0$ are also used and the Theorem is proven.
\end{proof}

In the our main setting \(k^\epsilon=W^\epsilon*V^\epsilon\) we provide the following rough estimate. 

\begin{corollary} \label{cor: general_case_estimate}
Let \(k^\epsilon=W^\epsilon*V^\epsilon\) and \(W^\epsilon,V^\epsilon\) be admissible with rates \(a_W,a_V\). If Theorem~\ref{theorem: emp_measure_l2_estimate} holds, then 
\begin{align*}
&\E\bigg(    \sup\limits_{0 \le  t \le T} \norm{V^\epsilon*\mu_t^{N,\epsilon}  -V^\epsilon*\rho_t^{\epsilon}}_{L^2(\R)}^2\bigg) +   \frac{\sigma^2}{8}  \E  \bigg(   \int\limits_{0}^T  \norm{ V_x^\epsilon *\mu_s^{N,\epsilon} - V_x^\epsilon*\rho_s^{\epsilon}}_{L^2(\R)} \Id s \bigg)  \\
&\le \frac{C}{N \epsilon^{2a_W+4a_V}}
+ \frac{C}{	N^{2\alpha} \epsilon^{2a_W+4a_V}}
+ \frac{C}{N^{\alpha+\frac{1}{2}} \epsilon^{a_W+3a_V}} 
+  \frac{C}{N^\gamma \epsilon^{2a_W+4a_V}}. 
\end{align*}
\end{corollary}

\begin{proof}
Estimating all norms of \(V^\epsilon\) by \(\norm{V^\epsilon}_{H^2(\R)} \le C \epsilon^{-a_V}\) and using Young's inequality to find 
\begin{align*}
\norm{k^\epsilon}_{L^\infty(\R)} + \norm{k^\epsilon_x}_{L^\infty(\R)} 
\le 2 \norm{W^\epsilon}_{L^2(\R)} \norm{V^\epsilon}_{H^2(\R)} \le C \epsilon^{-a_W-a_V}.
\end{align*}
Hence the right hand side of the main inequality in Theorem~\ref{theorem: emp_measure_l2_estimate} can be estimated by 
\begin{equation*}
\frac{C}{N \epsilon^{2a_W+4a_V}}
+ \frac{C}{	N^{2\alpha} \epsilon^{2a_W+4a_V}}
+ \frac{C}{N^{\alpha+\frac{1}{2}} \epsilon^{a_W+3a_V}} 
+  \frac{C}{N^\gamma \epsilon^{2a_W+4a_V}}. 
\end{equation*}
\end{proof}

Now, that we have proven our main estimate, we are ready to demonstrate the relative entropy estimates by combining Theorem~\ref{theorem: emp_measure_l2_estimate} and Lemma~\ref{lemma: relative_entropy_L_2_estimate}.
We start with the first main result of this paper

\begin{proof}[Proof of Theorem~\ref{maintheorem}]
  We combine the assumptions of Theorem~\ref{theorem: emp_measure_l2_estimate} and the results from Lemma~\ref{lemma: relative_entropy_L_2_estimate}, Theorem~\ref{theorem: emp_measure_l2_estimate} and Corollary~\ref{cor: general_case_estimate}, to find a small $\beta_1 \le \beta_\alpha$ so that for \(0 <\beta \le \beta_1\), \(\epsilon= N^{-\beta}\) and small $0<\lambda\ll 1$
  \begin{equation*}
    {\mathcal H}_N(\rho_t^{N,\epsilon} \vert \rho_t^{\otimes N , \epsilon} )  \le \; \frac{C}{N^{\frac{1}{2}+\lambda}} =o \bigg(\frac{1}{\sqrt{N}} \bigg).
  \end{equation*}

  This allows us to demonstrate strong convergence in the \(L^\infty([0,T];L^1(\R))\)-norm. Indeed, let us recall the Csisz{\'a}r--Kullback--Pinsker inequality~\cite[Chapter~22]{Villani2009}, which states that for any \(m \in \N\) and function \(f,g \colon \R^m \to \R \) we have
  \begin{equation}\label{eq: CKP_inequality}
	\norm{f-g}_{L^1(\R^m)} \le \sqrt{ 2m {\mathcal H}_m( f \;  | \;g) }
  \end{equation}
  and the relative entropy inequaity~\cite[Lemma~3.9]{MoralMiclo2001}
  \begin{equation}\label{eq: realtive_entropy_marginal_inequality}
	{\mathcal H}_m( \rho^{N,m,\epsilon}_t \;  | \; \rho^{\otimes m,\epsilon}_t)
	\le 2{\mathcal H}_N( \rho^{N,\epsilon}_t \;  | \; \rho_t^{\otimes N,\epsilon})
  \end{equation}
  for \(m \le N \). Consequently,
  \begin{equation*}
	\norm{\rho_t^{N,2,\epsilon} - \rho_t^{ \epsilon} \otimes \rho_t^{ \epsilon}}_{L^1(\R^2)}^2
	\le 2{\mathcal{H}}_2(\rho_t^{N,1,\epsilon} \vert \rho_t^{ \epsilon})
	\le 4{\mathcal{H}}_N(\rho_t^{N,\epsilon} \vert \rho_t^{\otimes N , \epsilon})
	=o \bigg(\frac{1}{\sqrt{N}} \bigg).
  \end{equation*}

  In the case \(k^\epsilon = (W^\epsilon*V^\epsilon)_x\) the estimate~\eqref{eq: main_thm_entropy_estimate} is derived analogously. The key is to recognize that we actually derived an estimate on the derivative of \(V^\epsilon\), which we have not used so far. In the case of \(k^\epsilon = (W^\epsilon*V^\epsilon)_x\) we utilize it and as a result we obtain the same convergence rates. The estimate for the modulated energy follows also directly from equality~\eqref{eq: expextationmodulatedenergy_connection}, Young's inequality and an application of Theorem \ref{theorem: emp_measure_l2_estimate} for \(V^\epsilon\) and \(\hat{W}^\epsilon\) under the assumption that \(W^\epsilon,V^\epsilon\) are strongly admissible.
\end{proof}

\subsection{Special Choices of \texorpdfstring{\(W^\epsilon\)}{W} and \texorpdfstring{\(V^\epsilon\)}{V}}

We present a series of corollaries for Theorem~\ref{theorem: emp_measure_l2_estimate} for different 
choices of \(V^\epsilon\). In most applications we want to take a mollified sequence. In the special case \(V^\epsilon=J^\epsilon\) we obtain the following corollary.

\begin{corollary} \label{cor: mollifier_l2_estimate}
 
  Suppose Theorem~\ref{theorem: emp_measure_l2_estimate} holds true. Let \(V^\epsilon= J^\epsilon\) be a mollification, then for \(\epsilon = N^{-\beta}\) with some \(\beta < \beta_\alpha \) and \(\norm{k^\epsilon}_{L^\infty(\R)} \le \epsilon^{-a_k}\), \(\norm{k^\epsilon_x}_{L^\infty(\R)} \le \epsilon^{-a_{k}-1}\) for some \(a_k>0\), then we obtain the following \(L^2\)-estimate ,
  \begin{align*}
    & \E\bigg(    \sup\limits_{0 \le  t \le T} \norm{J^\epsilon*\mu_t^{N,\epsilon}  -J^\epsilon*\rho_t^{\epsilon}}_{L^2(\R)}^2\bigg) +   \frac{\sigma^2}{2}  \E  \bigg(   \int\limits_{0}^T  \norm{ J_x^\epsilon *\mu_s^{N,\epsilon} - J_x^\epsilon*\rho_s^{\epsilon}}_{L^2(\R)} \Id s \bigg)  \\
    &\quad\le \;  \frac{C}{N^{2\alpha-\beta}} + \frac{C}{N^{\alpha+1/2-4\beta}} +\frac{C }{N^{2\alpha - \beta -(2a_k+2)\beta} } +\frac{C}{N^{\alpha+\frac12 - 3\beta -(a_k+1)\beta} }\\&\quad\quad + \frac{C}{N^{2\alpha-3\beta-a_k \beta}} +\frac{C}{N^\gamma N ^{(3+2a_k)\beta}} \\
    &\quad\le \;  \frac{C }{N^{2\alpha - \beta -(2a_k+2)\beta} } +\frac{C}{N^{\alpha+\frac12 - 3\beta -(a_k+1)\beta} } + \frac{C}{N^{2\alpha-3\beta-a_k \beta}} +\frac{C}{N^\gamma N ^{(3+2a_k)\beta}}
  \end{align*}
  for a constant \(C\), which depends on $T$, $\sigma$, $\gamma$, $C_{\mathrm{BDG}}$. In particular if \(k \in L^\infty(\R)\) and \(k^\epsilon= (\zeta^\epsilon(J^\epsilon*k))*J^\epsilon\) the above estimate holds with \(\epsilon = N^{-\beta}\) and \(a_k=0\).
\end{corollary}

\begin{proof}
  If $V^\epsilon=J^\epsilon$, we obtains easily that
  \begin{align*}
    \norm{J^\epsilon}_{H^{m}(\R)}
    &=\frac{1}{\epsilon} \norm{\mathcal{F}^{-1}\bigg[\bigg(1+\bigg|\frac{\xi}{\epsilon}\bigg|^2\bigg)^{\frac{m}{2}} \mathcal{F}[J](\xi)\bigg]\bigg(\frac{\cdot}{\epsilon}\bigg)}_{L^2(\R)} \\
    &=\frac{1}{\epsilon^{1/2+m}}\norm{\mathcal{F}^{-1}[(\epsilon^2+|\xi|^2)^{\frac{m}{2}} \mathcal{F}[J](\xi)]}_{L^2(\R)}\le \frac{C}{\epsilon^{1/2+m}}.
  \end{align*}
  Therefore, we obtained with $\epsilon=N^{-\beta}$,
  \begin{align*}
    & \E\bigg(    \sup\limits_{0 \le  t \le T} \norm{J^\epsilon*\mu_t^{N,\epsilon}  -J^\epsilon*\rho_t^{\epsilon}}_{L^2(\R)}^2\bigg) +   \frac{\sigma^2}{2}  \E  \bigg(   \int\limits_{0}^T  \norm{ J_x^\epsilon *\mu_s^{N,\epsilon} - J_x^\epsilon*\rho_s^{\epsilon}}_{L^2(\R)} \Id s \bigg)  \\
    &\quad\le  \;  \frac{C}{N^{2\alpha-\beta}} + \frac{C}{N^{\alpha+1/2-4\beta}} +\frac{C }{N^{2\alpha - \beta } \epsilon^{2a_k+2} } +\frac{C}{N^{\alpha+\frac12 - 3\beta } \epsilon^{a_k+1} } + \frac{C}{N^{2\alpha-3\beta} \epsilon^{a_k}} +\frac{C}{N^\gamma \epsilon^{3+2a_k}}
  \end{align*}
  The second claim follows by Young's inequality and the scaling of the mollifier. More precisely,
  \begin{align*}
	\norm{J^\epsilon}_{W^{m,1}(\R)}
	&=\frac{1}{\epsilon}\norm{\mathcal{F}^{-1}[(1+|\xi|^2)^{\frac{m}{2}} \mathcal{F}[J](\xi)]}_{L^1(\R)}
	\le \frac{C}{\epsilon^{m}}
  \end{align*}
\end{proof}

\begin{corollary} \label{cor: bounded_case}
  Suppose Assumptions~\ref{ass: convergence_in_probability},~\ref{ass: law_of_large_numbers} hold for \(\alpha \in (\frac{1}{4},\frac{1}{2})\) and suppose the bounded force \(k\) has the approximation \(k^\epsilon=W^\epsilon*V^\epsilon\) with \(W^\epsilon= \zeta^{\epsilon} (k*J^\epsilon)\) and \(V^\epsilon= J^\epsilon\). Then for \(\epsilon= N^{-\beta}\) and \(\beta < \min\bigg(\frac{\alpha}{6},\frac{1}{10}(4\alpha-1), \beta_\alpha\bigg)\), there exists an \(0 < \lambda \ll 1\) such that
  \begin{equation*}
    \sup\limits_{t \in [0,T]}  {\mathcal{H}}_N(\rho_t^{N,\epsilon} \vert \rho_t^{\otimes N , \epsilon})
    \le \frac{C}{N^{\frac12+\lambda}}
  \end{equation*}
  for a constant \(C\), which depends on $T$, $\sigma$, $\gamma$, $C_{\mathrm{BDG}}$.
\end{corollary} 

\begin{proof}
  Let us start by estimating \(\norm{W^\epsilon}_{L^2(\R)}^2\) in inequality~\eqref{entropyW},
  \begin{equation*}
    \norm{W^\epsilon}_{L^2(\R)}^2
    = \int_{\R} |\zeta^{\epsilon}(x)|^2 \bigg|\int_{\R}  k(x-y)J^\epsilon(y) \Id y \bigg|^2 \Id x
    \le 4\epsilon^{-2} \norm{k}_{L^\infty(\R)}
    = 4N^{2\beta} \norm{k}_{L^\infty(\R)}
  \end{equation*}
  Now, applying Corollary~\ref{cor: mollifier_l2_estimate} to inequality~\eqref{entropyW}, keeping track of the powers, we obtain the result. More precisely, notice that we have to first fix $\alpha \in (\frac{1}{4},\frac{1}{2})$, then choos a $\beta_J$ such that the terms for $0<\beta\leq\beta_J$ smaller than $\frac{C}{N^{\frac{1}{2}+\lambda}}$, and then fix $\gamma$ such that the estimate holds.
\end{proof}

Next, we provide a similar corollary in the case the force \(k\) is a potential and has a convolution structure.

\begin{corollary}\label{cor: potential_case}
  Suppose Assumptions~\ref{ass: convergence_in_probability},~\ref{ass: law_of_large_numbers} hold for \(\alpha \in (\frac{1}{4},\frac{1}{2})\). Let the force \(k\) be given by a potential \(k=(W*V)_x\) with \(W,V \in H^{\frac{1}{2}}(\R)\) and its approximation is given by \(k^\epsilon= ((W*J^\epsilon)*(V*J^\epsilon))_x\). Then for \(\epsilon= N^{-\beta}\) and \(\beta < \min\bigg(\frac{1}{3},\alpha-\frac{1}{4},\beta_{\alpha}\bigg)\), there exists a \(0 < \lambda \ll 1\) such that
  \begin{align*}
    \sup\limits_{t \in [0,T]}  {\mathcal{H}}_N(\rho_t^{N,\epsilon} \vert \rho_t^{\otimes N , \epsilon})
    &\le \frac{C}{N^{\frac12+\lambda}} , \\
	\sup\limits_{t \in [0,T]} |{\mathcal K}_{N}  (\rho_t^{N,\epsilon}\vert \rho_t^{\otimes N, \epsilon})|
	 &\le \frac{C}{N^{\frac12+\lambda}} . 
  \end{align*}
  for a constant \(C\), which depends on $T$, $\sigma$, $\gamma$, $C_{\mathrm{BDG}}$.
\end{corollary}  

\begin{proof}
  Since \(W\in H^{\frac{1}{2}}(\R)\) we know that \(W*J^\epsilon \in L^2(\R)\) and therefore we only need to estimate the \(L^2\)-norm for \(V^\epsilon_x = (V*J^\epsilon)_x \) in inequality~\eqref{entropyW} to obtain the convergence rates. We emphasize that in Theorem~\ref{theorem: emp_measure_l2_estimate} we also obtained an estimate on the gradient \(V^\epsilon\). Consequently, we can use Theorem~\ref{theorem: emp_measure_l2_estimate} for the function \(V^\epsilon =  V*J^\epsilon\). By the estimate
\begin{align*}
	\norm{J^\epsilon}_{W^{m,1}(\R)}
	&=\frac{1}{\epsilon}\norm{\mathcal{F}^{-1}[(1+|\xi|^2)^{\frac{m}{2}} \mathcal{F}[J](\xi)]}_{L^1(\R)}
	\le \frac{C}{\epsilon^{m}}\\
		\norm{J^\epsilon_{x}}_{W^{m,1}(\R)}
	&=\frac{1}{\epsilon}\norm{\mathcal{F}^{-1}[(1+|\xi|^2)^{\frac{m}{2}} \mathcal{F}[J_x](\xi)]}_{L^1(\R)}
	\le \frac{C}{\epsilon^{m+1}}
\end{align*}
for any \(m \ge 0\)
we know that 
\begin{align}
&\nonumber	\norm{V^\epsilon}_{L^2(\R)}\leq \norm{V}_{L^2(\R)}\leq C\\
&\nonumber		\norm{V_x^\epsilon}_{L^2(\R)}\leq \norm{V}_{H^\frac12(\R)}\norm{J^\epsilon}_{W^{\frac{1}{2},1}(\R)}\leq \frac{C}{\epsilon^\frac12}\\
&\label{eq: l2_bound_V}			\norm{V_{xx}^\epsilon}_{L^2(\R)}\leq \norm{V}_{H^\frac12(\R)}\norm{J_x^\epsilon}_{W^{\frac12,1}(\R)}\leq \frac{C}{\epsilon^\frac32}\\
&\nonumber \norm{k^\epsilon}_{L^\infty(\R)}\leq \norm{W^\epsilon}_{H^{\frac12}(\R)}\norm{V^\epsilon}_{H^{\frac12}(\R)}\leq C\\
&\nonumber \norm{k^\epsilon_x}_{L^\infty(\R)}\leq \norm{W^\epsilon}_{H^1(\R)}\norm{V^\epsilon}_{H^1(\R)}\leq \frac{C}{\varepsilon}.
\end{align}
Plugging in all estimates with \(\epsilon=N^{-\beta}\) into Theorem~\ref{theorem: emp_measure_l2_estimate} and having equality~\eqref{eq: expextationmodulatedenergy_connection} in mind we obtain the rate of \(\beta\) and the estimate on the modulated energy. 
\end{proof}
We have now shown in two cases how to derive explicit estimates on the relative entropy \({\mathcal{H}}_N(\rho_t^{N,\epsilon} \vert \rho_t^{\otimes N , \epsilon})\) with the help of Theorem~\ref{theorem: emp_measure_l2_estimate}. 
 In general, if the function $W^\epsilon,V^\epsilon$ have low regularity, we need to mollify them to make them admissible. Hence, we borrow the necessary regularity from \(J\) and consequently, get higher rates of \(N\) in our estimates. Compare for instance Corollary~\ref{cor: bounded_case} and Corollary~\ref{cor: potential_case}. Therefore the estimate by using the regularity of the $J^\epsilon$ term, will lead to weaker convergence rates. The benefit is of course that one does not require a potential field and the convolution structure of the potential.

\section{De-regularization of the high dimensional PDE and the limiting PDE }\label{sec: PDE convergence}

The goal of this section is to prove Theorem~\ref{maintheorem2}, i.e. the strong form of propagation of chaos on the PDE level in the \(L^1\)-norm. For the de-regularization of Liouville equation \eqref{eq: regularized_Liouville_equation} we need \(k\in L^\infty(\R)\). We take the following approximation \(k^\epsilon= (\zeta^\epsilon(k*J^\epsilon))*J^\epsilon)\). We need convergence results between \(\rho_t^{N,1,\epsilon}\) and \( \rho_t^{N,1}\) as well as \(\rho_t^\epsilon\) and \(\rho_t\). The latter convergence was shown in~\cite[Section~3]{Nikolaev2023}. More precisely,
\begin{equation} \label{eq: l1_coonv_mean_pde}
   \lim\limits_{\epsilon \to 0} \norm{\rho^\epsilon - \rho}_{L^1([0,T];L^1(\R))} = 0 .
\end{equation}

It remains to show that the approximated Liouville equation converges in entropy to the Liouville equation. An application of inequality~\eqref{eq: CKP_inequality} implies also the \(L^1\)-convergence. 

\begin{lemma}\label{lemma: Lioville_equations_approximation_convergence}
  Let \(k\in L^\infty(\R)\), \(\rho^{N, \epsilon}\) be the solution of the regularized Liouville equation~\eqref{eq: regularized_Liouville_equation} and \(\rho^{N}\) the solution of the Liouville equation~\eqref{eq: Liouville_equation}. Then, we have
  \begin{align*}
    & \;  \sup\limits_{t \in [0,T]} {\mathcal H}_N( \rho^{N,\epsilon}_t \;  | \; \rho^{N}_t) +  \frac{\sigma^2}{4N}   \sum\limits_{i=1}^N \int\limits_0^T \int\limits_{\R^N} \rho_s^{N,\epsilon} 	  \bigg|\partial_{x_i} \log \bigg(\frac{\rho^{N,\epsilon}_s}{\rho^{N}_s} \bigg) \bigg|^2 \Id \mathrm{X}^N \Id s   \\
    \le &\;   C T \norm{k}_{L^\infty(\R)}^2 \sup\limits_{t \in [0,T]}  \sqrt{ {\mathcal H}_N( \rho^{N,\epsilon}_t \;  | \; \rho^{\otimes N,  \epsilon}_t)   }
    + 2C \norm{k}_{L^\infty(\R)}^2   \norm{\rho_s^\epsilon - \rho_s}_{L^1([0,T];L^1(\R))}  \\
    &+   \int\limits_0^T \int\limits_{\R^2} |k (x_{1} -x_{2} ) -k^\epsilon (x_{1} -x_{2} )|^2  \rho_s(x_1) \rho_s(x_2)  \Id x_1 \Id x_2 \Id s .
  \end{align*}
  In particular, the last term vanishes by dominated convergence.
\end{lemma} 

\begin{proof}
  We start by computing the time derivative of \(\rho^{N,\epsilon}_t \log \bigg(\frac{\rho^{N,\epsilon}_t}{\rho^{N}_t} \bigg)\). We have
  \begin{align*}
    & {\mathcal H}_N( \rho^{N,\epsilon}_t \;  | \; \rho^{N}_t)=  {\mathcal H}_N( \rho^{N,\epsilon} \;  | \; \rho^{N})(0) + \int\limits_0^t \frac{\Id}{\Id s} {\mathcal H}_N( \rho^{N,\epsilon}_s \;  | \; \rho^{N}_s)  \Id s  \\
    = &\;-  \frac{\sigma^2}{2N}  \sum\limits_{i=1}^N   \int\limits_0^t \int\limits_{\R^N}  \partial_{x_i} \rho_s^{N,\epsilon} \partial_{x_i} \log \bigg(\frac{\rho^{N,\epsilon}_s}{\rho^{N}_s} \bigg) \Id \mathrm{X}^N \Id s   \\
    & -  \frac{1}{N}  \sum\limits_{i=1}^N   \int\limits_0^t \int\limits_{\R^N}
    \Bigg ( \rho^{N,\epsilon}_s \frac{1}{N} \sum\limits_{j=1}^N  k^\epsilon (x_{i} -x_{j} )  \Bigg) \partial_{x_i} \log \bigg(\frac{\rho^{N,\epsilon}_s}{\rho^{N}_s} \bigg) \Id \mathrm{X}^N \Id s  \\
    &+  \frac{\sigma^2}{2N}  \sum\limits_{i=1}^N   \int\limits_0^t \int\limits_{\R^N}  \partial_{x_i} \rho_s^N \partial_{x_i} \bigg(\frac{\rho^{N,\epsilon}_s}{\rho^{N}_s} \bigg) \Id \mathrm{X}^N \Id s  \\
    &+ \frac{1}{N}   \sum\limits_{i=1}^N  \int\limits_0^t \int\limits_{\R^N} \Bigg ( \rho^N_s \frac{1}{N} \sum\limits_{j=1}^N  k (x_{i} -x_{j} ) \Bigg) \partial_{x_i}
    \bigg(\frac{\rho^{N,\epsilon}_s}{\rho^{N}_s} \bigg)
    \Id \mathrm{X}^N \Id s \\
    = &\;  -  \frac{\sigma^2}{2N}  \sum\limits_{i=1}^N  \int\limits_0^t \int\limits_{\R^N} \rho_s^{N,\epsilon}   \bigg|\partial_{x_i} \log \bigg(\frac{\rho^{N,\epsilon}_s}{\rho^{N}_s} \bigg) \bigg|^2 \Id \mathrm{X}^N \Id s  \\
    &+   \frac{1}{N^2} \sum\limits_{i,j=1}^N \int\limits_0^t \int\limits_{\R^N}  ( k (x_{i} -x_{j} ) -k^\epsilon (x_{i} -x_{j} )) \rho^{N,\epsilon}_s  \partial_{x_i} \log
    \bigg(\frac{\rho^{N,\epsilon}_s}{\rho^{N}_s} \bigg)
    \Id \mathrm{X}^N \Id s \\
   \le  &\; -   \frac{\sigma^2}{4N}   \sum\limits_{i=1}^N \int\limits_0^t \int\limits_{\R^N} \rho_s^{N,\epsilon} 	  \bigg|\partial_{x_i} \log \bigg(\frac{\rho^{N,\epsilon}_s}{\rho^{N}_s} \bigg) \bigg|^2 \Id \mathrm{X}^N \Id s  \\
    &+  \frac{1}{\sigma^2 N^2} \sum\limits_{i,j=1}^N \int\limits_0^t \int\limits_{\R^N} |k (x_{i} -x_{j} ) -k^\epsilon (x_{i} -x_{j} )|^2  \rho^{N,\epsilon}_s
    \Id \mathrm{X}^N \Id s .
  \end{align*}
  Now, it is enough to show, that the last term vanishes for \(N \to \infty\) and consequently for \(\epsilon \to 0 \). We start by using the fact that the particle system~\eqref{eq: regularized_particle_system} is exchangeable. We obtain
  \begin{align*}
    &  \;  \frac{1}{\sigma^2 N^2} \sum\limits_{i,j=1}^N \int\limits_0^t \int\limits_{\R^N} |k (x_{i} -x_{j} ) -k^\epsilon (x_{i} -x_{j} )|^2  \rho^{N,\epsilon}_s (\mathrm{X}^N)
    \Id \mathrm{X}^N \Id s \\
    =&\;  \frac{1}{\sigma^2}  \int\limits_0^t \int\limits_{\R^2} |k (x_{1} -x_{2} ) -k^\epsilon (x_{1} -x_{2} )|^2  \rho^{N,2, \epsilon}_s (x_1,x_2)
    \Id x_1 \Id x_2 \Id s.
  \end{align*}
  Hence, we obtained an expression in which the dimension does not change in the limit. By applying mass conservation, the Csisz{\'a}r--Kullback--Pinsker inequality~\eqref{eq: CKP_inequality} and inequality~\eqref{eq: realtive_entropy_marginal_inequality} we further estimate the term
  \begin{align*}
    &\;   \int\limits_0^t \int\limits_{\R^2} |k (x_{1} -x_{2} ) -k^\epsilon (x_{1} -x_{2} )|^2  \rho^{N,2, \epsilon}_s(x_1,x_2)
    \Id x_1 \Id x_2 \Id s  \\
    =&\;   \int\limits_0^t \int\limits_{\R^2} |k (x_{1} -x_{2} ) -k^\epsilon (x_{1} -x_{2} )|^2 ( \rho^{N,2, \epsilon}_s- (\rho_s^\epsilon \otimes \rho_s^\epsilon)(x_1,x_2))
    \Id x_1 \Id x_2 \Id s  \\
    &+  \int\limits_0^t \int\limits_{\R^2} |k (x_{1} -x_{2} ) -k^\epsilon (x_{1} -x_{2} )|^2 (\rho_s^\epsilon \otimes \rho_s^\epsilon)(x_1,x_2)
    \Id x_1 \Id x_2 \Id s \\
     \le &\;  C \norm{k}_{L^\infty(\R)}^2 \int\limits_0^t \norm{\rho^{N,2, \epsilon}_s- \rho^{\otimes 2, \epsilon}}_{L^1(\R^2)}  \Id s
    +  \int\limits_0^t \int\limits_{\R^2} |k (x_{1} -x_{2} ) -k^\epsilon (x_{1} -x_{2} )|^2 \\
    &\quad\quad \bigg( (\rho_s^\epsilon (x_1) - \rho_s(x_1) ) \rho_s^\epsilon(x_2)
    + \rho_s(x_1)  (\rho_s^\epsilon (x_2) - \rho_s(x_2) )   +  \rho_s(x_1) \rho_s(x_2) \bigg)  \Id x_1 \Id x_2 \Id s \\
    \le &\;   C T \norm{k}_{L^\infty(\R)}^2 \sup\limits_{t \in [0,T]} \sqrt{ {\mathcal H}_N( \rho^{N,\epsilon}_t \;  | \; \rho^{\otimes N,  \epsilon}_t)  }
    + 2C \norm{k}_{L^\infty(\R)}^2   \norm{\rho_s^\epsilon - \rho_s}_{L^1([0,T];L^1(\R))}  \\
    &+  \int\limits_0^t \int\limits_{\R^2} |k (x_{1} -x_{2} ) -k^\epsilon (x_{1} -x_{2} )|^2  \rho_s(x_1) \rho_s(x_2)  \Id x_1 \Id x_2 \Id s .
  \end{align*}
  Plugging this estimate into our above entropy calculation and taking the supremum in time proves the lemma.
\end{proof}

Combing both implies the strong convergence on the PDE-level of any observable \(\rho^{N,m}\) to the law \(\rho^{\otimes m}\) in the \(L^1(\R)\)-norm. 
    
\begin{proof}[Proof of theorem \ref{maintheorem2}]
  For $k\in L^\infty(\R)$, let $k^\epsilon=(\zeta^\epsilon(k*J^\epsilon))*J^\epsilon$, therefore we take $W^\epsilon=\zeta^\epsilon( k^\epsilon * J^\epsilon) $ and \(V^\epsilon = J^\epsilon\) with \(\epsilon = \epsilon (N) = N^{-\beta}\). By assumption of the Theorem (see also \cite[Theorem 6.1]{Nikolaev2023}), there exists a $\beta_{\alpha} \in (0,\frac12)$ such that for all \(\beta \le \beta_{\alpha}\) the convergence in probability, Assumption~\ref{ass: convergence_in_probability}, and the law of large numbers, Assumption~\ref{ass: law_of_large_numbers} both hold. Therefore we can apply the result from Corollary~\ref{cor: bounded_case} for \(0<\beta < \min\bigg(\frac{1}{3}, \alpha-\frac{1}{4},\beta_{\alpha} \bigg)\) and obtain the convergence of the relative entropy \({\mathcal H}( \rho^{N,\epsilon}_t \;  | \; \rho_t^{\otimes N, \epsilon}) \) to zero. We can even get better convergence rate \(\beta\), since we are not interested in the order of convergence of the relative entropy. Applying~\eqref{eq: CKP_inequality} and ~\eqref{eq: realtive_entropy_marginal_inequality}, we obtain
  \begin{align*}
    &\; \norm{\rho^{N,m}-\rho^{\otimes m}}_{L^1([0,T];L^1(\R^m))}  \\
    \le &\;  \norm{\rho^{N,m}-\rho^{N,m,\epsilon}}_{L^1([0,T];L^1(\R^m))} + \norm{\rho^{N,m,\epsilon}-\rho^{\otimes m,\epsilon} }_{L^1([0,T];L^1(\R^m))} \\
    &+  \norm{\rho^{\otimes m,\epsilon}  -\rho^{\otimes m }}_{L^1([0,T];L^1(\R^m))} \\
    \le &\;  \int\limits_0^T \sqrt{2m {\mathcal H}_m( \rho^{N,m,\epsilon}_t \;  | \; \rho_t^{N,m})} +  \sqrt{2m {\mathcal   H}_m( \rho^{N,m,\epsilon}_t \;  | \; \rho^{\otimes m,\epsilon} _t)} +  \norm{\rho^{\otimes m,\epsilon} _t -\rho_t^{\otimes m}}_{L^1(\R^m)} \Id t  \\
    \le &\;  \int\limits_0^T \sqrt{4m {\mathcal H}_N( \rho^{N,\epsilon}_t \;  | \; \rho_t^{N})} +  \sqrt{4m {\mathcal H}_N( \rho^{N,\epsilon}_t \;  | \; \rho^{\otimes N, \epsilon}_t)}+  \norm{\rho^{\otimes m,\epsilon} _t -\rho^{\otimes m}_t}_{L^1(\R^m)} \Id t  .
  \end{align*}
  As mentioned the second term converges to zero. For the first term we use the inequality in Lemma~\ref{lemma: Lioville_equations_approximation_convergence} together with the fact that the \({\mathcal H}_N( \rho^{N,\epsilon}_t \;  | \; \rho_t^{\otimes N, \epsilon}) \) converges to zero and the dominated convergence to obtain
  \begin{align*}
    \limsup\limits_{N \to \infty} \int\limits_0^T  \sqrt{4m {\mathcal H}_N( \rho^{N,\epsilon}_t \;  | \; \rho_t^{N})}  \Id t
    & \le C(\norm{k}_{L^\infty(\R)},m) \limsup\limits_{N\to \infty} \int\limits_0^T \norm{\rho^\epsilon_t -\rho_t}_{L^1(\R)}^{\frac{1}{2}} \Id t \\
    &\le C(\norm{k}_{L^\infty(\R)},m,T) \limsup\limits_{N\to \infty} \bigg( \int\limits_0^T \norm{\rho^{\epsilon}_t -\rho_t}_{L^1(\R)} \Id t \bigg)^{\frac{1}{2}}\\
    &=0 .
  \end{align*}
  where the last equality follows by~\eqref{eq: l1_coonv_mean_pde}. Consequently, it remains to show that the third term vanishes, i.e.
  \begin{equation}\label{eq: aux_chaotic_law_conv}
    \limsup\limits_{N\to \infty} \norm{\rho^{\otimes m,\epsilon} _t -\rho^{\otimes m}_t}_{L^1([0,T];L^1(\R^m))} = 0.
  \end{equation}
  Again this follows by~\eqref{eq: l1_coonv_mean_pde} and an induction argument. Indeed let us assume \(m=2\), then by mass conservation we have
  \begin{align*}
    &\norm{\rho^{\otimes 2,\epsilon} _t -\rho^{\otimes 2}_t}_{L^1([0,T];L^1(\R^2))} \\
    &\quad=  \int\limits_0^T \int_{\R^2} |(\rho^\epsilon_t(x_1)-\rho_t(x_1)) \rho^\epsilon_t(x_2) + \rho_t(x_1) (\rho^\epsilon_t(x_2)-\rho_t(x_2))| \Id x_1 \Id x_2 \Id t  \\
    &\quad \le  2\norm{\rho^\epsilon_t -\rho_t}_{L^1([0,T];L^1(\R))} \xrightarrow{N  \to \infty}      \quad  0,
  \end{align*}
  which proves the initial case for the induction. Now, by the same argument one can prove the induction step and therefore equation~\eqref{eq: aux_chaotic_law_conv}.
\end{proof}

\section{Application}\label{sec: application}

We provide some examples for which Theorem~\ref{maintheorem} can be shown with the same techniques developed in Section~\ref{sec: relative_entropy_method}. In particular we demonstrate the convergence in relative entropy in the attractive Coulomb case on the whole space. Note that the rate of converges may vary across these examples. As stated in Remark~\ref{remark: general_kernels_l2main_thm} we only need the existence of approximated PDE~\eqref{eq: regularized_aggregation_diffusion_pde}, the particle system~\eqref{eq: regularized_mean_field_trajectories}, the convergence in probability  of the particle system \(\mathbf{X}^N\) to the mean-field limit \(\mathbf{Y}^N\)(Assumption~\eqref{ass: convergence_in_probability}) and the law of large numbers (Assumption~\ref{ass: law_of_large_numbers}). Since we are working on the regularized level, we can often assume the existence of the above results. 

Although the result in Theorem~\ref{maintheorem} also works with rotational field, it worth to study directly a convolution type of potential field to achieve better cut-off rate, in other words, to allow bigger $\beta$. For a given potential field, the challenging part is to find a convolution structure for the potential described in Section~\ref{sec: relative_entropy_method}. The first idea to obtain interesting kernels, beside the Delta-Distribution, which was given in~\cite{Oeschlager1987}, is to look at infinite divisible distributions. Assume that \(k^\epsilon=U_x\) is infinitely divisible. Then, there exists a \(V^\epsilon\) such that \(U^\epsilon = V^\epsilon*V^\epsilon\). Hence, if we can approximate the antiderivative of our kernel by a infinitely divisible distribution (multiplied by a constant if necessary) we are able to find candidates for interesting kernels.

Another powerful tool is the Fourier analysis. On the Fourier side the equation \( k^\epsilon = V^\epsilon*W^\epsilon \) becomes 
\begin{equation*}
  \mathcal{F}(W^\epsilon) = \mathcal{F}(V^\epsilon)\mathcal{F}({W}^\epsilon),
\end{equation*}
which can be explored. In particular for singular kernels we have representations of the Fourier transforms, see for instance~\cite{Stein1970}. Consequently, we can use this approach to obtain a wide range of interesting examples used in biology or physics.  

In the rest of the section we provide some fascinating examples for which the case of convolution structure in Theorem~\ref{theorem: emp_measure_l2_estimate} can be obtained. 

\subsection{Uniform bounded confidence model}

Let \(V(x)= i \indicator{\big[-\frac{R}{2},\frac{R}{2} \big]}(x)\) be a complex-valued function. Then
\begin{align*}
  \begin{cases}
   V*V \colon &\R \to \R  \\
   & x \to \begin{cases}
   0 \quad &\mathrm{if} \; x > |R|, \\
   -x-R \quad &\mathrm{if} \;  -R \le x \le 0, \\
   x-R \quad &\mathrm{if} \;  0 < x \le R .
   \end{cases}
   \end{cases}
\end{align*}
is a Lipschitz-continuous function with bounded support. Furthermore, we have 
\begin{equation*}
  \nabla(V*V) =- \indicator{[-R,0]} + \indicator{[0,R]} =: k_{\scriptscriptstyle{U}}
  \quad \text{a.e.}
\end{equation*}
Consequently, the uniform bounded confidence model, satisfies the assumption of Section~\ref{sec: relative_entropy_method} with the usual mollification approximation. Also, it is well known that the indicator function \(\indicator{\big[-\frac{R}{2},\frac{R}{2} \big]} \in H^s(\R)\) for all \(s<1/2\). We also have the convergence in probability by~\cite[Lemma 4.7, Theorem 6.1.]{Nikolaev2023}. Hence, we obtain the following proposition

\begin{proposition}
  Let \(k_{\scriptscriptstyle{U}}\) be given above, then the first marginal \((\rho_t^{N,1}, t \ge 0)\) of the law of the system \(\mathbf{X}^N\) converges to the law \((\rho_t, t \ge 0)\) of \(\mathbf{Y}^N\) in the \(L^1([0,T];L^1(\R))\)-norm.
\end{proposition} 

\subsection{Parabolic-Elliptic Keller--Segel System}

In this subsection we provide an approximation for the elliptic-parabolic Keller--Segel model~\cite{KellerSegel1970} in \(\R^d\). The underlying PDE is given by 
\begin{align*}
  \begin{cases}
  \partial_t \rho_t &= \frac{\sigma^2}{2} \Delta \rho_t - \nabla \cdot ( \chi \rho_t \nabla c_t)   \\
  - \Delta c_t &= \rho_t
  \end{cases}
\end{align*}
for \(\chi,\sigma > 0\). Decoupling the above system by setting \(c_t = \Phi* \rho_t\) with \(\Phi\) being the fundamental solution of the Laplace equation we can formally derive the above equation from the particle system~\eqref{eq: regularized_particle_system} with the interaction force kernel \(k =- \nabla \Phi\).
In particular, if \(d \ge 2 \) we have
\begin{align*}
  \Phi(x) = 
  \begin{cases}
  - \frac{1}{2\pi} \log(|x|), &\quad x \neq 0 ,\quad \mathrm{if} \; d = 2, \\
  \frac{1}{d(d-2)\lambda(B_1(0))} \frac{1}{|x|^{d-2}}, &\quad x \neq 0 , \quad \mathrm{if } \; d \ge 3 , 
  \end{cases}
\end{align*}
is the fundamental solution of the Laplace equation. 

In the following we present two approaches to mollify our kernel. For the first approach, let us define a mollification kernel \(J_{KS}\), which satisfies \(J_{KS} \ge 0\), \(\norm{J_{KS}}_{L^1(\R^d)}=1\) and \(\mathrm{supp}(J_{KS}) \subset B(0,1/2)\) and is infinitely differentiable. As always we set \(J_{KS}^\epsilon(x)= \frac{1}{\epsilon^d} J_{KS}\big(\frac{x}{\epsilon}\big)\). Then \(k^\epsilon =- \nabla (J^\epsilon_{KS}*\Phi*J^\epsilon_{KS})\) satisfies all properties of~\cite[Theorem 2.1]{HuangHiuLius2019}. Hence, the convergence in probability Assumption~\ref{ass: convergence_in_probability}, the law of large numbers Assumption~\ref{ass: law_of_large_numbers} is satisfied. 

Hence, under consideration of Remark~\ref{remark: alpha_restriction} we can obtain a relative entropy convergence results on the approximated \(d\)-dimensional attractive Keller--Segel system on the whole space \(\R^d\). We formulate the following proposition as combination of Lemma~\ref{lemma: relative_entropy_L_2_estimate} and Theorem~\ref{theorem: emp_measure_l2_estimate}. 

\begin{proposition} \label{prop: Keller_Segel2d_e_e}
  Let \(k^\epsilon = -\nabla (W^\epsilon*V^\epsilon)_x\) with \(W^\epsilon = J^\epsilon_{KS}*\Phi\), \(V^\epsilon = J^\epsilon_{KS}\). Let \(\rho^{N,\epsilon}\) be the solution of the Lioville equation~\eqref{eq: regularized_Liouville_equation} and \(\rho^{\epsilon}\) be the solution to regularized Keller--Segel equation, i.e. to the PDE~\eqref{eq: regularized_aggregation_diffusion_pde}. Then, there exists a \(\beta > 0 \) depending on the dimension \(d\) such that for \(\epsilon= \epsilon(N) = N^{-\beta}\) there exists a \(\lambda >0\) such that
  \begin{equation*}
     \sup\limits_{t \in [0,T]} {\mathcal H}_N( \rho^{N,\epsilon(N)}_t \;  | \; \rho^{\otimes N, \epsilon(N)}_t) \le C N^{-\lambda}.
  \end{equation*}
\end{proposition}

\begin{remark}
  By going through the proof of Theorem~\ref{theorem: emp_measure_l2_estimate} one can obtain a convergence rate and precise condition for \(\beta\). Furthermore, by inequality~\eqref{eq: CKP_inequality} we have proven convergence of the \(L^1(\R^d)\)-norm of the marginals  
  \begin{equation*}
    \lim\limits_{N\to \infty}  \norm{\rho_t^{N,2,\epsilon(N)}-\rho_t^{\epsilon(N)} \otimes \rho_t^{\epsilon(N)} }_{L^1(\R^d)} = 0 .
  \end{equation*}
  It is also well-known that under additional assumptions on the initial data \(\rho_0\)  and in the sub-critical regime \(\chi < 8\pi\) in the case \(d=2\) the density \(\rho_t^{\epsilon(N)}\) converges  in \(L^1(\R^d)\). Hence, we have shown that in the sub-critical case the density of the two marginal \(\rho_t^{N,2,\epsilon(N)}\) converges in \(L^1(\R^d)\) to the solution of the Keller--Segel equation.
\end{remark}

In the case \(d \ge 3\) we can obtain an even better approximation, which has a symmetric convolution structure given by \(k^\epsilon=- \nabla(V^\epsilon*V^\epsilon)\). 
Indeed, define the approximation of \(\Phi\) as \(\Phi^\epsilon:=  J^\epsilon_{KS}*\Phi*J^\epsilon_{KS}\). Then, for \(c_\alpha = \pi^{ -\alpha/2} \Gamma(\alpha/2)\) and 
\begin{equation} \label{eq: regularized_fundamental_solution_laplace_by_mollifier}
  V^\epsilon =\sqrt{\frac{c_2}{c_{d-2} d(d-2)\lambda(B_1(0))}} \mathcal{F}^{-1}(|\xi|^{-1} \mathcal{F}(J_{KS}^\epsilon)(\xi))
\end{equation}
we have 
\begin{equation}\label{eq: Keller_Segel_conv_split}
  \Phi^\epsilon = V^\epsilon*V^\epsilon.
\end{equation}
More precisely, for fix \(\epsilon > 0\) we have \(J_{KS}^\epsilon \in L^p(\R^d)\) for all \(p \ge 1 \). Hence the Fourier transform \(\mathcal{F}(J_{KS}^\epsilon)\) is well-defined and by the Hardy--Littlewood--Sobolev inequality~\cite[Chapter~5, Theorem~1]{Stein1970},~\cite[Corollary~5.10]{LiebElliottH2010A} \(\Phi^\epsilon \in L^2(\R)\) and the Fourier transform exists. Similar \(|\cdot|^{-1} \mathcal{F}(J_{KS}^\epsilon)(\cdot) \in L^2(\R^d)\). A simple calculation shows
\begin{align*}
  \norm{|\cdot|^{-1} \mathcal{F}^{\frac{1}{2}}(J_{KS}^\epsilon)(\cdot)}_{L^2(\R)}^2
  &= \int_{\R^d} |\xi|^{-2}  |\mathcal{F}(J_{KS}^\epsilon)(\xi)| \Id \xi  \\
  &\le \norm{\mathcal{F}(J_{KS}^\epsilon)}_{L^\infty(\R)} \int_{B_1(0)}  |\xi|^{-2} \Id  \xi + \int_{B_1(0)^{\mathrm{c}}}    |\mathcal{F}(J_{KS}^\epsilon)(\xi)| \Id \xi
  < \infty
\end{align*}
since \(d>2\) and \(\mathcal{F}(J_{KS}^\epsilon)\) is a Schwartz function. As a result, to verify~\eqref{eq: Keller_Segel_conv_split} we need to show 
\begin{equation*}
  \mathcal{F}(\Phi^\epsilon) = \mathcal{F}(V^\epsilon)^2,
\end{equation*}
where the right-hand is square integrable. Now, by~\cite[Corollary~5.10]{LiebElliottH2010A} we have
\begin{equation*}
  \mathcal{F}(V^\epsilon)^2(\xi) = \frac{c_2}{c_{d-2} d(d-2)\lambda(B_1(0))} |\xi|^{-2} \mathcal{F}(J_{KS}^\epsilon)(\xi) = \mathcal{F}(\Phi^\epsilon)(\xi),
\end{equation*}
where the left-hand side is in \(L^2(\R^d)\) by similar arguments as before. Therefore,~\eqref{eq: Keller_Segel_conv_split} is proven and we can find an appropriate approximation for the Keller--Segel interaction kernel. In particular, we can derive similar estimates to~\eqref{eq: l2_bound_V} with the help of Fourier analysis and the Hardy--Littlewood--Sobolev inequality~\cite[Chapter~ 5. Theorem~1]{Stein1970}. Clearly, this estimates will now depend on the dimension \(d\) and therefore the convergence rate parameters also depend on the dimension \(d\). 

\begin{proposition}\label{prop: Keller_Segel_e_e}
  Let \(d\ge 3\) and \(k^\epsilon = -\nabla (W^\epsilon*V^\epsilon)_x\) with \(W^\epsilon,V^\epsilon\) defined by the same expression~\eqref{eq: regularized_fundamental_solution_laplace_by_mollifier}. Then the conclusion of Proposition~\ref{prop: Keller_Segel2d_e_e} holds. Additionally we have the modulated energy estimate
  \begin{equation*}
    \sup\limits_{t \in [0,T]}  | {\mathcal K}_{N}  (\rho_t^{N,\epsilon}\vert \rho_t^{\otimes N, \epsilon})| \le C N^{-\lambda}
  \end{equation*}
  for some \(C>0, \lambda>0\). 
\end{proposition} 

Another approach to approximate the Coulomb kernel \(k\) is by utilizing the following approximation in dimension \(d \ge 3\), 
\begin{equation}\label{eq: regularized_fundamental_solution_laplace}
  \Phi^\epsilon(x) = \frac{1}{d(d-2)\lambda(B_1(0))} \int_{\R} \frac{h^\epsilon(y)}{|x-y|^{d-2}} \Id y ,
\end{equation}
where 
\begin{equation}\label{eq: def_weierstrass_kernel}
  h^\epsilon(y) = \frac{1}{(4\pi \epsilon)^{d/2}} \exp\bigg( -\frac{|y|^2}{4\epsilon}\bigg)
\end{equation}
is the Weierstrass kernel. 
Indeed, we note first that the square root \(\mathcal{F}^{\frac{1}{2}}(h^\epsilon)\) is well-defined since the Fourier transform of a Gaussian is still a Gaussian or in other words the normal distribution is infinitely divisible. More precisly, by~\cite[Theorem~5.2]{LiebElliottH2010A} we have
\begin{equation*}
  \mathcal{F}(h^\epsilon)(\xi) =  \exp\bigg( -4 \epsilon \pi^2 |\xi|^2 \bigg) .
\end{equation*} 

Hence, similar to the first approximation, we obtain Proposition~\ref{prop: Keller_Segel_e_e}. By using the Weierstrass kernel over an abstract mollification kernel we obtain explicit sharp convergence rates. 
For instance, using Plancherel theorem we obtain
\begin{align*}
  \norm{\nabla V^\epsilon}_{L^2(\R^d)}^2
  &= \frac{c_2}{c_{d-2} d(d-2)\lambda(B_1(0))} \norm{ \nabla \mathcal{F}^{-1}(|\xi|^{-1} \mathcal{F}^{\frac{1}{2}}(h^\epsilon)(\xi))(\cdot)}_{L^2(\R^d)}^2  \\
  &= \frac{c_2}{c_{d-2} d(d-2)\lambda(B_1(0))} \int_{\R^d} \sum\limits_{i=1}^d | \partial_{x_{i}} \mathcal{F}^{-1}(|\xi|^{-1} \mathcal{F}^{\frac{1}{2}}(h^\epsilon)(\xi))(x)|^2 \Id  x \\
  &= \frac{c_2}{c_{d-2} d(d-2)\lambda(B_1(0))} \int_{\R^d} \sum\limits_{i=1}^d  |\mathcal{F}^{-1}( 2\pi i \xi_{i} |\xi|^{-1} \mathcal{F}^{\frac{1}{2}}(h^\epsilon)(\xi))(x)|^2 \Id  x \\
  &=\frac{2\pi c_2}{c_{d-2} d(d-2)\lambda(B_1(0)) }  \sum\limits_{i=1}^d  \int_{\R^d}  |\xi_{i} \xi^{-1} \mathcal{F}^{\frac{1}{2}}(h^\epsilon)(\xi)|^2 \Id \xi \\
  &=\frac{2\pi c_2}{c_{d-2} d(d-2)\lambda(B_1(0)) }  \sum\limits_{i=1}^d  \int_{\R^d}  |\xi_{i} \xi|^{-1} \exp\bigg( -2 \epsilon \pi^2 |\xi|^2 \bigg)|^2 \Id \xi  \\
  &=\frac{2\pi c_2}{c_{d-2} d (d-2)\lambda(B_1(0)) \epsilon^{d/2}}  \sum\limits_{i=1}^d  \int_{\R^d}  \bigg|\xi_{i}\xi^{-1} \exp\bigg( -2 \pi^2 |\xi|^2 \bigg) \bigg|^2 \Id \xi  \\
  &\le \frac{2\pi  c_2}{c_{d-2}  (d-2)\lambda(B_1(0)) \epsilon^{d/2}}   \int_{\R^d} \exp\bigg( -4 \pi^2 |\xi|^2 \bigg)| \Id \xi  \\
  &= \frac{ c_2}{2^{d-1} \pi^{d/2-1} c_{d-2} (d-2)\lambda(B_1(0)) }  \epsilon^{-d/2}.
\end{align*}

\begin{remark}
  The above potential is attractive and therefore as far as we know regularization/approximation is necessary to obtain a solution of the underlying Lioville equation on \(\R^d\). Nevertheless, one can obtain tightness of the empirical measure in the super subcritical regime~\cite{FournierJourdain2017}. Our approach provides propagation of chaos of the intermediate system on the level of the relative entropy. Hence, it can be used as a tool to develop further results on the propagation of chaos for the Keller--Segel model without regularization.
\end{remark}

\subsection{Parabolic-Elliptic Keller--Segel System with Bessel potential}

Let us recall the parabolic-elliptic Keller--Segel model~\cite{KellerSegel1970} in \(\R^d\) given by 
\begin{align*}
  \begin{cases}
  \partial_t \rho_t &= \frac{\sigma^2}{2} \Delta \rho_t - \nabla \cdot ( \rho_t \nabla c_t)   \\
  c_t &=  \Delta c_t + \rho_t .
  \end{cases}
\end{align*}
Again solving the second equation by setting  
\begin{equation*}
  c_t = (I-\Delta)^{-1} \rho_t = G *\rho_t
\end{equation*}
with the \(L^1\) function \(G\) defined by 
\begin{equation}\label{eq: besse_potenial}
  G(x):= \mathcal{F}^{-1}[   (1+4\pi^2|\xi|^2)^{-1}  ](x) = \frac{1}{(4\pi)^{d/2} }
  \int\limits_0^\infty \exp\bigg( -t -\frac{|x|^2}{4t} \bigg) t^{-\frac{n}{2}} \Id t ,
\end{equation}
we can decouple the system and obtain an analogous result by using the following approximations of \(G\),
\begin{equation}\label{eq: bessel_potenial_approximation}
  G^\epsilon (x) = G*h^\epsilon(x),
\end{equation}
or 
\begin{equation} \label{eq: bessel_potenial_approximation2}
G^\epsilon (x) = G*J^\epsilon(x),
\end{equation}
where \(h^\epsilon\) is the Weierstrass kernel given by~\eqref{eq: def_weierstrass_kernel}. Setting 
\begin{equation*}
  V^\epsilon(x) =  \mathcal{F}^{-1}[   (1+4\pi^2|\xi|^2)^{-1/2} \mathcal{F}^{\frac{1}{2}}[h^\epsilon](\xi)  ](x),
\end{equation*}
it can be shown similar to the elliptic-parabolic Keller--Segel model that 
\begin{equation*}
  G^\epsilon = V^\epsilon*V^\epsilon.
\end{equation*}

Consequently, we obtain the analogous result. 

\begin{proposition}
  Let \(k^\epsilon = - \nabla G^\epsilon\) with \(G^\epsilon\) defined by~\eqref{eq: besse_potenial}~\eqref{eq: bessel_potenial_approximation} or~\eqref{eq: bessel_potenial_approximation2} and suppose the Assumptions~\ref{ass: convergence_in_probability} and~\ref{ass: law_of_large_numbers} hold. Moreover, for this \(k^\epsilon\) let \(\rho^{N,\epsilon}\) be the solution of the Lioville equation~\eqref{eq: regularized_Liouville_equation} and \(\rho^{\epsilon}\) be the solution to regularized Keller--Segel equation, i.e. to the PDE~\eqref{eq: regularized_aggregation_diffusion_pde}. Then, there exists a \(\beta > 0 \) depending on the dimension \(d\) such that for \(\epsilon= \epsilon(N) = N^{-\beta}\) there exists a \(\lambda >0\) such that
  \begin{equation*}
     \sup\limits_{t \in [0,T]} {\mathcal H}_N( \rho^{N,\epsilon(N)}_t \;  | \; \rho^{\otimes N, \epsilon(N)}_t) + \sup\limits_{t \in [0,T]}   | {\mathcal K}_{N}  (\rho_t^{N,\epsilon}\vert \rho_t^{\otimes N, \epsilon})| \le C N^{-\lambda}.
  \end{equation*}
\end{proposition} 

\begin{remark}
  By going through the proof of Theorem~\ref{theorem: emp_measure_l2_estimate} one can obtain a convergence rate and precise condition for \(\beta\). We also assumed the convergence probability, since we can not reference a concrete result. Nevertheless, we think that this assumption should be true for good enough initial data.
\end{remark}

\bibliography{quellen}
\bibliographystyle{amsalpha}

\end{document}